\documentclass[a4paper,legno]{amsproc}
\usepackage[english]{babel}
\usepackage{graphicx}
\usepackage[all]{xy}
\usepackage{amsthm}
\usepackage{amsfonts}
\usepackage{amssymb}
\usepackage{mathrsfs}
\usepackage{amsbsy}
\usepackage{stmaryrd}
\usepackage{amsmath}

\makeatletter
\DeclareFontFamily{OMX}{MnSymbolE}{}
\DeclareSymbolFont{MnLargeSymbols}{OMX}{MnSymbolE}{m}{n}
\SetSymbolFont{MnLargeSymbols}{bold}{OMX}{MnSymbolE}{b}{n}
\DeclareFontShape{OMX}{MnSymbolE}{m}{n}{
    <-6>  MnSymbolE5
   <6-7>  MnSymbolE6
   <7-8>  MnSymbolE7
   <8-9>  MnSymbolE8
   <9-10> MnSymbolE9
  <10-12> MnSymbolE10
  <12->   MnSymbolE12
}{}
\DeclareFontShape{OMX}{MnSymbolE}{b}{n}{
    <-6>  MnSymbolE-Bold5
   <6-7>  MnSymbolE-Bold6
   <7-8>  MnSymbolE-Bold7
   <8-9>  MnSymbolE-Bold8
   <9-10> MnSymbolE-Bold9
  <10-12> MnSymbolE-Bold10
  <12->   MnSymbolE-Bold12
}{}

\let\llangle\@undefined
\let\rrangle\@undefined
\DeclareMathDelimiter{\llangle}{\mathopen}%
                     {MnLargeSymbols}{'164}{MnLargeSymbols}{'164}
\DeclareMathDelimiter{\rrangle}{\mathclose}%
                     {MnLargeSymbols}{'171}{MnLargeSymbols}{'171}
\makeatother

\usepackage[usenames,dvipsnames]{color}
\usepackage{color}

\RequirePackage{color}
\definecolor{bwgreen}{rgb}{0.01,0.10,0.25}
\definecolor{bwmagenta}{rgb}{0.25,0.0,0.1}
\definecolor{bwblue}{rgb}{0.317,0.161,1}

\usepackage[colorlinks=true]{hyperref}
\hypersetup{
	bookmarksnumbered=true,
	linkcolor=bwmagenta,
	citecolor=bwgreen,
}

\usepackage[colorlinks=true]{hyperref}
\hypersetup{
	bookmarksnumbered=true,
	linkcolor=bwmagenta,
	citecolor=bwgreen,
}

\usepackage{calligra}
\makeatletter
\def\@splitop#1#2\@nil{$\mathscr{#1}\!\!$\calligra#2\,\,}
\newcommand*\DeclareCursiveOperator[2]{%
  \newcommand#1{\mathop{\mbox{\@splitop#2\@nil}}\nolimits}}
\makeatother
\DeclareCursiveOperator{\TAY}{Hess}
\DeclareCursiveOperator{\HOM}{Hom}

\newtheorem{theo}{Theorem}[section]
\newtheorem{cor}[theo]{Corollary}

\newtheorem{lemma}[theo]{Lemma}
\newtheorem{remark}[theo]{Remark}

\newtheorem{proposition}[theo]{Proposition}

\newtheorem*{theoremA}{Theorem A}
\newtheorem*{theoremB}{Theorem B}
\newtheorem*{theoremC}{Theorem C}
\newtheorem*{theoremD}{Theorem D}
\newtheorem*{theoremE}{Theorem E}

\newtheorem*{remm}{Remark}

\newcommand{\quo}[1]{ \mathbf{Z}/p^{n}\mathbf{Z}  }

\newcommand{\iw}{\Lambda}

\newcommand{\gaun}{  \mathfrak{G}    }

\newcommand{\fre}[1]{\stackrel{#1}{\rightarrow}}

\newcommand{\defo}{\mathbb{T}}

\usepackage[OT2,T1]{fontenc}
\DeclareSymbolFont{cyrletters}{OT2}{wncyr}{m}{n}
\DeclareMathSymbol{\sha}{\mathalpha}{cyrletters}{"58}
\newcommand{\inlim}{\mathop{\varprojlim}\limits}

\newcommand{\dia}[1]{\left<{}#1\right>}

\newcommand{\hid}{\iw}
\newcommand{\lri}[1]{\left(#1\right)}

\newcommand{\cts}{C_{\mathrm{cont}}}
\newcommand{\ctsb}{C_{\mathrm{cont}}^{\bullet}}
\newcommand{\sco}{\widetilde{C}_{f}}
\newcommand{\scob}{\widetilde{C}_{f}^{\bullet}}

\newcommand{\derco}{\widetilde{\mathbf{R}\Gamma}_{f}}
\newcommand{\exsel}{\widetilde{H}_f}

\newcommand{\Hom}[1]{\mathrm{Hom}_{#1}}

\newcommand{\xari}{\mathcal{X}^{\mathrm{arith}}}

\newcommand{\derot}[1]{\otimes_{#1}^{\mathbf{L}}}

\newcommand{\neko}{Nekov\'a\v{r}}
\newcommand{\dercts}{\mathbf{R}\Gamma_{\mathrm{cont}}}

\newcommand{\hwp}[2]{\left\llangle #1,#2 \right\rrangle_{p}}
\newcommand{\bki}{\mathfrak{Z}^{\mathrm{BK}}}
\newcommand{\bock}{\widetilde{\beta}_{p}}

\newcommand{\lfre}[1]{\stackrel{#1}{\longrightarrow}}
\newcommand{\T}{\mathbb{T}}
\newcommand{\R}{\overline{R}}
\newcommand{\Ic}{\iw_{\mathrm{cyc}}}
\newcommand{\M}{\texttt{M}}
\newcommand{\Mm}{\overline{\texttt{M}}}
\newcommand{\Z}{\mathbf{Z}}
\newcommand{\Q}{\mathbf{Q}}

\newcommand{\N}{\mathbf{N}}

\author{Rodolfo Venerucci}

\begin{document}

\renewcommand{\addresses}{{
  \bigskip
  \footnotesize

  R.~Venerucci, \textrm{Universität Duisburg--Essen,
Fakultät für Mathematik,
Mathematikcarrée,
Thea--Leymann--Stra\ss{}e 9, 45127 Essen.
    }\par\nopagebreak
  \textit{E-mail address}: \texttt{rodolfo.venerucci@uni-due.de}

}}

\ 

\title{Exceptional zero formulae and a conjecture of Perrin-Riou}
\maketitle

\begin{abstract} Let $A/\Q$ be an elliptic curve with split multiplicative reduction at a prime $p$. We prove  (an analogue of) a 
conjecture of Perrin-Riou, relating  $p$-adic  Beilinson--Kato elements to Heegner points in $A(\Q)$,
and a large part of the rank-one case of the Mazur--Tate--Teitelbaum exceptional zero conjecture
for the cyclotomic  $p$-adic $L$-function of $A$.
More generally, let $f$ be the weight-two newform associated with $A$,
let $f_{\infty}$ be the Hida family of $f$, and let $L_{p}(f_{\infty},k,s)$ be the Mazur--Kitagawa two-variable $p$-adic $L$-function attached to $f_{\infty}$. We prove a $p$-adic Gross--Zagier formula, expressing the quadratic term of the Taylor expansion 
of $L_{p}(f_{\infty},k,s)$ at $(k,s)=(2,1)$ as 
a non-zero rational 
multiple of the 
extended height-weight of a Heegner point in $A(\Q)$. 
\end{abstract}

\setcounter{tocdepth}{1}

\section{Introduction}

Let $A$ be an elliptic curve over $\Q$ of  conductor  $Np$, with $p>3$ a prime  of \emph{split} multiplicative reduction. 
Fix algebraic closures $\overline{\Q}$ and $\overline{\Q}_{p}$ of $\Q$
and $\Q_{p}$ respectively, and an embedding $i_{p} : \overline{\Q}\hookrightarrow{}\overline{\Q}_{p}$.
Assume throughout this paper that the $p$-torsion subgroup $A_{p}$ of $A(\overline{\Q})$ is an irreducible $\mathbf{F}_{p}[G_{\Q}]$-module, where $G_{\Q}:=\mathrm{Gal}(\overline{\Q}/\Q)$. 

For every $n\in{}\N$, write $\Q_{n}/\Q$ for the cyclic sub-extension of $\Q(\mu_{p^{n+1}})/\Q$ of degree $p^{n}$ and let
$\Q_{\infty}=\bigcup_{n\in{}\N}\Q_{n}$ be the cyclotomic $\Z_{p}$-extension of $\Q$.
Denote by $G_{\infty}:=\mathrm{Gal}(\Q_{\infty}/\Q)$ 
the Galois group of $\Q_{\infty}$ over $\Q$ and by  $\Ic:=\Z_{p}\llbracket{}G_{\infty}\rrbracket$ the cyclotomic Iwasawa algebra. 
Associated with  $A/\Q$ (and $i_{p}$) there is a  $p$-adic $L$-function $$L_{p}(A/\Q)\in{}\Ic,$$ interpolating 
the critical values $L(A/\Q, \chi,1)$ of the Hasse--Weil $L$-function  of $A/\Q$ twisted by finite order characters $\chi : G_{\infty}\fre{}\overline{\Q}_{p}^{\ast}$.
Thanks to the results  of Kato and Coleman--Perrin-Riou, it is known that $L_{p}(A/\Q)$ arises from an Euler system for the $p$-adic Tate module of $A/\Q$.
More precisely, denote by $\Q_{p,\infty}=\bigcup_{n\in{}\N}\Q_{p,n}$ the cyclotomic $\Z_{p}$-extension of $\Q_{p}$ (with notations similar to those introduced above),
and by $T_{p}(A)$ the $p$-adic Tate module of $A$.
For $K=\Q$ or $\Q_{p}$, let
$H^{1}_{\mathrm{Iw}}(K_{\infty},T_{p}(A))$ be the inverse limit of the cohomology groups $H^{1}(K_{n},T_{p}(A))$.
The work of  Coleman--Perrin-Riou yields   a   \emph{big dual exponential}
\[
          \mathcal{L}_{A} : H^{1}_{\mathrm{Iw}}(\Q_{p,\infty},T_{p}(A))\lfre{}\Ic.
\]
It is a morphism of $\Ic$-modules, which interpolates the Bloch--Kato dual exponential maps 
attached to the  twists  of $T_{p}(A)$ by finite order characters $\chi$ of $G_\infty$
(see Section $\ref{colprsec}$ for the precise definition). 
In \cite{kateul} Kato constructs 
a cyclotomic  Euler system for $T_{p}(A)$, related to $L_{p}(A/\Q)$ via $\mathcal{L}_{A}$.
In particular he constructs an element
$\zeta^{\mathrm{BK}}_{\infty}=\big(\zeta^{\mathrm{BK}}_{n}\big)_{n\in{}\N}\in{}H^{1}_{\mathrm{Iw}}(\Q_{\infty},T_{p}(A))$ 
such that
\begin{equation}\label{eq:kato reciprocity +}
               \mathcal{L}_{A}\big(\mathrm{res}_{p}\big(\zeta_{\infty}^{\mathrm{BK}}\big)\big)=L_{p}(A/\Q).
\end{equation}
Kato's Euler system is built out of Steinberg symbols of certain Siegel modular units, which also appeared in the work 
of Beilinson. The classes $\zeta^{\mathrm{BK}}_{n}$ are then called \emph{$p$-adic Beilinson--Kato classes}.

\subsection{\textbf{A conjecture of Perrin-Riou}} 
Set $V_{p}(A):=T_{p}(A)\otimes_{\Z_{p}}\Q_{p}$ and denote by $\zeta^{\mathrm{BK}}$
the natural image of the class $\zeta^{\mathrm{BK}}_{0}\in{}H^{1}(\Q,T_{p}(A))$
in $H^{1}(\Q,V_{p}(A))$. We call $\zeta^{\mathrm{BK}}$
\emph{the $p$-adic Beilinson--Kato class} attached to $A$.
According to \emph{Kato's  reciprocity law} \cite{kateul}
\begin{equation}\label{eq:kato reciprocity}
            \exp_{A}^{\ast}\big(\mathrm{res}_{p}\big(\zeta^{\mathrm{BK}}\big)\big)=\lri{1-\frac{1}{p}}\frac{L(A/\Q,1)}{\Omega_{A}^{+}}\in{}\Q,
\end{equation}
where $\Omega_{A}^{+}\in{}\mathbf{R}^{\ast}$ is the real N\'eron period of $A$ and $\exp_{A}^{\ast} : H^{1}(\Q_{p},V_{p}(A))\fre{}
\mathrm{Fil}^{0}D_{\mathrm{dR}}(V_{p}(A))\cong{}\Q_{p}$ is the Bloch--Kato dual exponential map 
(see Section $\ref{expbk}$ for the last isomorphism).
In particular this implies that the complex Hasse--Weil $L$-function $L(A/\Q,s)$ vanishes at $s=1$
precisely if $\zeta^{\mathrm{BK}}$ is a Selmer class, i.e. if it belongs  to the  Bloch--Kato Selmer group  $H^{1}_{f}(\Q,V_{p}(A))\subset{}H^{1}(\Q,V_{p}(A))$
of  $V_{p}(A)$. 

When $L(A/\Q,1)=0$, it is natural to ask whether $\zeta^{\mathrm{BK}}$ is still related to the special values of $L(A/\Q,s)$.
Perrin-Riou addresses this question in \cite{PRconj} for elliptic curves with \emph{good} reduction at $p$. In that setting, she conjectures that the logarithm of the $p$-adic Beilinson--Kato 
class equals the square of the logarithm of a Heegner point on the elliptic curve, up to a
non-zero rational factor. In particular, she predicts 
that the Beilinson--Kato class is non-zero precisely if the Hasse--Weil $L$-function has a simple zero at $s=1$.
The first  aim of this paper is to prove the analogue of Perrin-Riou's conjecture in our 
multiplicative setting.

Since $A/\Q_{p}$ is split multiplicative, Tate's theory provides a $G_{\Q_{p}}$-equivariant  $p$-adic uniformisation
\begin{equation}\label{eq:Tate iso}
     \Phi_{\mathrm{Tate}} : \overline{\Q}_{p}^{\ast}/q_{A}^{\Z}\cong{}A(\overline{\Q}_{p}),
\end{equation}
where $q_{A}\in{}p\Z_{p}$ is the \emph{Tate period} of $A/\Q_{p}$. 
Denote by
$\log_{q_{A}} : \Q_{p}^{\ast}/q_{A}^{\Z}\fre{}\Q_{p}$ the branch of the $p$-adic logarithm which vanishes at $q_{A}$ and by  $$\log_{A}=\log_{q_{A}}\circ{}\Phi_{\mathrm{Tate}}^{-1}: A(\Q_{p})\lfre{}\Q_{p}$$
the formal group logarithm on $A/\Q_{p}$.
It induces an isomorphism $\log_{A} : A(\Q_{p})\widehat{\otimes}\Q_{p}\cong{}\Q_{p}$ on $p$-adic completions.

\begin{theoremA} Assume that $L(A/\Q,1)=0$, i.e. that $\zeta^{\mathrm{BK}}$ is a Selmer class. 

$1.$ There exist a non-zero rational number $\ell_{1}\in{}\Q^{\ast}$ and a rational point $\mathbf{P}\in{}A(\Q)\otimes\Q$ such that 
\[
                \log_{A}\big(\mathrm{res}_{p}\big(\zeta^{\mathrm{BK}}\big)\big)=\ell_{1}\cdot{}\log_{A}^{2}(\mathbf{P}).
\]

$2.$ $\mathbf{P}$ is non-zero if and only if $L(A/\Q,s)$ has a simple zero at $s=1$.\\
In particular: $\mathrm{res}_{p}(\zeta^{\mathrm{BK}})\not=0$ if and only if $L(A/\Q,s)$ has a simple zero at $s=1$.
\end{theoremA}

The point $\mathbf{P}\in{}A(\Q)\otimes\Q$ which appears in the statement  is a Heegner point, coming from a 
certain Shimura curve parametrisation of $A$ (see Section $\ref{BD formula}$).
Theorem A then compares two Euler systems of a different nature: Kato's  Euler system, belonging  to the \emph{cyclotomic} Iwasawa theory of $A$,
and the Euler system of Heegner points, which pertains to the \emph{anticyclotomic} Iwasawa 
theory of $A$ (and a suitable quadratic imaginary field).

The proof of Theorem A relies on Hida's theory of $p$-adic families of modular forms.  
Together with the work of Kato and Coleman--Perrin-Riou mentioned above, the exceptional zero formula 
proved by Bertolini and Darmon in \cite{B-D}, and \neko's theory of Selmer complexes \cite{Ne} are the key ingredients 
in our proof.

\begin{remm}\emph{$1.$ Assume that $L(A/\Q,s)$ has a simple zero at $s=1$. By the theorem of Gross--Zagier--Kolyvagin, 
$A(\Q)$ has rank one and $A(\Q)\otimes\Q_{p}=H^{1}_{f}(\Q,V_{p}(A))$ is generated by $\mathbf{P}$. 
By Theorem A, $\zeta^{\mathrm{BK}}$ is equal to $\log_{A}(\mathbf{P})\cdot{}\mathbf{P}$, up to a non-zero 
\emph{rational} factor.
According to \cite[Corollaire 2]{Ber-1}, $\log_{A}(\mathbf{P})\in{}\Q_{p}^{\ast}$ is transcendental over $\Q$, so that 
$\zeta^{\text{BK}}\not\in{}A(\Q)\otimes{}\overline{\Q}$.
In particular, $\zeta^{\mathrm{BK}}$ does not come from a rational point in $A(\Q)\otimes\Q$.
}

\emph{$2.$   Bertolini and Darmon have recently announced \cite{B-Dkato} a proof of Perrin-Riou's conjecture for elliptic curves with good ordinary 
reduction at $p$. Their approach, based on the $p$-adic Beilinson formula proved in \emph{loc. cit.}
and the $p$-adic Gross--Zagier formula proved in \cite{Be-Da-Pr}, is markedly different from ours. 
}
\end{remm}

Combining Theorem A, the results of Kato and Kolyvagin's method, we deduce the following result. 

\begin{theoremB} $\zeta^{\mathrm{BK}}$ is non-zero if and only if $\mathrm{ord}_{s=1}L(A/\Q,s)\leq{}1$.
\end{theoremB}

\subsection{\textbf{$p$-adic Gross--Zagier formulae}}\label{intro 2} Let $\chi_{\mathrm{cyc}} : G_{\infty}\cong{}1+p\Z_{p}$
denote the $p$-adic cyclotomic character.  For every $s\in{}\Z_{p}$, set $L_{p}(A/\Q,s):=\chi_{\mathrm{cyc}}^{s-1}\big(L_{p}(A/\Q)\big)$.
Then $L_{p}(A/\Q,s)$ is a $p$-adic analytic function on $\Z_{p}$.  
Since $A$ has split multiplicative reduction at $p$,
the phenomenon  of exceptional  zeros discovered in \cite{M-T-T} implies that 
$L_{p}(A/\Q,1)=0$  independently of whether 
$L(A/\Q,s)$ vanishes or not at $s=1$. The \emph{exceptional zero conjecture} formulated in \emph{loc. cit.}
states that  $\mathrm{ord}_{s=1}L_{p}(A/\Q,s)=\mathrm{ord}_{s=1}L(A/\Q,s)+1$, and that the leading term in the Taylor expansion 
of $L_{p}(A/\Q,s)$ at $s=1$ equals, up to a non-zero  rational factor, the determinant of the lattice $A^{\dag}(\Q)/\mathrm{torsion}$, computed with
respect to  the extended cyclotomic $p$-adic height pairing. Here $A^{\dag}(\Q)$ is the extended Mordell--Weil group, whose elements are pairs 
$(P,y_{P})\in{}A(\Q)\times{}\Q_{p}^{\ast}$ such that $\Phi_{\mathrm{Tate}}(y_{P})=P$; it is an extension 
of $A(\Q)$ by the $\Z$-module generated by the Tate period $q_{A}=(0,q_{A})\in{}A^{\dag}(\Q)$.
When $L(A/\Q,1)\not=0$ the conjecture predicts 
\[
             \frac{d}{ds}L_{p}(A/\Q,s)_{s=1}=\mathscr{L}_{p}(A)\frac{L(A/\Q,1)}{\Omega_{A}^{+}},
\]
where $\mathscr{L}_{p}(A)=\log_{p}(q_{A})/\mathrm{ord}_{p}(q_{A})$ is the $\mathscr{L}$-invariant of $A/\Q_{p}$.
This formula was proved by Greenberg and Stevens in \cite{G-S}.
(We give a slightly different proof of it in Theorem $\ref{main GS}$ below.)

Our second aim in this paper is to prove (a large part of) the above exceptional zero conjecture when $\mathrm{ord}_{s=1}L(A/\Q,s)=1$
and, 
more generally, a two-variable $p$-adic Gross--Zagier formula for the Mazur--Kitagawa $p$-adic $L$-function 
of the Hida family attached to $A/\Q$. 
Let $f\in{}S_{2}(\Gamma_{0}(Np),\Z)$ be the weight-two newform associated with  $A/\Q$ by the modularity theorem, and let 
$f_{\infty}=\sum_{n=1}^{\infty}a_{n}(k)\cdot{}q^{n}\in{}\mathscr{A}_{U}\llbracket{}q\rrbracket$ be the Hida family passing through 
$f$. Here  $U\subset{}\Z_{p}$ is a $p$-adic disc centred at $2$, 
and $\mathscr{A}_{U}\subset{}\Q_{p}\llbracket{}k-2\rrbracket$
is the subring of power series in the variable $k-2$ which converge for  $k\in{}U$.
For every $k\in{}U\cap{}\Z^{\geq{}2}$, the $q$-expansion $f_{k}:=\sum_{n=1}^{\infty}a_{n}(k)\cdot{}q^{n}\in{}S_{k}(\Gamma_{1}(Np),\Z_{p})$ 
is an $N$-new $p$-ordinary 
Hecke eigenform of weight $k$, and $f_{2}=f$ (cf. Section $\ref{mellinhida}$). Thanks to the work of Mazur--Kitagawa \cite{Kit} and 
Greenberg--Stevens \cite{G-S}, the $p$-adic $L$-functions of the forms $f_{k}$, for $k\in{}U\cap{}\Z^{\geq{2}}$,
can be packaged into a single two-variable $p$-adic $L$-function $L_{p}(f_{\infty},k,s)\in{}\mathscr{A},$
where $\mathscr{A}\subset{}\Q_{p}\llbracket{}k-2,s-1\rrbracket$ is the ring of formal power series converging for 
every $(k,s)\in{}U\times{}\Z_{p}$ (cf. Section \ref{mellinhida}).
In particular one has $L_{p}(f_{\infty},2,s)=L_{p}(A/\Q,s)$ and the exceptional zero phenomenon implies  
that $L_{p}(f_{\infty},k,s)\in{}\mathscr{J}$, where $\mathscr{J}\subset{}\mathscr{A}$ is the ideal of functions vanishing at $(k,s)=(2,1)$.

Let $\exsel^{1}(\Q,V_{p}(A))$ be \emph{\neko's extended Selmer group}. It is a $\Q_{p}$-module, equipped with a 
natural inclusion 
$A^{\dag}(\Q)\otimes\Q_{p}\hookrightarrow{}\exsel^{1}(\Q,V_{p}(A))$, which is an 
isomorphism precisely when the  $p$-primary part of the  Tate--Shafarevich group of $A/\Q$
is finite.  
In general $\exsel^{1}(\Q,V_{p}(A))$ is canonically isomorphic to the 
direct sum of the Bloch--Kato Selmer group $H^{1}_{f}(\Q,V_{p}(A))$ and the $1$-dimensional vector space $\Q_{p}\cdot{}q_{A}$
generated by the Tate period of $A/\Q_{p}$ (see Section $\ref{Extended Selmer Section}$). 
Using \neko's results and ideas (especially \cite[Section 11]{Ne}), we introduce in Section $\ref{pairing section}$
a canonical $\Q_{p}$-bilinear form 
\[
         \hwp{-}{-}  : \exsel^{1}(\Q,V_{p}(A))\otimes_{\Q_{p}}\exsel^{1}(\Q,V_{p}(A))\lfre{}\mathscr{J}/\mathscr{J}^{2},
\]
called the \emph{cyclotomic height-weight pairing}. One can write 
\[
          \hwp{-}{-}=\dia{-,-}_{p}^{\mathrm{cyc}}\cdot{}\{s-1\}+\dia{-,-}^{\mathrm{wt}}_{p}\cdot{}\{k-2\},
\]
where $\dia{-,-}_{p}^{\mathrm{cyc}}$ and $\dia{-,-}_{p}^{\mathrm{wt}}$
are canonical $\Q_{p}$-valued pairings on $\exsel^{1}(\Q,V_{p}(A))$
and  $\{\cdot{}\} : \mathscr{J}\twoheadrightarrow{}\mathscr{J}/\mathscr{J}^{2}$ denotes  the  projection. 
It turns out that the restriction
\[
         \dia{-,-}^{\mathrm{cyc}}_{p}  : H^{1}_{f}(\Q,V_{p}(A))\otimes_{\Q_{p}}H^{1}_{f}(\Q,V_{p}(A))\lfre{}\Q_{p}
\]
of $\dia{-,-}^{\mathrm{cyc}}_{p}$ to the Bloch--Kato Selmer group is the cyclotomic
canonical $p$-adic height pairing, as defined, e.g. in \cite[Section 7]{Nekh}
(see Section \ref{basic properties} for more details). On the other hand, the \emph{weight pairing} 
$\dia{-,-}^{\mathrm{wt}}_{p}$ is intrinsically associated with  Hida's $p$-ordinary deformation 
of $T_{p}(A)$ (cf. Section $\ref{hida rep}$).
For every Selmer class $x\in{}H^{1}_{f}(\Q,V_{p}(A))$, def{}ine its \emph{extended $p$-adic height-weight}
\begin{equation}\label{eq:hwp definition}
           \widetilde{h}_{p}(x):=\det\begin{pmatrix}\hwp{q_{A}}{q_{A}} && \hwp{q_{A}}{x} \\ && \\ \hwp{x}{q_{A}} && \hwp{x}{x}\end{pmatrix}\in
           \mathscr{J}^{2}/\mathscr{J}^{3}.
\end{equation}

Let $\mathrm{sign}(A/\Q)\in\{\pm1\}$ be the sign 
in the functional equation of $L(A/\Q,s)$, and consider the condition
\[
   \mathbf{(Loc)}\ \  L(A/\Q,1)=0 \text{\ and the restriction map\ } \mathrm{res}_{p} : H^{1}_{f}(\Q,V_{p}(A))\fre{}A(\Q_{p})\widehat{\otimes}\Q_{p}
   \text{ is non-zero}. 
\]
The work of Gross--Zagier--Kolyvagin guarantees that this condition is satisfied when $A(\Q)$ is infinite
and (in particular) when $L(A/\Q,s)$
has a simple zero at $s=1$.
We can finally state the two-variable $p$-adic Gross--Zagier formula mentioned above.

\begin{theoremC} Assume that $\mathrm{sign}(A/\Q)=-1$ and that $\mathbf{(Loc)}$ holds true.
Let $\mathbf{P}\in{}A(\Q)\otimes\Q$ be as in Theorem A. 
Then $L_{p}(f_{\infty},k,s)\in{}\mathscr{J}^{2}$ and
there exists a non-zero rational number $\ell_{2}\in{}\Q^{\ast}$ such that  
\[
           L_{p}(f_{\infty},k,s)\ \ \mathrm{mod}\ \mathscr{J}^{3}=\ell_{2}\cdot{}\widetilde{h}_{p}(\mathbf{P}).
\]
Moreover, $L_{p}(f_{\infty},k,s)\in{}\mathscr{J}^{3}$ if and only if $\mathbf{P}=0$
(i.e. $L(A/\Q,s)$ vanishes to order greater than one at $s=1$).
\end{theoremC}

\noindent 1.2.1. \emph{Application to the exceptional zero conjecture.} 
Recalling that $\log_{p}(q_{A})\not=0$ by \cite{M-man}, define the \emph{Schneider height} 
\[
                \dia{-,-}_{p}^{\mathrm{Sch}} : H^{1}_{f}(\Q,V_{p}(A))\otimes_{\Q_{p}}H^{1}_{f}(\Q,V_{p}(A))\lfre{}\Q_{p}
\]
as the symmetric, $\Q_{p}$-bilinear form which for $x,y\in{}H^{1}_{f}(\Q,V_{p}(A))$ is given by the formula 
\[
         \dia{x,y}^{\mathrm{Sch}}_{p}:=\dia{x,y}^{\mathrm{cyc}}_{p}-\frac{\log_{A}\big(\mathrm{res}_{p}(x)\big)\cdot{}\log_{A}\big(\mathrm{res}_{p}(y)\big)}
         {\log_{p}(q_{A})}.
\]
The terminology is justif{}ied by the fact that 
$\dia{-,-}_{p}^{\mathrm{Sch}}$ is the norm-adapted height constructed in \cite{Sch-1}
(cf. Section 7.14 of \cite{Ne} and Chapter II, \S6 of \cite{M-T-T}).
As a consequence of Theorem C and the properties of $\hwp{-}{-}$, one deduces the following
$p$-adic Gross--Zagier formula for $L_{p}(A/\Q,s)$, predicted by \emph{Conjecture BSD(p)-exceptional case} 
in \cite[Chapter II, \S10]{M-T-T}.

\begin{theoremD} Assume that $\mathbf{(Loc)}$ holds true and let $\mathbf{P}\in{}A(\Q)\otimes\Q$ be as in Theorem A.
Then $L_{p}(A/\Q,s)$ vanishes to order at least $2$ at $s=1$, and there exists a non-zero rational number $\ell_{3}\in{}\Q^{\ast}$ such that 
\[
             \frac{d^{2}}{ds^{2}}L_{p}(A/\Q,s)_{s=1}=\ell_{3}\cdot{}\mathscr{L}_{p}(A)\cdot\dia{\mathbf{P},\mathbf{P}}_{p}^{\mathrm{Sch}}.
\]
\end{theoremD}

The preceding result enriches our repertoire of $p$-adic Gross--Zagier formulae for cyclotomic and anticyclotomic 
$p$-adic $L$-functions of elliptic curves, which already includes  the main results of \cite{PRpbsd}, \cite{B-Dcd} and \cite{KobGZ}.

As $\mathscr{L}_{p}(A)\not=0$, Theorem D implies that  
$\mathrm{ord}_{s=1}L_{p}(A/\Q,s)=2$ precisely if $\mathrm{ord}_{s=1}L(A/\Q,s)=1$ \emph{and} $\dia{-,-}_{p}^{\mathrm{Sch}}$
is non-zero.
On the other hand, 
it is not known that the Schneider height  is non-zero when 
$L(A/\Q,s)$ has a simple zero at $s=1$.\vspace{1mm}

\noindent1.2.2. \emph{The derivative of the improved $p$-adic $L$-function.}\label{improvder} As explained in \cite{G-S}, 
the restriction of $L_{p}(f_{\infty},k,s)$ to the vertical line $s=1$ admits a factorisation
$L_{p}(f_{\infty},k,1)=\lri{1-a_{p}(k)^{-1}}\cdot{}L_{p}^{\ast}(f_{\infty},k)$
in $\mathscr{A}_{U}$.
The results of  \cite{G-S} and \cite{M-man}
imply that the function $1-a_{p}(k)^{-1}$ has a simple zero 
at $k=2$. 
The following 
$p$-adic Gross--Zagier formula for  the  \emph{improved $p$-adic $L$-function} 
$L_{p}^{\ast}(f_{\infty},k)$ is again a consequence of 
Theorem C and the properties of the height-weight pairing.

\begin{theoremE} Assume that hypothesis $\mathbf{(Loc)}$ holds and that $\mathrm{sign}(A/\Q)=-1$.
Let $\mathbf{P}\in{}A(\Q)\otimes\Q$ be as in Theorem A.
Then $L_{p}^{\ast}(f_{\infty},2)=0$ and there exists a non-zero rational number $\ell_{4}\in{}\Q^{\ast}$ such that 
\[
                          -\ell_{4}\cdot{}\dia{\mathbf{P},\mathbf{P}}^{\mathrm{cyc}}_{p}=
                          \frac{d}{dk}L_{p}^{\ast}(f_{\infty},k)_{k=2}=2\ell_{4}\cdot{}\dia{\mathbf{P},\mathbf{P}}^{\mathrm{wt}}_{p}.
\]
\end{theoremE}

\subsection{Outline of the proofs}\label{outline of the proof} We  brief{}ly sketch  the strategy of the  proofs of Theorems A and C,
assuming for simplicity that $L(A/\Q,s)$ has a simple zero at $s=1$.

Denote by $L_{p}^{\mathrm{cc}}(f_{\infty},k):=L_{p}(f_{\infty},k,k/2)\in{}\mathscr{A}_{U}$  the restriction of $L_{p}(f_{\infty},k,s)$
to the   \emph{central critical line} $s=k/2$.\\
According to the   exceptional zero formula proved by Bertolini--Darmon  in \cite{B-D},
$L_{p}^{\mathrm{cc}}(f_{\infty},k)$ has order of vanishing $2$ at $k=2$ and \vspace{-1mm}
\begin{equation}\label{eq:main BD introduction}
          \frac{d^{2}}{dk^{2}}L_{p}^{\mathrm{cc}}(f_{\infty},k)_{k=2}=\ell\cdot{}\log_{A}^{2}(\mathbf{P}),
\end{equation}
where $\ell\in{}\Q^{\ast}$ and $\mathbf{P}\in{}A(\Q)\otimes\Q$ is a  Heegner point. 
(See Section $\ref{BD formula}$ for more details.)

On the algebraic side, write $\widetilde{h}^{\mathrm{cc}}_{p} : H^{1}_{f}(\Q,V_{p}(A))\fre{}\Q_{p}$
for the composition of the extended height-weight $\widetilde{h}_{p}$ with the morphism 
$\mathscr{J}^{2}/\mathscr{J}^{3}\twoheadrightarrow{}\Q_{p}$
which on the class of $\alpha(k,s)\in{}\mathscr{J}^{2}$ takes the value $\frac{d^{2}}{dk^{2}}\alpha(k,k/2)_{k=2}$.
The properties satisfied by the height-weight pairing (cf. Theorem $\ref{mainreg}$) yield
\begin{equation}\label{eq:height s=k/2}
          \widetilde{h}_{p}^{\mathrm{cc}}(x)=\frac{1}{2}\log_{A}^{2}\big(\mathrm{res}_{p}(x)\big),
\end{equation}
for every Selmer class $x$.
Equation $(\ref{eq:main BD introduction})$ can then be rephrased as the  $p$-adic Gross--Zagier formula
\begin{equation}\label{eq:existing GZ}
                \frac{d^{2}}{dk^{2}}L_{p}^{\mathrm{cc}}(f_{\infty},k)_{k=2}=2\ell\cdot{}\widetilde{h}_{p}^{\mathrm{cc}}(\mathbf{P}).
\end{equation}
This shows that the formula displayed in Theorem C  holds true, once one 
restricts  both $L_{p}(f_{\infty},k,s)$ and $\widetilde{h}_{p}(\mathbf{P})$
to the central critical line $s=k/2$. 
Instead of  trying to extend $(\ref{eq:existing GZ})$ to the $(k,s)$-plane directly, we first prove an  analogue of  Theorem C,  in which   the 
Heegner point $\mathbf{P}$ is replaced by the Beilinson--Kato class $\zeta^{\mathrm{BK}}$.
Precisely, making use of the work of Kato and Ochiai, we prove in Section $\ref{Rubin GZ}$
the equality in $\mathscr{J}^{2}/\mathscr{J}^{3}$:
\begin{equation}\label{eq:main intro}
         \log_{A}\big(\mathrm{res}_{p}\big(\zeta^{\mathrm{BK}}\big)\big)\cdot{} L_{p}(f_{\infty},k,s)\ 
         \ \mathrm{mod}\ \mathscr{J}^{3}=
         \frac{-1}{\mathrm{ord}_{p}(q_{A})}\lri{1-\frac{1}{p}}^{-1}\cdot{}\widetilde{h}_{p}(\zeta^{\mathrm{BK}}).
\end{equation}
%
Combined with $(\ref{eq:main BD introduction})$ and  $(\ref{eq:height s=k/2})$, this gives 
\[
           \log_{A}^{2}\big(\mathrm{res}_{p}\big(\zeta^{\mathrm{BK}}\big)\big)=\ell_{1}\cdot{}\log_{A}^{2}(\mathbf{P})\cdot{}
           \log_{A}\big(\mathrm{res}_{p}\big(\zeta^{\mathrm{BK}}\big)\big),
\]
where $\ell_{1}:=-2\ell\cdot{}\mathrm{ord}_{p}(q_{A}){}\big(1-p^{-1}\big)$.
We then show that  $\mathrm{res}_{p}\big(\zeta^{\mathrm{BK}}\big)\not=0$ and deduce Theorem A.
Now, thanks to the   
theorem of  Gross--Zagier--Kolyvagin, one has  $\zeta^{\mathrm{BK}}=\lambda\cdot{}\mathbf{P}$,
with $\lambda=\log_{A}\big(\mathrm{res}_{p}\big(\zeta^{\mathrm{BK}}\big)\big)
/\log_{A}(\mathbf{P})\in{}\Q_{p}^{\ast}$. Then $\widetilde{h}_{p}(\zeta^{\mathrm{BK}})=\lambda^{2}\cdot{}\widetilde{h}_{p}(\mathbf{P})$.
If one sets  $\ell_{2}:=2\ell$, Theorem A and equation $(\ref{eq:main intro})$ yield Theorem C, namely
\[
              L_{p}(f_{\infty},k,s)\ \mathrm{mod}\ \mathscr{J}^{3}=\ell_{2}\cdot{}\widetilde{h}_{p}(\mathbf{P}).\vspace{2mm}
\]

\vspace{-5mm}
\ \\
\emph{Organisation of the paper.} 
Section 2 recalls the known results needed in the rest of the paper.
This includes some basic facts from Hida's theory, 
the main result of \cite{B-D} mentioned above, 
Ochiai's construction of a  \emph{two variable  big dual exponential} and  a general version of 
Kato's reciprocity law.
In Section 3 we compute the \emph{derivative} of Ochiai's big dual exponential.
Section 4 introduces the height-weight pairing $\hwp{-}{-}$ and discusses its basic properties. In Section 5, 
we use the computations carried out in Section 3 to prove
certain \emph{exceptional zero Rubin's formulae}, 
relating the big dual exponential and the height-weight pairing. Combining these formulae with Kato's work,
we are able to prove a variant of the main result of \cite{G-S} and to prove the key equality $(\ref{eq:main intro})$ appearing above.
Finally, in Section 6 we prove the results stated above.\vspace{2mm}\\
\emph{Acknowledgements.}
Much of the work on this article was carried out during my Ph.D. at the University of Milan.
It is a pleasure to express my sincere gratitude to my supervisor, Prof. Massimo Bertolini, who constantly
encouraged and motivated my work. Every meeting with him has been a source of ideas and enthusiasm;
this paper surely originated from and grew up through these meetings.
I would like to thank Marco Seveso for a careful reading of the paper and for many interesting discussions related to this work.
I am also grateful to the anonymous referee; the current version of the article is greatly inspired by
his/her corrections and valuable comments, which  helped me 
to significantly clarify and  improve the exposition.


\section{Hida families, exceptional zeros and Euler systems}\label{hidath}

\subsection{The Hida family}\label{hifa} 
Set $\Gamma:=1+p\mathbf{Z}_{p}$ and $\iw:=\mathbf{Z}_{p}\llbracket{}\Gamma\rrbracket$.
Let $C$ be a finite,  flat $\hid$-algebra.
A continuous  $\mathbf{Z}_{p}$-algebra morphism
$\nu : C\fre{}\overline{\mathbf{Q}}_{p}$ is an \emph{arithmetic point} of \emph{weight} $k$ and \emph{character} $\chi$ 
if its restriction to $\Gamma$ under the structural morphism 
is of the form $\gamma\mapsto{}\gamma^{k-2}\cdot{}\chi(\gamma)$,
for an integer $k\geq{}2$  and a  character  $\chi : \Gamma\fre{}\overline{\mathbf{Q}}_{p}^{\ast}$
of finite order. Denote by
$\xari(C)$
the  set of arithmetic points of $C$.
 
Let  $f=\sum_{n=1}^{\infty}a_{n}(A)\cdot{}q^{n}\in{}S_{2}(\Gamma_{0}(Np),\mathbf{Z})$ be the weight-two newform attached to $A/\mathbf{Q}$
by the modularity theorem of Wiles, Taylor--Wiles \emph{et alii}. 
According to the work of  Hida \cite{H-1}, \cite{H-2} there exists 
an \emph{$R$-adic  eigenform} of tame level $N$:
\[
                \mathbf{f}=\sum_{n=1}^{\infty}\mathbf{a}_{n}\cdot{}q^{n}\in{}R\llbracket{}q\rrbracket
\]
passing through $f$. Here $R=R_{f}$ is a \emph{normal} local Noetherian domain, finite and flat over $\hid$,
and $\mathbf{f}$ is a formal power series with coeff{}icients in $R$
satisfying the following properties. For every arithmetic point $\nu\in{}\xari(R)$ of weight $k\geq{}2$ and
character $\chi$,  the \emph{$\nu$-specialisation}
\[
                  f_{\nu}:=\sum_{n=1}^{\infty}\nu(\mathbf{a}_{n})\cdot{}q^{n}\in{}S_{k}(\Gamma_{0}(Np^{r}),\chi\omega^{2-k})
\]
is the $q$-expansion of an $N$-new $p$-ordinary  Hecke eigenform of level $Np^{r}$, weight $k$, and character 
$\chi\cdot{}\omega^{2-k}$.
Here  $r$ is the smallest positive integer such that $1+p^{r}\mathbf{Z}_{p}\subset{}\ker(\chi)$
and $\omega$ is the Teichm\"uller character.  
Moreover, there exists a distinguished  arithmetic point
$\psi=\nu_{f}\in{}\xari(R)$ of weight $2$ and trivial character such that
\[
                     f=f_{\psi}.
\]
With the notations of Section 1 of \cite{H-2}, let $h^{o}(N;\Z_{p})$
be the universal $p$-ordinary Hecke algebra of tame level $N$. 
Diamond operators give a morphism of $\Z_{p}$-algebras $[\cdot{}] : \iw\fre{}h^{o}(N;\Z_{p})$, making
$h^{o}(N;\Z_{p})$ a free, finitely generated $\iw$-module \cite[Theorem 3.1]{H-1}.
(We assume here that $[\cdot{}]$ is normalised as in Section 1.4 of \cite{N-P}.)
The ring $R$, denoted $\mathscr{I}(\mathscr{K})$ in \cite{H-2}, is the integral closure of 
$\iw$ in the primitive component $\mathscr{K}=\mathscr{K}_{f}$ of $h^{o}(N;\Z_{p})\otimes_{\iw}\mathrm{Frac}(\iw)$ to which $f$ \emph{belongs}
\cite[Corollary 1.3]{H-2}.

Let $\nu\in{}\xari(R)$. By \cite[Corollary 1.4]{H-2} the localisation of 
$R$ at the kernel of $\nu$  is a discrete valuation ring,
unramified over the localisation of $\hid$ at $\hid\cap{}\ker(\nu)$. In particular,  
fix a topological generator $\gamma_{0}\in{}\Gamma$,  let $\varpi:=\gamma_{0}-1\in{}\iw$ and write 
$\mathfrak{p}=\mathfrak{p}_{\psi}:=\ker(\psi)$.
Then
\begin{equation}\label{eq:uniforler}
                \mathfrak{p}R_{\mathfrak{p}}=\varpi\cdot{}R_{\mathfrak{p}},
\end{equation}
i.e. $\varpi$ is a prime element of $R_{\mathfrak{p}}$.

\subsection{Hida's $R$-adic representation}\label{hida rep}
Let $\T=\T_{\mathbf{f}}$ be the $p$-ordinary $R$-adic representation attached by Hida to $\mathbf{f}$
in \cite[Theorem 2.1]{H-2}. More precisely, let $J_{\infty}^{o}[p^{\infty}]$ be the `big' 
$p$-divisible group 
appearing in Section 8 of \cite{H-2}, which is a $h^{o}(N;\Z_{p})$-module of co-finite rank.  We define   
$\T:=\Hom{\Z_{p}}(J_{\infty}^{o}[p^{\infty}],\mu_{p^{\infty}})\otimes_{h^{o}(N;\Z_{p})}R$.
It is a rank-two $R$-module, equipped with a continuous $R$-linear action of $G_{\Q}$, which is unramified 
at every rational prime $l\nmid{}Np$.
According to Th\'eor\`eme 7 of \cite{M-T} our assumption on the
irreducibility of $A_{p}$ implies that  
$\T$  is a free $R$-module of rank two 
and that
\[
         \mathrm{Trace}\big(\mathrm{Frob}_{l}|\T\big)=\mathbf{a}_{l}; \ \ \ 
              \det\big(\mathrm{Frob}_{l}|\T\big)=l[l]
\]
for every $l\nmid{}Np$,  
where 
$\mathrm{Frob}_{l}$ is an arithmetic Frobenius at $l$
and $[\cdot{}] : \Gamma\subset{}\iw\fre{}R$ is the structural morphism
\footnote{Th\'eor\`eme 7 of \cite{M-T} proves these facts assuming that the residual Galois representation 
$\overline{\rho}_{\mathbf{f}}$
of $\T$ is absolutely irreducible. As pointed out to us by J. \neko{}, \emph{loc. cit.} also requires $\overline{\rho}_{\mathbf{f}}$ to be \emph{$p$-distinguished} (see \cite{EPW}).
As $\overline{\rho}_{\mathbf{f}}\cong{}A_{p}$ and $p\not=2$, this hypothesis is 
automatically satisfied in our case, by Tate's theory of $p$-adic uniformisation.}. 

\subsubsection{Ramification at $p$} Let $G_{p}:=G_{\mathbf{Q}_{p}}\hookrightarrow{}G_{\mathbf{Q}}$
be the decomposition group determined by our choice of  $i_{p} : \overline{\mathbf{Q}}\hookrightarrow{}\overline{\mathbf{Q}}_{p}$
and let $I_{p}:=I_{\mathbf{Q}_{p}}$ be its inertia subgroup.  
By \emph{loc. cit.} (see also \cite[Section 1.5]{N-P}) there exists an exact sequence 
of $R[G_{p}]$-modules \vspace{-2mm}
\begin{equation}\label{eq:filtrationT}
              0\lfre{}\T^{+}\lfre{i^{+}}\T\lfre{p^{-}}\T^{-}\lfre{}0,
\end{equation}
where $\T^{+}$ and $\T^{-}$ are free $R$-modules of rank $1$ and $\T^{-}$
is unramified. 
Moreover,  write 
$\widetilde{\mathbf{a}}_{p} : G_{p}\twoheadrightarrow{}G_{p}/I_{p}\fre{}R^{\ast}$
for the unramified character  sending the arithmetic Frobenius $\mathrm{Frob}_{p}
\in{}G_{p}/I_{p}$
to the $p$-th Hecke operator $\mathbf{a}_{p}$.
Then $G_{p}$ acts on $\T^{-}$ via $\widetilde{\mathbf{a}}_{p}$
and on $\T^{+}$ via  
$\widetilde{\mathbf{a}}_{p}^{-1}\chi_{\mathrm{cyc}}\left[\kappa_{\mathrm{cyc}}\right]$,
i.e.
\begin{equation}\label{eq:actionramT}
                        \T^{+}\cong{}
                        R\Big(\chi_{\mathrm{cyc}}{}\big[\kappa_{\mathrm{cyc}}\big]{}\widetilde{\mathbf{a}}_{p}^{-1}\Big); \ \ 
                        \T^{-}\cong{}R\big(\widetilde{\mathbf{a}}_{p}\big).
\end{equation}
As in the introduction,  
$\chi_{\mathrm{cyc}} : G_{\mathbf{Q}}\twoheadrightarrow{}\mathbf{Z}_{p}^{\ast}$ is the $p$-adic cyclotomic character,
and $\kappa_{\mathrm{cyc}} : G_{\mathbf{Q}}\twoheadrightarrow{}\Gamma$ is the composition  of $\chi_{\mathrm{cyc}}$ 
with the projection to principal units.

\subsubsection{Specialisations}\label{specialisations}
Let  $\nu\in{}\xari(R)$, let $K_{\nu}:=\mathrm{Frac}(\nu(R))$ and let 
$V_{\nu}$ be the contragredient of the $K_{\nu}$-adic Deligne representation of $G_{\Q}$ attached to 
the eigenform $f_{\nu}$.  
It follows from \cite[Theorem 1.4.3]{Ohta} that 
the representation $\T_{\nu}:=\T\otimes_{R,\nu}\nu(R)$ is canonically isomorphic to a Galois-stable 
$\nu(R)$-lattice in $V_{\nu}$; in particular 
there is a natural isomorphism \vspace{-1mm}
\begin{equation}\label{eq:ESdis}
                    \T_{\nu}\otimes_{\Z_{p}}\Q_{p}\cong{}V_{\nu}.
\end{equation}
We identify from now on $\T_{\nu}$ with a Galois-stable $\nu(R)$-lattice in $V_{\nu}$.

Considering the arithmetic point $\psi\in{}\xari(R)$ corresponding to $f$,
one has $\T_{\psi}\otimes_{\Z_{p}}\Q_{p}\cong{}V_{p}(A)$.
Indeed, the  irreducibility of $A_{p}$ implies that $\psi$ induces a canonical isomorphism  of 
$\mathbf{Z}_{p}[G_{\mathbf{Q}}]$-modules
\begin{equation}\label{eq:ES tate}
          \pi_{f} :  \T_{\psi}\cong{}T_{p}(A).
\end{equation}

Recall the Tate parametrisation $\Phi_{\mathrm{Tate}}$ introduced in $(\ref{eq:Tate iso})$.
As $q_{A}$ has positive valuation, $\Phi_{\mathrm{Tate}}$ induces 
on the $p$-adic Tate modules a short exact sequence of $\mathbf{Z}_{p}[G_{p}]$-modules \vspace{-2mm}
\begin{equation}\label{eq:tate exact}
                     0\lfre{}\mathbf{Z}_{p}(1)\lfre{i^{+}}T_{p}(A)\lfre{p^{-}}\mathbf{Z}_{p}\lfre{}0.
\end{equation}
We also write  $T_{p}(A)^{+}:=\mathbf{Z}_{p}(1)$ and $T_{p}(A)^{-}:=\mathbf{Z}_{p}$.
By $(\ref{eq:actionramT})$ there are  isomorphisms of $G_{p}$-modules
\begin{equation}\label{eq:ES tate pm}
               \pi_{f}^{+} : \T_{\psi}^{+}:=\T^{+}\otimes_{R,\psi}\mathbf{Z}_{p}\cong{}\mathbf{Z}_{p}(1);\ \ \ 
               \pi_{f}^{-} : \T_{\psi}^{-}:=\T^{-}\otimes_{R,\psi}\mathbf{Z}_{p}\cong{}\mathbf{Z}_{p}.
\end{equation}
We can, and will, normalise $\pi_{f}^{\pm}$ in such a way that they are  
compatible with $\pi_{f}$.

\subsection{$p$-adic $L$-functions}\label{padicl}
Let $G_{\infty}$
and  $\Ic$
be as in the introduction,
and define $\overline{R}:=R\llbracket{}G_{\infty}\rrbracket=R\widehat{\otimes}_{\mathbf{Z}_{p}}\Ic$. Under our assumptions
Section 3.4 of \cite{EPW}
(using ideas from  \cite{Kit} and  \cite{G-S}) attaches to $\mathbf{f}$ an element
\[
          L_p(\mathbf{f})\in{}\R,
\]
unique  up to multiplication by units in $R$,
which interpolates  the Mazur--Tate--Teitelbaum $p$-adic $L$-functions of the arithmetic specialisations of $\mathbf{f}$.
More precisely, given $\nu\in{}\xari(R)$,
let $\overline{R}_{\nu}:=\nu(R)\llbracket{}G_{\infty}\rrbracket$
and write again 
$\nu : \R\twoheadrightarrow{}\overline{R}_{\nu}$
for the morphism of $\Ic$-algebras induced by $\nu$.
Fix also a \emph{canonical Shimura period} $\Omega_{\nu}\in{}\mathbb{C}^{\ast}$ for 
$f_{\nu}$ (see \cite[Sec. 3.1]{EPW}).
Then, for every $\nu\in{}\xari(R)$, there exists a scalar $\lambda_{\nu}\in{}\nu(R)^{\ast}$ such that
\[
          \nu\lri{L_p(\mathbf{f})}
          =\lambda_{\nu}\cdot{}L_{p}(f_{\nu})\in{}\overline{R}_{\nu},
\]
where $L_{p}(f_{\nu})=L_{p,\Omega_{\nu}}(f_{\nu})$
is the Mazur--Tate--Teitelbaum $p$-adic $L$-function attached in \cite{M-T-T} to $f_{\nu}$,
normalised with respect to  $\Omega_{\nu}$ (see also \cite[Section 4]{G-S}).
It is characterised by the following interpolation property: let $k_{\nu}$
be the weight of $\nu$. Then 
for every finite order character 
$\chi : G_{\infty}\fre{}\overline{\mathbf{Q}}_{p}^{\ast}$
and every integer $0<s_{0}<k_{\nu}$
\begin{equation}\label{eq:intmtt}
      \chi\cdot{}\chi_{\mathrm{cyc}}^{s_{0}-1}\big(L_{p}(f_{\nu})\big)
      =\nu(\mathbf{a}_{p})^{-m} 
      \lri{1-\frac{\chi\omega^{1-s_0}(p)\cdot{}p^{s_0-1}}{\nu(\mathbf{a}_{p})}}L^{\mathrm{alg}}
      (f_{\nu},\chi\omega^{1-s_0},s_0),
\end{equation}
where $m$ is the $p$-adic valuation of the conductor of  $\chi$ and 
\[
        L^{\mathrm{alg}}(f_{\nu},\chi\omega^{1-s_{0}},s_0)
        :=\tau\big(\chi\omega^{1-s_{0}}\big)_{}{}{p^{m(s_0-1)}(s_0-1)!}
        \frac{L(f_{\nu},\chi^{-1}\omega^{s_{0}-1},s_0)}{(2\pi{}i)^{s_0-1}\Omega_{\nu}}\in{}\overline{\mathbf{Q}}.
\]
For a Dirichlet character $\mu$, $\tau(\mu)$ is the  Gauss sum  of $\mu$  and $L(f_{\nu},\mu,s)$
is the Hecke  $L$-function of $f_{\nu}$ twisted by $\mu$.

According to \cite[Sec. 3]{GV}, under our assumptions we can choose $\Omega_{\psi}=\Omega_{A}^{+}$
as the real N\'eron period of $A/\mathbf{Q}$,
so that $L_{p}(A/\mathbf{Q}):=L_{p}(f_{\psi})$ is the \emph{$p$-adic $L$-function of $A/\mathbf{Q}$}. 
Here we insist to make this choice and to normalise 
$L_p(\mathbf{f})$ by requiring  $\lambda_{\psi}=1$, i.e.
\begin{equation}\label{eq:hht}
     \psi\big(L_p(\mathbf{f})\big)=L_p(A/\mathbf{Q}).
\end{equation}
Then $L_p(\mathbf{f})$
is a well-defined element of $\overline{R}$ up to multiplication by units  $\alpha\in{}R^{\ast}$ such that  $\psi(\alpha)=1$.

\subsubsection{Exceptional zeros}\label{exczerosec} The \emph{$p$-adic multiplier}
\[
     E_{p}(\nu,\chi\cdot{}\chi_{\mathrm{cyc}}^{j}):=\lri{1-\frac{\chi\omega^{-j}(p)\cdot{}p^{j}}{\nu(\mathbf{a}_{p})}}
\]      
which appears 
in the interpolation formula $(\ref{eq:intmtt})$ is  responsible for the phenomenon of  \emph{exceptional zeros} 
mentioned in the introduction (cf. \cite{M-T-T}). 
Indeed $\psi(\mathbf{a}_{p})=a_{p}(A)=+1$ in our setting, and
$E_{p}(\psi,1)=0$.
In particular, let $I=I_{\mathrm{cyc}}$ be the augmentation ideal of $\Ic$
and let  $\overline{\mathfrak{p}}=(\mathfrak{p},I)$ be the ideal of $\overline{R}$ 
generated by $I$ and $\mathfrak{p}$. Then
\begin{equation}\label{eq:exzerolmk}
    L_{p}(\mathbf{f})\in{}\overline{\mathfrak{p}};\ \ L_{p}(A/\mathbf{Q})\in{}I.
\end{equation}

%
\subsubsection{The improved $p$-adic $L$-function}\label{improved integral} Let $\varepsilon : \overline{R}\twoheadrightarrow{}R$
be the augmentation map. 
By \cite[Remark 3.4.5]{EPW} (generalising a result of \cite{G-S})
there is a  factorisation
\begin{equation}\label{eq:improved integral}
             \varepsilon\big(L_{p}(\mathbf{f})\big)=\lri{1-\mathbf{a}_{p}^{-1}}\cdot{}L_{p}^{\ast}(\mathbf{f}),
\end{equation}
for an element $L_{p}^{\ast}(\mathbf{f})\in{}R$
called the \emph{improved $p$-adic $L$-function} of $\mathbf{f}$.

\subsection{The analytic Mellin transform}\label{mellinhida} 
As explained in \cite[Section 2.6]{G-S} (see also \cite[Section 1.4.7]{N-P}),
there exist a disc $U\subset{}\Z_{p}$ centred at $2$ and a unique morphism of $\iw$-modules 
\[
               \M=\M_{f} : R\longrightarrow{}\mathscr{A}_{U}
\]
such that $\M(r)|_{k=2}=\psi(r)$ for every $r\in{}R$. 
Here
$\mathscr{A}_{U}\subset{}\Q_{p}\llbracket{}k-2\rrbracket$
(see Section $\ref{intro 2}$)
is endowed with the structure of a $\hid$-algebra via the character $\Gamma\fre{}\mathscr{A}_{U}$
which sends  $\gamma\in{}\Gamma$
to the power series $\gamma^{k-2}:=\exp_{p}\big((k-2)\cdot{}\log_{p}(\gamma)\big)$. 
%
The morphism $\M$ is called the \emph{Mellin transform centred at $k=2$}.
For every $n\in{}\mathbf{N}$, 
set $a_{n}(k):=\M(\mathbf{a}_{n})$
and define 
\[
            f_{\infty}:=\sum_{n=1}^{\infty}a_n(k)\cdot{}q^n\in{}\mathscr{A}_{U}\llbracket{}q\rrbracket.
\]

Let $\mathscr{A}\subset{}\mathbf{Q}_{p}\llbracket{}k-2,s-1\rrbracket{}$ 
and $\mathscr{J}\subset{}\mathscr{A}$ be as in Section $\ref{intro 2}$. 
Then $\mathscr{A}$ has a structure of $\Ic$-algebra, induced by the character 
$G_{\infty}\fre{}\mathscr{A}$ mapping
$g\in{}G_{\infty}$ to 
$\chi_{\mathrm{cyc}}(g)^{s-1}:=\exp_{p}\big((s-1)\cdot{}\log_{p}\big(\chi_{\mathrm{cyc}}(g)\big)\big)$. 
Moreover there exists a unique morphism of $\Ic$-algebras 
$$\Mm=\Mm_{f} : \overline{R}\lfre{}\mathscr{A}$$
whose restriction to $R$ equals $\M$,  called the   \emph{Mellin transform centred at $(k,s)=(2,1)$}.
Define the \emph{Mazur--Kitagawa $p$-adic $L$-function of $f_{\infty}$}:
\begin{equation}\label{eq:definition LMK}
             L_{p}(f_{\infty},k,s):=\Mm\big(L_{p}(\mathbf{f})\big)\in{}\mathscr{J}
\end{equation}
as the Mellin transform of $L_{p}(\mathbf{f})\in{}\overline{R}$. More precisely, it is a well-defined element of $\mathscr{A}$
up to multiplication by a nowhere-vanishing  function $\alpha(k)\in\mathscr{A}_{U}$ such that $\alpha(2)=1$,
and belongs to $\mathscr{J}$
by equation $(\ref{eq:exzerolmk})$.
In the introduction we defined  $L_{p}(A/\mathbf{Q},s):=\chi_{\mathrm{cyc}}^{s-1}\big(L_{p}(A/\mathbf{Q})\big)=\Mm\big(L_{p}(A/\mathbf{Q})\big)$, so that 
equation
$(\ref{eq:hht})$ gives
\begin{equation}\label{eq:speMK}
                  L_{p}(f_{\infty},2,s)=L_{p}(A/\mathbf{Q},s).
\end{equation}
According to Theorem 5.15 of \cite{G-S} $L_{p}(f_{\infty},k,s)$ satisf{}ies the functional equation
\begin{equation}\label{eq:funeqfinal}
                  \Lambda_{p}(f_{\infty},k,s)=-\mathrm{sign}(A/\mathbf{Q})\cdot{}\Lambda_{p}(f_{\infty},k,k-s),
\end{equation}
where $\Lambda_{p}(f_{\infty},k,s):=\dia{N}^{s/2}\cdot{}L_{p}(f_{\infty},k,s)$, $\langle\cdot{}\rangle : \Z_{p}^{\ast}\twoheadrightarrow{}1+p\Z_{p}$
denotes the projection to principal units and $\mathrm{sign}(A/\Q)\in{}\{\pm1\}$ is the sign in the functional equation 
satisfied by the Hasse--Weil $L$-function of $A/\Q$. 
Note that the \emph{central critical line} $s=k/2$ is the `centre of symmetry' of the functional equation. 
In particular, when $\mathrm{sign}(A/\mathbf{Q})=+1$, $L_{p}(f_{\infty},k,k/2)$ 
vanishes identically.

Write $L_{p}^{\ast}(f_{\infty},k):=\M\big(L_{p}^{\ast}(\mathbf{f})\big)\in{}\mathscr{A}_{U}$. As $\M\circ{}\varepsilon=\Mm(\cdot)|_{s=1}$, equation $(\ref{eq:improved integral})$ gives a factorisation
in $\mathscr{A}_{U}$: 
\begin{equation}\label{eq:essendefimproved}
                  L_{p}(f_{\infty},k,1)=\lri{1-a_{p}(k)^{-1}}\cdot{}L_{p}^{\ast}(f_{\infty},k).
\end{equation}
The function $L_{p}^{\ast}(f_{\infty},k)$ is called  the \emph{improved $p$-adic $L$-function of $f_{\infty}$}.

\subsection{The Bertolini--Darmon exceptional zero formula}\label{BD formula}
The following result has been proved in \cite{B-D}, assuming a mild technical condition subsequently removed in 
\cite[Section 6]{Mok}. Denote by $L_{p}^{\mathrm{cc}}(f_{\infty},k)\in{}\mathscr{A}_{U}$
the restriction of $L_{p}(f_{\infty},k,s)$ to the central critical line $s=k/2$.

\begin{theo}\label{main Bertolini-Darmon} There exist a non-zero rational number $\ell\in{}\Q^{\ast}$ and a rational point $\mathbf{P}\in{}A(\Q)\otimes\Q$
such that 
\[
             \frac{d^{2}}{dk^{2}}L_{p}^{\mathrm{cc}}(f_{\infty},k)_{k=2}=\ell\cdot{}\log_{A}^{2}(\mathbf{P}).
\]
Moreover, $\mathbf{P}$ is non-zero if and only if $L(A/\Q,s)$ has a simple zero at $s=1$.
\end{theo}

\begin{remark}\label{remark main BD}\emph{ Assume for simplicity that $\mathrm{sign}(A/\Q)=-1$ and that 
$N\not=1$ is not square-full   (see \cite{Mok} for the general case).
As explained in \cite{B-D}, the definitions of $\mathbf{P}$ and $\ell$ rest on the choice of an auxiliary 
imaginary quadratic field $K/\Q$ satisfying the following conditions. Let $D_{K}$
and  $\epsilon_{K} : \lri{\Z/D_{K}\Z}^{\ast}\fre{}\{\pm1\}$ denote the 
discriminant and  the quadratic character of $K$ respectively. 
\begin{itemize}
\item[$(\alpha)$] $(D_{K},Np)=1$ and there is a  factorisation $Np=pN^{+}N^{-}$, such that $pN^{-}$ is square-free and a prime divisor 
of $Np$ divides $pN^{-}$ if  and only if  it is inert in $K$.
\item[$(\beta)$] The special value $L(A/\Q,\epsilon_{K},1)$ is non-zero. 
\end{itemize}
Then $\mathbf{P}$
is defined as the trace to $\Q$
of a Heegner point  in $A(K)\otimes{}\Q$,  coming from a parametrisation of $A/\Q$ by the Shimura curve 
$X_{N^{+},pN^{-}}$ associated with  an Eichler order of level $N^{+}$ in the
indefinite quaternion algebra  of discriminant  $pN^{-}$.
The rational number $\ell$ is defined by the  relation
\[
              2\ell^{-1}=\eta_{f}\cdot{}\sqrt{D_{K}}\cdot{}\frac{L(A/\Q,\epsilon_{K},1)}{\Omega_{A}^{-}}\in{}\Q^{\ast}.
\] 
Here
$\Omega_{A}^{-}\in{}i\mathbf{R}^{\ast}$ is such that  $\Omega_{A}^{+}\cdot{}\Omega_{A}^{-}$
is  the  Petersson norm of $f$.
The constant $\eta_{f}:=\dia{\phi_{f},\phi_{f}}\in{}\Q^{\ast}$
is the Petersson norm of a (suitably normalised)  Jacquet--Langlands lift of $f$ to an eigenform $\phi_{f}$ on the  definite quaternion 
algebra 
of discriminant  $N^{-}\infty$ (cf. Sections  2.2 and 2.3 of \cite{B-D}).
Note that both $\mathbf{P}$ and $\ell$ depend on the choice of  $K/\Q$,
while the product $\ell\cdot{}\log_{A}^{2}(\mathbf{P})$ does not.}
\end{remark}

\subsection{Ochiai's  big dual exponential}\label{expbk}
We recall here  the definition of   Ochiai's  \emph{two-variable big dual exponential} 
for $\T$, constructed  in  \cite{Ochiaicol} using  previous work of Coleman--Perrin-Riou.

\subsubsection{Notations}
For every $n\in{}\mathbf{N}\cup\{\infty\}$, let $\Q_{p,n}$ 
be as in the introduction.
The Galois group of $\Q_{p,\infty}/\Q_{p}$ is naturally identified with $G_{\infty}=\mathrm{Gal}(\Q_{\infty}/\Q)$,
via the unique prime of $\Q_{\infty}$ dividing $p$.

Given $n\in{}\N$ and  a $p$-adic representation $V$ of $G_{p}=G_{\Q_{p}}$,
let $D_{\mathrm{dR},n}(V):=H^{0}(\Q_{p,n},V\otimes_{\Q_{p}}B_{\mathrm{dR}})$, where $B_{\mathrm{dR}}$
is Fontaine's field of periods.
It is equipped with a complete and separated decreasing filtration 
$\mathrm{Fil}^{\bullet}D_{\mathrm{dR},n}(V)$, arising from the 
filtration $\{\mathrm{Fil}^{n}B_{\mathrm{dR}}:=t^{n}B_{\mathrm{dR}}^{+}\}_{n\in{}\Z}$, where 
$B_{\mathrm{dR}}^{+}$ is the ring of integers of $B_{\mathrm{dR}}$ and $t:=\log(\zeta_{\infty})$,
for a fixed generator $\zeta_{\infty}\in{}\Z_{p}(1)$. Denote by
$\mathrm{tg}_{n}(V):=D_{\mathrm{dR},n}(V)/\mathrm{Fil}^{0}$
the tangent space of the $G_{\Q_{p,n}}$-representation $V$. If  $n=0$,
it will be omitted  from the notations (e.g. $D_{\mathrm{dR}}(V)=D_{\mathrm{dR},0}(V)$).
If $V$ is a de Rham representation of $G_{p}$, there is a  natural $\mathrm{Gal}(\Q_{p,n}/\Q_{p})$-equivariant 
isomorphism 
of filtered modules $D_{\mathrm{dR},n}(V)=D_{\mathrm{dR}}(V)\otimes_{\Q_{p}}\Q_{p,n}$.

Let $S$ be a complete, local Noetherian ring with finite residue field of characteristic $p$ and let 
$\mathbb{X}$ be a free $S$-module of finite rank, equipped with a continuous $S$-linear action of $G_{p}$.
Define  
\[
               H^{q}_{\mathrm{Iw}}(\Q_{p,\infty},\mathbb{X}):=\inlim_{n\in{}\mathbf{N}}H^{q}(\Q_{p,n},\mathbb{X}),
\]
where the limit  is taken with respect to the corestriction  maps in Galois cohomology.
Galois conjugation equips  $H^{q}_{\mathrm{Iw}}(\Q_{p,\infty},\mathbb{X})$ with the  structure of
a module over  the completed group algebra $\overline{S}:=S\llbracket{}G_{\infty}\rrbracket$.

For every $R$-module $\mathbb{M}$ and every $\nu\in{}\xari(R)$,
write $\mathbb{M}_{\nu}:=\mathbb{M}\otimes_{R,\nu}\nu(R)$.

\subsubsection{de Rham modules} Set $\check{\T}:=\mathrm{Hom}_{R}(\T,R)$ and 
$\check{\T}^{\pm}:=\mathrm{Hom}_{R}(\T^{\mp},R)$. 
Let  $\nu\in{}\xari(R)$.
Since $\T_{\nu}$ is a Galois-stable lattice in $V_{\nu}$ by $(\ref{eq:ESdis})$,
$\check{\T}_{\nu}$
is a Galois-stable lattice in 
the Deligne representation $\check{V}_{\nu}=\mathrm{Hom}_{K_{\nu}}(V_{\nu},K_{\nu})$ of $f_{\nu}$,
where $K_{\nu}:=\mathrm{Frac}(\nu(R))$.
Define $V_{\nu}^{\pm}:=\T_{\nu}^{\pm}\otimes_{\Z_{p}}\Q_{p}$ 
and  $\check{V}_{\nu}^{\pm}:=\check{\T}_{\nu}^{\pm}\otimes_{\Z_{p}}\Q_{p}$. 
According to $(\ref{eq:filtrationT})$, for $M_{\nu}\in{}\{V_{\nu},\check{V}_{\nu}\}$  there is  a short exact sequence of $K_{\nu}[G_{p}]$-modules 
\vspace{-2mm}
\[
         0\lfre{}M_{\nu}^{+}\lfre{i^{+}}M_{\nu}\lfre{p^{-}}M_{\nu}^{-}\lfre{}0.
\]

The representation $\check{V}_{\nu}$ is known to be de Rham, and then so is
$V_{\nu}$.
In addition, $\mathrm{Fil}^{0}D_{\mathrm{dR}}(\check{V}_{\nu})=D_{\mathrm{dR}}(\check{V}_{\nu})$ 
and $\mathrm{Fil}^{m}D_{\mathrm{dR}}(\check{V}_{\nu})$ is $1$-dimensional over $K_{\nu}$
(resp., zero) for every $1\leq{}m\leq{}k_{\nu}-1$ (resp., $m\geq{}k_{\nu}$),
where $k_{\nu}\geq{}2$ is the weight of $\nu$.
It follows  easily from $(\ref{eq:actionramT})$ that $p^{-}
: V_{\nu}\twoheadrightarrow{}V_{\nu}^{-}$ and $i^{+} : \check{V}_{\nu}^{+}\hookrightarrow{}\check{V}_{\nu}$
induce isomorphisms of $K_{\nu}$-modules
\begin{equation}\label{eq:tangent and fil}
                  \mathrm{Fil}^{0}D_{\mathrm{dR},n}(V_{\nu})\cong{}D_{\mathrm{dR},n}(V_{\nu}^{-});\ \ \ 
               D_{\mathrm{dR},n}(\check{V}_{\nu}^{+}(1))\cong{}\mathrm{tg}_{n}(\check{V}_{\nu}(1))
\end{equation}
for every $n\in{}\mathbf{N}$, which we  consider as equalities in what follows.

For every $n\in{}\N$ the duality $V_{\nu}\times{}\check{V}_{\nu}(1)\fre{}K_{\nu}(1)$ induces
a $K_{\nu}$-bilinear form
\[
         \dia{-,-}_{\mathrm{dR}}=\dia{-,-}_{\mathrm{dR},n} : \mathrm{Fil}^{0}D_{\mathrm{dR},n}(V_{\nu})\times{}\mathrm{tg}_{n}(\check{V}_{\nu}(1))
         \lfre{\cup\ }D_{\mathrm{dR},n}(K_{\nu}(1))=\Q_{p,n}\otimes_{\Q_{p}}K_{\nu}.
\] 
Under the isomorphisms $D_{\mathrm{dR},n}(M)=D_{\mathrm{dR}}(M)\otimes_{\Q_{p}}\Q_{p,n}$,
for $M=V_{\nu},\check{V}_{\nu}(1)$,
the pairing  $\dia{-,-}_{\mathrm{dR},n}$ is identified with the  $\Q_{p,n}$-base change of $\dia{-,-}_{\mathrm{dR},0}$.
Denote also by   $\dia{-,-}_{\mathrm{dR}} : \mathrm{Fil}^{0}D_{\mathrm{dR},n}(V_{\nu})\times{}\mathrm{tg}_{n}(\check{V}_{\nu}(1))
\fre{}K_{\nu}(\mu_{p^{n+1}})$  the bilinear form defined  by composing  $\dia{-,-}_{\mathrm{dR}}$
with  the multiplication $K_{\nu}\otimes_{\Q_{p}}\Q_{p,n}\fre{}K_{\nu}(\mu_{p^{n+1}})$.

\subsubsection{Variation of periods}\label{beka0}
Let $\Q_{p}^{\mathrm{un}}$ be the maximal unramified extension of $\Q_{p}$
and let $\widehat{\Z}_{p}^{\mathrm{un}}$ be the $p$-adic completion of its ring of integers. 
Following \cite[Section 3]{Ochiaicol}, define the $R$-module 
\[
        \mathcal{D}:=H^{0}\big(\Q_{p},\widehat{\Z}_p^{\mathrm{un}}\widehat{\otimes}_{\Z_p}\check{\T}^{+}\big).
\]
By $(\ref{eq:filtrationT})$ and $(\ref{eq:actionramT})$, the $G_{p}$-module $\check{\T}^{+}$ is unramified
and free of rank one as an $R$-module. 
Then $\mathcal{D}$ is also a free $R$-module of rank one, by Lemma $3.3$ of  \cite{Ochiaicol}.
As $H^{0}(\Q_{p}^{\mathrm{un}},B_{\mathrm{dR}})=\widehat{\Z}_{p}^{\mathrm{un}}\otimes_{\Z_{p}}\Q_{p}$,
this easily implies (cf. \emph{loc. cit.}) that for every  $\nu\in{}\xari(R)$
there is  a natural isomorphism of $K_{\nu}$-modules 
$\mathcal{D}_{\nu}\otimes_{\Z_{p}}\Q_{p}\cong{}D_{\mathrm{dR}}(\check{V}_{\nu}^{+})$.
This induces  a natural 
\emph{$\nu$-specialisation map}
\[
                  \mathcal{D}\lfre{}D_{\mathrm{dR}}(\check{V}_{\nu}^{+}).
\]
For every $X\in{}\mathcal{D}$, denote by $X_{\nu}$ the $\nu$-specialisation of $X$.

Fix  a  generator $\mho$ of the $R$-module $\mathcal{D}$, which also fixes a $K_{\nu}$-basis 
\[
             \mho_{\nu}(1):=\mho_{\nu}\otimes{}\zeta_{\mathrm{dR}}\in{}
\mathrm{tg}(\check{V}_{\nu}(1)).
\]
Here $\zeta_{\mathrm{dR}}:=\zeta_{\infty}\otimes{}\log(\zeta_{\infty})^{-1}\in{}D_{\mathrm{dR}}(\Q_{p}(1))$ is the canonical $\Q_{p}$-basis
associated to a generator $\zeta_{\infty}\in{}\Z_{p}(1)$ and $\cdot\otimes{}\zeta_{\mathrm{dR}}$
is the natural isomorphism $D_{\mathrm{dR}}(\check{V}^{+}_{\nu})\cong{}
D_{\mathrm{dR}}(\check{V}^{+}_{\nu})\otimes_{\Q_{p}}D_{\mathrm{dR}}(\Q_{p}(1))=D_{\mathrm{dR}}(\check{V}^{+}_{\nu}(1))$.


By $(\ref{eq:ES tate})$ and $(\ref{eq:ES tate pm})$ one has $\T_{\psi}\cong{}T_{p}(A)$
and $\T_{\psi}^{-}\cong{}\Z_{p}$ respectively. 
Then 
$\check{V}_{\psi}(1)$ and $\check{V}^{+}_{\psi}(1)$ are identified 
with $\check{V}_{p}(A)(1)$ and $\Q_{p}(1)$ respectively,
where $\check{V}_{p}(A):=\mathrm{Hom}_{\Q_{p}}(V_{p}(A),\Q_{p})$. 
In particular $\zeta_{\mathrm{dR}}$ can be identified with an element of 
$\mathrm{tg}\big(\check{V}_{p}(A)(1)\big)$ (cf. equation $(\ref{eq:tangent and fil})$).
After possibly multiplying  $\mho$ by a unit in $R$, we can assume 
\begin{equation}\label{eq:norm mho}
              \mho_{\psi}(1)=\zeta_{\mathrm{dR}}\in{}\mathrm{tg}\big(\check{V}_{p}(A)(1)\big).
\end{equation}

\subsubsection{Ochiai's  two-variable big dual exponential}\label{ochiaisec} 

For every  $\nu\in{}\xari(R)$
and every  finite order character $\chi : G_{\infty}\fre{}\overline{\Q}_{p}^{\ast}$
write $\nu\times{}\chi : \overline{R}\fre{}\overline{\Q}_{p}$ for the unique morphism 
of $\Z_{p}$-algebras
 whose restriction to $R$ (resp., $G_{\infty}$) equals $\nu$ (resp., $\chi$).
Let $\T^{?}$ denote either $\T$ or $\T^{\pm}$.
For every  $\mathfrak{Z}_{n}\in{}H^{q}(\Q_{p,n},\T^{?})$ 
let $\mathfrak{Z}_{n,\nu}\in{}H^{q}(\Q_{p,n},V_{\nu}^{?})$  be the image of 
$\mathfrak{Z}_{n}$ under the morphism induced in cohomology by
$\T^{?}\twoheadrightarrow{}\T^{?}_{\nu}\subset{}V_{\nu}^{?}$.
Finally, for every $n\in{}\N$, write
\[
            \exp^{\ast}=\exp^{\ast}_{V_{\nu}^{-}} : H^{1}(\Q_{p,n},V_{\nu}^{-})\lfre{}
            D_{\mathrm{dR},n}(V_{\nu}^{-})\cong{}
            \mathrm{Fil}^{0}D_{\mathrm{dR},n}(V_{\nu})
\]
for the Bloch--Kato dual exponential map defined in \cite[Chapter II]{Katiw}.

The following proposition is proved in Section 5 of \cite{Ochiaicol}
(see in particular Proposition 5.1) building on previous work of Coleman \cite{Col} and Perrin-Riou \cite{PRlog}.

\begin{proposition}\label{modifiedochiai}
There exists a unique morphism of $\overline{R}$-modules
\[
                 \mathcal{L}_{\T}:=\mathcal{L}_{\T,\mho} : H^{1}_{\mathrm{Iw}}(\Q_{p,\infty},\mathbb{T}^{-})\longrightarrow{}\overline{\mathfrak{p}}\subset\overline{R}
\]
such that:
for every $\mathfrak{Z}=(\mathfrak{Z}_n)\in{}H^1_\mathrm{Iw}(\Q_{p,\infty},\mathbb{T}^{-})$,
every weight-two arithmetic point $\nu\in{}\xari(R)$ and every  character  
$\chi : \mathrm{Gal}(\Q_{p,n}/\Q_{p})\fre{}\overline{}\overline{\Q}_p^\ast$ of 
conductor  $p^{m}\leq{}p^{n+1}$
\begin{align*}
      \nu\times{}\chi\big(\mathcal{L}_{\T}(\mathfrak{Z})\big)=
      \mathcal{E}(\nu,\chi)\sum_{\sigma\in{}\mathrm{Gal}(\Q_{p,n}/\Q_{p})}\chi(\sigma)^{-1}\cdot{}
      \big\langle\exp^{\ast}    
       \big(\mathfrak{Z}_{n,\nu}^{\sigma}\big),\mho_{\nu}(1)\big\rangle_{\mathrm{dR}},
\end{align*}
where 
\[
   \mathcal{E}(\nu,\chi):=\tau(\chi)\nu(\mathbf{a}_{p})^{-m}\lri{1-\frac{\chi(p)\nu(\mathbf{a}_{p})}{p}}^{-1}   
\lri{1-\frac{\chi(p)}{\nu(\mathbf{a}_{p})}}.
\]
\end{proposition}

With a slight abuse of notation, write again 
$$\mathcal{L}_{\T} : H^{1}_{\mathrm{Iw}}(\Q_{p,\infty},\T)\lfre{}\overline{\mathfrak{p}}$$
for the composition of $\mathcal{L}_{\T}$ with the morphism $H^{1}_{\mathrm{Iw}}(\Q_{p,\infty},\T)\fre{}H^{1}_{\mathrm{Iw}}(\Q_{p,\infty},\T^{-})$
induced by $p^{-} : \T\twoheadrightarrow{}\T^{-}$.

\subsection{Beilinson--Kato elements and Kato's reciprocity law}\label{bekaochiai}
We now state a general version of Kato's reciprocity law, following Section $6$ of \cite{Ochiai} (see in particular Corollary 6.17).

Denote by $\overline{\Q}(Np)/\Q$ the maximal algebraic extension of $\Q$ which is unramified at every finite prime 
$l\nmid{}Np$, and 
set $\gaun_{n}:=\mathrm{Gal}\big(\overline{\Q}(Np)/\Q_{n}\big)$.
Let $S$ be a local complete Noetherian ring with finite residue field of characteristic $p$ 
and let $\mathbb{X}$ be a free $S$-module of finite rank, equipped with a continuous $S$-linear action of $\gaun_{0}$.
Define 
\[
   H^{q}_{\mathrm{Iw}}(\Q_{\infty},\mathbb{X}):=\inlim_{n\in{}\mathbf{N}}H^{q}(\gaun_{n},\mathbb{X}),\vspace{-1mm}
\]
where 
the limit is taken with respect to the corestriction maps.
According to \cite[Corollary B.3.6]{Rub}, if $q=1$ and $S=\Z_{p}$,
the $\Ic$-module $H^{1}_{\mathrm{Iw}}(\Q_{\infty},\mathbb{X})$ 
is isomorphic to the inverse limit of the cohomology groups 
$H^{1}(\Q_{n},\mathbb{X})$. In particular the definition of $H^{1}_{\mathrm{Iw}}(\Q_{\infty},T_{p}(A))$
given here agrees with the one given in the introduction.

\begin{theo}\label{maincohpad} There exists $\bki_{\infty}=\big(\bki_{n}\big)_{n\in\N}\in{}H^{1}_{\mathrm{Iw}}(\Q_{\infty},\defo)$
such that
$$\mathcal{L}_{\T}\Big(\mathrm{res}_{p}\big(\bki_{\infty}\big)\Big)=L_{p}(\mathbf{f}).$$
\end{theo}

\begin{remark}\emph{The preceding theorem  comes principally  from the work of Kato \cite{kateul}.
For every arithmetic point  $\nu\in{}\xari(R)$, \cite{kateul} attaches to $f_{\nu}$ a cyclotomic Euler system   
for $\T_{\nu}$, using \emph{Beilinson--Kato elements} in the $K_{2}$ of modular curves. In particular this gives a class
$\zeta^{\mathrm{BK}}_{\infty,\nu}\in{}H^{1}_{\mathrm{Iw}}(\Q_{\infty},\T_{\nu})$,
related to the $p$-adic $L$-function $L_{p}(f_{\nu})$ via the Perrin-Riou big dual exponential 
(see in particular Theorem 16.6 of  \cite{kateul}). According to Theorem 6.11 of  \cite{Ochiai}, 
the  classes $\{\zeta^{\mathrm{BK}}_{\infty,\nu}\}_{\nu}$ can be interpolated by a  \emph{two-variable Beilinson--Kato class}
$\bki_{\infty}\in{}H^{1}_{\mathrm{Iw}}(\Q_{\infty},\T)$, satisfying the conclusions of the theorem.}
\end{remark}

\section{The derivative of  Ochiai's big dual exponential}\label{derbigdualsec}
Consider the morphism of $\overline{R}$-modules
\[
          \mathcal{L}_{\T}(\cdot{},k,s):=\Mm\circ{}\mathcal{L}_{\T} : H^{1}_{\mathrm{Iw}}(\Q_{p,\infty},\T^{-})\lfre{}\mathscr{J}
          \subset{}\mathscr{A},
\]
defined as the composition of Ochiai's big dual exponential $\mathcal{L}_{\T}$ with the Mellin transform $\Mm$;
note that
$\mathcal{L}_{\T}(\cdot{},k,s)$ takes values in $\mathscr{J}\subset{}\mathscr{A}$ since
$\Mm$ maps by construction the ideal $\overline{\mathfrak{p}}$  into $\mathscr{J}$.
With a slight abuse of notation, denote again by
$\mathcal{L}_{\T}(\cdot{},k,s) : H^{1}_{\mathrm{Iw}}(\Q_{p,\infty},\T)\lfre{}\mathscr{J}$
the composition of $\mathcal{L}_{\T}(\cdot{},k,s)$ with the morphism induced by the projection $p^{-} : \T\twoheadrightarrow{}\T^{-}$.
The aim of this section is to prove 
Theorem $\ref{mainderalg}$ below, which gives  a simple expression for the 
derivative of $\mathcal{L}_{\T}(\cdot{},k,s)$. 


Denote by $\mathrm{rec}_{p} : \Q_{p}^{\ast}\fre{}G_{p}^{\mathrm{ab}}:=G_{\Q_{p}}^{\mathrm{ab}}$
the local reciprocity map,
normalised so that $\mathrm{rec}_{p}(p^{-1})$ is an arithmetic Frobenius. It induces an isomorphism 
$\mathrm{rec}_{p} : \Q_{p}^{\ast}\widehat{\otimes}\Q_{p}\cong{}G_{p}^{\mathrm{ab}}\widehat{\otimes}\Q_{p}$,
where $G\widehat{\otimes}\Q_{p}:=\big(\inlim{}_{n\in{}\N}{}G/p^{n}G\big)\otimes_{\Z_{p}}\Q_{p}$
for every abelian group $G$.
This yields an isomorphism of $\Q_{p}$-vector spaces
$$H^{1}(\Q_{p},\Q_{p})=\mathrm{Hom}_{\mathrm{cont}}(G_{p}^{\mathrm{ab}}\widehat{\otimes}\Q_{p},\Q_{p})
\cong\mathrm{Hom}_{\mathrm{cont}}(\Q_{p}^{\ast}\widehat{\otimes}\Q_{p},\Q_{p})
=\mathrm{Hom}_{\mathrm{cont}}(\Q_{p}^{\ast},\Q_{p}),$$
which we consider as an equality.
For
 every $\mathfrak{Z}=(\mathfrak{Z}_{n})\in{}H^{1}_{\mathrm{Iw}}(\Q_{p,\infty},\T^{-})$,
the class $\mathfrak{Z}_{0,\psi}\in{}H^{1}(\Q_{p},\Q_{p})$ is then identified with a continuous morphism on 
$\Q_{p}^{\ast}$ (see Section $\ref{ochiaisec}$ for the notations).
Let   $$\exp^{\ast}_{A} : H^{1}(\Q_{p},V_{p}(A))
\fre{}\mathrm{Fil}^{0}D_{\mathrm{dR}}(V_{p}(A))\cong\Q_{p}$$
be the Bloch--Kato dual exponential map (cf.   $(\ref{eq:tangent and fil})$).  
Finally, set  $e(1):=(1+p)\widehat{\otimes}\log_{p}(1+p)^{-1}\in{}\Z_{p}^{\ast}\widehat{\otimes}\Q_{p}$.

\begin{theo}\label{mainderalg} $1.$ Let $\mathfrak{Z}=(\mathfrak{Z}_{n})\in{}
H^{1}_{\mathrm{Iw}}(\Q_{p,\infty},\mathbb{T}^{-})$ and let
$\mathfrak{z}:=\mathfrak{Z}_{0,\psi}\in{}\mathrm{Hom}_{\mathrm{cont}}(\Q_{p}^{\ast},\Q_{p})$. Then
\[
\lri{1-\frac{1}{p}}{}\mathcal{L}_{\T}(\mathfrak{Z},k,s)\equiv{}\mathfrak{z}\big(p^{-1}\big)\cdot{}(s-1)-
\frac{1}{2}\mathscr{L}_p(A)\cdot{}\mathfrak{z}\big(e(1)\big)\cdot{}(k-2)\ \ \big(\mathrm{mod}\ \mathscr{J}^{2}\big).
\]\par

$2.$ Let $\mathfrak{Z}=(\mathfrak{Z}_{n})\in{}H^{1}_{\mathrm{Iw}}(\Q_{p,\infty},\mathbb{T})$
and let $\mathfrak{z}:=\mathfrak{Z}_{0,\psi}\in{}H^{1}(\Q_p,V_p(A))$.
Then
\[
     \lri{1-\frac{1}{p}}{}\mathcal{L}_{\T}(\mathfrak{Z},k,s)\equiv{}
     \mathscr{L}_p(A)\cdot{}\exp_{A}^{\ast}(\mathfrak{z})\cdot{}(s-k/2)\ \ 
     \big(\mathrm{mod}\ \mathscr{J}^{2}\big).
\]
\end{theo}

The proof  of Theorem $\ref{mainderalg}$ is given in  Section $\ref{proofmainderalg}$. We  consider separately the partial derivatives
of  $\mathcal{L}_{\T}(\cdot{},k,s)$ with respect to the cyclotomic variable $s$ and
the weight variable $k$. In order to compute the derivative in the cyclotomic direction, we  make use of the work of
Wiles \cite{Wiles-rec} and Coleman \cite{Col}.
To compute the derivative in the weight direction, we prove the existence of   an \emph{improved big dual exponential},
and then invoke a formula of Greenberg--Stevens which relates the  derivative of the $p$-th Fourier coefficient of $f_{\infty}$
to the $\mathscr{L}$-invariant $\mathscr{L}_{p}(A)$
\cite{G-S}.

\subsection{The Coleman map}\label{colprsec}
In this section we first recall, following \cite{Rubs}, the definition of the cyclotomic  big dual exponential 
$\mathcal{L}_{A}:=\mathcal{L}_{T_{p}(A)}$ for the $p$-adic Tate module of $A$,
called the \emph{Coleman map}. In our exceptional zero situation, it is a morphism of $\Ic$-algebras 
$$\mathcal{L}_{A} : H^{1}_{\mathrm{Iw}}(\Q_{p,\infty},T_{p}(A))\lfre{}I,$$
factoring  through the Iwasawa cohomology of $\Z_{p}=T_{p}(A)^{-}$, 
where $I=I_{\mathrm{cyc}}$ is the 
augmentation ideal of $\Ic$. 
We then prove in Proposition $\ref{dercol}$ a  simple formula for its derivative at the augmentation ideal.
While versions of Proposition $\ref{dercol}$
already appear in the literature (e.g. it follows from Proposition A.3.1
of \cite{L-V-Z}), we give here a proof in our setting for the convenience of the reader.

\subsubsection{Definition of $\mathcal{L}_{A}$}\label{fix varsigma} For every $n\in{}\N\cup{}\{\infty\}$, 
identify  $G_{n}:=\mathrm{Gal}(\Q_{n}/\Q)$
with the Galois group of $\Q_{p,n}/\Q_{p}$ via the  unique prime of $\Q_{\infty}$ dividing $p$.
Then $\Ic=\Z_{p}\llbracket{}G_{\infty}\rrbracket$ is identified with the Iwasawa algebra of the 
cyclotomic $\Z_{p}$-extension  $\Q_{p,\infty}/\Q_{p}$. 
Let $\Z_{p,n}$ and $\mathfrak{m}_{n}$ be the  ring of integers of $\Q_{p,n}$
and its maximal ideal respectively,
and let $N_{m,n} : \Q_{p,m}^{\ast}\fre{}\Q_{p,n}^{\ast}$
be the norm map, for  $m\geq{}n$.

Fix a generator $\zeta_{\infty}=(\zeta_{p^{n}})_{n\in{}\N}\in{}\Z_{p}(1)$.
As in \cite[Appendix]{Rubs}, define for every $n\in{}\N$:
\[
        x_{n}:=p+\mathrm{Trace}_{\Q_{p}(\mu_{p^{n+1}})/\Q_{p,n}}\lri{\sum_{k=0}^{n}\frac{\zeta_{p^{n+1-k}}-1}{p^{k}}}\in{}\Q_{p,n}.
\]
A simple computation shows that these elements are compatible with respect to the trace maps.
The following key lemma is due to Coleman (cf. Theorem 24 of \cite{Col}).

\begin{lemma}\label{co2} There exists a unique principal unit $g(X)\in{}1+(p,X)\cdot{}\Z_p
\llbracket{}X\rrbracket$ such that:
\begin{itemize}
\item[$1.$] $\log_p(g(0))=p$;
\item[$2.$] $\mathfrak{C}_{n}:=g(\zeta_{p^{n+1}}-1)\in{}1+\mathfrak{m}_{n}$ and $\log_p\lri{\mathfrak{C}_{n}}=x_{n}$ for every $n\in{}\N$;
\item[$3.$] $N_{m,n}(\mathfrak{C}_{m})=\mathfrak{C}_{n}$ for every $m\geq{}n\geq{}0$.
\end{itemize}
\end{lemma}
\begin{proof} See \cite[Appendix]{Rubs} or \cite[Appendix D]{Rub}.
\end{proof}

Identify $H^{1}(\Q_{p,n},\Z_p(1))=\Q_{p,n}^{\ast}\widehat{\otimes}\Z_p$ by Kummer theory.  The preceding lemma allows us to define
\[
       \mathfrak{C}:=\lri{\mathfrak{C}_{n}\widehat{\otimes{}}1}_{n\in\N}\in{}H^{1}_{\mathrm{Iw}}(\Q_{p,\infty},\Z_p(1)).
\]
Fix a topological generator $\sigma_{0}\in{}G_{\infty}$, and  write $\varsigma:=\sigma_{0}-1\in{}I$
for the corresponding generator of  $I\subset{}\Ic$.

\begin{lemma}\label{der col class} 
There exists a unique $\mathfrak{C}^{\prime}:=\mathfrak{C}^{\prime}_{\varsigma}\in{}H^{1}_{\mathrm{Iw}}(\Q_{p,\infty},\Z_p(1))$ 
such that $\mathfrak{C}=\varsigma\cdot{}\mathfrak{C}^{\prime}$.
\end{lemma}
\begin{proof} The corestriction map induces an injective map:
$H^{1}_{\mathrm{Iw}}(\Q_{p,\infty},\Z_{p}(1))/\varsigma\hookrightarrow{}H^{1}(\Q_{p},\Z_{p}(1))$,
and the $\varsigma$-torsion submodule  $H^{1}_{\mathrm{Iw}}(\Q_{\infty,p},\Z_{p}(1))[\varsigma]$ is trivial
(being a quotient of $H^{0}(\Q_{p},\Z_{p}(1))$).
It is then sufficient to prove that the principal unit  $\mathfrak{C}_{0}$ is equal to $1$. 
Note that $x_{0}=0$,
as $\mathrm{Trace}_{\Q_{p}(\mu_{p})/\Q_{p}}(\zeta_{p}-1)=-p$. By Lemma $\ref{co2}(2)$ 
this implies  $\log_{p}(\mathfrak{C}_{0})=0$,
i.e. $\mathfrak{C}_{0}=1$ (as $p\not=2$).
\end{proof}

By local Tate duality, there is a natural morphism of $\Ic$-modules 
\[
                 \dia{-,-}_{\infty} : H^{1}_{\mathrm{Iw}}(\Q_{p,\infty},\Z_{p})\otimes_{\Ic}H^{1}_{\mathrm{Iw}}(\Q_{p,\infty},\Z_{p}(1))^{\iota}
                 \lfre{}\Ic.
\]
Here $\iota$ is Iwasawa's main involution on $\Ic$, i.e. the isomorphism of $\Z_{p}$-algebras which acts 
as inversion on $G_{\infty}$,
and $H^{1}_{\mathrm{Iw}}(\Q_{p,\infty},\Z_{p}(1))^{\iota}$
denotes the  $\Z_{p}$-module $H^{1}_{\mathrm{Iw}}(\Q_{p,\infty},\Z_{p}(1))$, with $\Ic$-action 
obtained by twisting the original action by $\iota$. (See, e.g. Section 2.1.5 of \cite{PR} for the definition of $\dia{-,-}_{\infty}$.) 
Define
\[
                \mathcal{L}_{A}:=-\dia{\cdot{},\mathfrak{C}}_{\infty} :  H^{1}_{\mathrm{Iw}}(\Q_{p,\infty},\Z_{p})\lfre{}I.
\]
The fact that $\mathcal{L}_{A}$ takes values in the augmentation ideal follows from Lemma $\ref{der col class}$
(as $\iota(\varsigma)=-\sigma_{0}^{-1}\varsigma\in{}I$).
 The following proposition is a version of the Coleman--Wiles explicit reciprocity law
 \cite{Wiles-rec}, \cite{Col};
we refer to \cite[Appendix]{Rubs} (or \cite[Section 13.2]{PhD}) for a proof in our setting.

\begin{proposition}\label{coleman map}  
For every $z=(z_{n})\in{}H^{1}_{\mathrm{Iw}}(\Q_{p,\infty},\Z_p)$ 
and every non-trivial character $\chi$ of $G_{n}$ 
\[
     \chi\big(\mathcal{L}_{A}(z)\big)=\tau(\chi)\sum_{\sigma\in{}G_{n}} \chi(\sigma)^{-1}
     \cdot{}\mathrm{exp}_{n}^{\ast}\big(z_{n}^{\sigma}\big),
\]
where $\exp_{n}^{\ast} : H^{1}(\Q_{p,n},\Q_{p})\fre{}D_{\mathrm{dR},n}(\Q_{p})=\Q_{p,n}$ is the Bloch--Kato dual exponential map. 
\end{proposition}

With a slight abuse of notation, denote again by
$\mathcal{L}_{A} : H^{1}_{\mathrm{Iw}}(\Q_{p,\infty},T_{p}(A))\lfre{}I$
the composition of $\mathcal{L}_{A}$ with the map induced by the projection $p^{-} : T_{p}(A)\twoheadrightarrow{}\Z_{p}$ 
(see $(\ref{eq:tate exact})$). 
Note the following corollary.

\begin{cor}\label{Ochiai vs Coleman} Let $\T^{?}$ denote either $\T^{-}$ or $\T$. For every $\mathfrak{Z}\in{}H^{1}_{\mathrm{Iw}}(\Q_{p,\infty},\T^{?})$: 
\[
               \psi\big(\mathcal{L}_{\T}(\mathfrak{Z})\big)=\mathcal{L}_{A}\big(\mathfrak{Z}_{\psi}\big),
\]
where $\mathfrak{Z}_{\psi}\in{}H^{1}_{\mathrm{Iw}}(\Q_{p,\infty},T_{p}(A)^{?})$ is the image of 
$\mathfrak{Z}$ under the morphism induced by
 $\T^{?}\twoheadrightarrow{}\T_{\psi}^{?}\cong{}T_{p}(A)^{?}$.
\end{cor}
\begin{proof} As $\psi(\mathbf{a}_{p})=1$, this follows from  $(\ref{eq:norm mho})$ and the interpolation properties of 
$\psi\circ{}\mathcal{L}_{\T}$ and $\mathcal{L}_{A}$.
\end{proof}

\subsubsection{The derivative of $\mathcal{L}_{A}$} If $M$ denotes either $T_{p}(A)$ or $\Z_{p}$, 
define the \emph{derivative of $\mathcal{L}_{A}$}:
\[
            \mathcal{L}_{A}^{\prime} : H^{1}_{\mathrm{Iw}}(\Q_{p,\infty},M)\lfre{}I/I^{2}
\]
as the composition of $\mathcal{L}_{A}$ with the projection $\{\cdot{}\} : I\twoheadrightarrow{}I/I^{2}$.
Set $\log_{p}(\varsigma):=\log_{p}\big(\chi_{\mathrm{cyc}}(\sigma_{0})\big)$ and 
$$l_{\varsigma}:=\log_{p}(\varsigma)\cdot{}\big(1-p^{-1}\big)\in{}\Z_{p}^{\ast},$$ where 
$\varsigma=\sigma_{0}-1$  is our fixed generator of $I$.
As in Section $\ref{derbigdualsec}$, the cohomology group  $H^{1}(\Q_{p},\Z_{p})$
is identified with $\mathrm{Hom}_{\mathrm{cont}}(\Q_{p}^{\ast},\Z_{p})$
via the local reciprocity map.

\begin{proposition}\label{dercol} Let $z=(z_{n})\in{}H^{1}_{\mathrm{Iw}}(\Q_{p,\infty},\Z_p)$.
Then
$l_{\varsigma}\cdot{}\mathcal{L}^{\prime}_{A}(z)=z_{0}\lri{p^{-1}}\left\{\varsigma\right\}$.
\end{proposition}

Before giving the proof of  Proposition $\ref{dercol}$, we deduce the following corollary. 

\begin{cor}\label{derivative coleman} Let $z=(z_{n})\in{}H^{1}_{\mathrm{Iw}}(\Q_{p,\infty},T_{p}(A))$. 
Then 
\[
     l_{\varsigma}\cdot{}\mathcal{L}_{A}^{\prime}(z)=\mathscr{L}_{p}(A)\cdot{}\exp_{A}^{\ast}(z_{0}){}
\left\{\varsigma\right\}.
\]
In particular $\mathcal{L}_{A}(z)\in{}I^{2}$ if and only if $z_{0}\in{}H^{1}_{f}(\Q_{p},V_{p}(A))\cong{}
A(\Q_{p})\widehat{\otimes}\Q_{p}$.
\end{cor}
\begin{proof} Consider the exact sequence\vspace{-2mm}
$$H^{1}(\Q_{p},\Q_{p}(1))\lfre{i^{+}} H^{1}(\Q_{p},V_{p}(A))\lfre{p^{-}} \mathrm{Hom}_{\mathrm{cont}}(\Q_{p}^{\ast},\Q_{p})\lfre{\delta}
H^{2}(\Q_{p},\Q_{p}(1))\stackrel{\mathrm{inv}_{p}}{\cong}\Q_{p}$$
arising from the  exact sequence $(\ref{eq:tate exact})$, where $\mathrm{inv}_{p}$
is the invariant map of local class field theory. 
A direct computation shows that $\delta(\cdot)=\mathrm{inv}_{p}\big(\cdot\cup{}q_{A}\widehat{\otimes}1\big)$,
where $\cup : H^{1}(\Q_{p},\Q_{p})\times{}H^{1}(\Q_{p},\Q_{p}(1))\fre{}H^{2}(\Q_{p},\Q_{p}(1))$
is the natural cup-product pairing and we identify as above $H^{1}(\Q_{p},\Q_{p}(1))=\Q_{p}^{\ast}\widehat{\otimes}\Q_{p}$. 
It then follows by local class field theory \cite{Ser} that  $\delta(\phi)=-\phi(q_{A})$ for every 
$\phi\in{}\mathrm{Hom}_{\mathrm{cont}}(\Q_{p}^{\ast},\Q_{p})$, so that
the image of $p^{-}$ is equal to the space of  morphisms $\phi$ such that $\phi(q_{A})=0$.
As $\log_{p}$ and $\mathrm{ord}_{p}$ form a $\Q_{p}$-basis of  
$\mathrm{Hom}_{\mathrm{cont}}(\Q_{p}^{\ast},\Q_{p})$, this implies 
\[
            \mathrm{Im}\lri{p^{-}}=\Q_{p}\cdot{}\log_{q_{A}},
\]
where $\log_{q_{A}}=\log_{p}-\mathscr{L}_{p}(A)\cdot{}\mathrm{ord}_{p}$ is the branch of the $p$-adic logarithm 
which vanishes  on $q_{A}\in{}p\Z_{p}$. 

Let  $z=(z_{n})\in{}H^{1}_{\mathrm{Iw}}(\Q_{p},T_{p}(A))$, and write 
$p^{-}(z_{0})=\alpha\cdot{}\log_{q_{A}}\in{}\mathrm{Hom}_{\mathrm{cont}}(\Q_{p}^{\ast},\Q_{p})$,
for some  $\alpha\in{}\Q_{p}$.  Then  $\exp_{A}^{\ast}(z_{0})=\exp^{\ast}(\alpha{}\cdot{}\log_{q_{A}})=\alpha$,
where 
$\exp^{\ast}=\exp^{\ast}_{0}$ is the Bloch--Kato dual exponential for $\Q_{p}$. 
Indeed, by its very  definition 
(see Chapter II of \cite{Katiw}), $\exp^{\ast}(\log_{p})=1$ and $\exp^{\ast}(\mathrm{ord}_{p})=0$. According  Proposition $\ref{dercol}$
\[
                    l_{\varsigma}\cdot{}\mathcal{L}_{A}^{\prime}(z)=\alpha{}\log_{q_{A}}(p^{-1})\cdot{}\{\varsigma\}
                    =\mathscr{L}_{p}(A)\cdot{}\exp_{A}^{\ast}(z_{0})\cdot{}\{\varsigma\}.
\]
The last assertion in the statement  follows from the non-vanishing of the $\mathscr{L}$-invariant 
\cite{M-man} and  the fact that the finite part  $H^{1}_{f}(\Q_{p},V_{p}(A))\cong{}A(\Q_{p})\widehat{\otimes}\Q_{p}$
\cite{B-K}
of the local cohomology group $H^{1}(\Q_{p},V_{p}(A))$ is the kernel of the dual exponential.
Indeed, the preceding discussion shows that
an element of $H^{1}(\Q_{p},V_{p}(A))$ belongs to the kernel of $\exp^{\ast}_{A}$ if and only if 
it is in the image of $i^{+} : H^{1}(\Q_{p},\Q_{p}(1))\fre{}H^{1}(\Q_{p},V_{p}(A))$,
and the latter equals $H^{1}_{f}(\Q_{p},V_{p}(A))$,
as follows easily from Kummer theory and the surjectivity of the  Tate parametrisation $(\ref{eq:Tate iso})$.
\end{proof}

\begin{proof}[Proof of Proposition $\ref{dercol}$] 
For every $n\in{}\N$, let $\pi_{n}:=\mathrm{Norm}_{\Q_{p}(\mu_{p^{n+1}})/\Q_{p,n}}(\zeta_{p^{n+1}}-1)$;
this is a uniformiser of $\Z_{p,n}$. 
Since $\Q_{p,n}^{\ast}$ has no non-trivial $p$-torsion, one has a decomposition 
\[
             H^{1}(\Q_{p,n},\Z_{p}(1))=\Q_{p,n}^{\ast}\widehat{\otimes}\Z_{p}=\widehat{\pi_{n}}\oplus{}1+\mathfrak{m}_{n},
\]
where  $\widehat{\pi_{n}}$ is the $p$-adic completion of $\pi_{n}^{\Z}$. 
Given $\alpha_{n}\in{}H^{1}(\Q_{p,n},\Z_{p}(1))$, let $\kappa_{n}(\alpha_{n})\in{}1+\mathfrak{m}_{n}$
be its projection to principal units, and $\mathrm{ord}_{n}(\alpha_{n})\in{}\Z_{p}$  its $\pi_{n}$-adic  valuation.  
Since $N_{m,n}(\pi_{m})=\pi_{n}$ for every integers $m\geq{}n$,
if $\alpha=(\alpha_{n})\in{}H^{1}_{\mathrm{Iw}}(\Q_{p,\infty},\Z_{p}(1))$
then $\mathrm{ord}(\alpha):=\mathrm{ord}_{n}(\alpha_{n})$ is independent of $n\in{}\N$, and  $\kappa(\alpha):=(\kappa_{n}(\alpha_{n}))_{n\in{}\N}$
is a compatible sequence with respect to the norm maps.
One can  then define maps
\[
          \mathrm{ord} : H^{1}_{\mathrm{Iw}}(\Q_{p,\infty},\Z_{p}(1))\fre{}\Z_{p};\ \ \kappa : H^{1}_{\mathrm{Iw}}(\Q_{p,\infty},\Z_{p}(1))
          \fre{}U_{\infty}^{1},
\]
where $U_{\infty}^{1}$ denotes the inverse limit of the groups  $1+\mathfrak{m}_{n}$. 
Write $\pi_{\infty}:=(\pi_{n})\in{}H^{1}_{\mathrm{Iw}}(\Q_{p,\infty},\Z_{p}(1))$.
By construction
$\alpha=\pi_{\infty}^{\mathrm{ord}(\alpha)}+\kappa(\alpha)$
for every $\alpha=(\alpha_{n})\in{}H^{1}_{\mathrm{Iw}}(\Q_{p,\infty},\Z_{p}(1))$.
Moreover, one has 
\begin{equation}\label{eq:buster 1}
             \alpha_{0}=p^{\mathrm{ord}(\alpha)}\in{}H^{1}(\Q_{p},\Z_{p}(1)).
\end{equation}
Indeed, local class field theory tells us that the image of the injective map 
$H^{1}_{\mathrm{Iw}}(\Q_{p,\infty},\Z_{p}(1))/\varsigma\hookrightarrow{}H^{1}(\Q_{p},\Z_{p}(1))$
induced by the corestriction 
equals $\widehat{p}=\widehat{\pi_{0}}$. 
Then $U_{\infty}^{1}\subset\varsigma\cdot{}H^{1}_{\mathrm{Iw}}(\Q_{p,\infty},\Z_{p}(1))$
and equation $(\ref{eq:buster 1})$
follows.

Let us now consider the element $\mathfrak{C}^{\prime}=\mathfrak{C}^{\prime}_{\varsigma}$
appearing in Lemma $\ref{der col class}$. For every $z=(z_{n})\in{}H^{1}_{\mathrm{Iw}}(\Q_{p,\infty},\Z_{p})$
\begin{equation}\label{eq:serveder}
         \mathcal{L}_{A}^{\prime}(z)=z_{0}(p^{-1})\cdot{}\mathrm{ord}(\mathfrak{C}^{\prime})\cdot{}\{\varsigma\}.
\end{equation}
Indeed, let $\dia{-,-} : H^{1}(\Q_{p},\Z_{p})\times{}H^{1}(\Q_{p},\Z_{p}(1))\fre{}\Z_{p}$ 
be the local Tate pairing. Then $\dia{z_{0},\mathfrak{C}^{\prime}_{0}}=\varepsilon\lri{\dia{z,\mathfrak{C}^{\prime}}_{\infty}},$
where  $\varepsilon$ is the augmentation map and we write $\mathfrak{C}^{\prime}=
(\mathfrak{C}^{\prime}_{n})$.
This implies
\begin{equation}\label{eq:auxxx}
    \mathcal{L}_{A}^{\prime}(z)=-\{\dia{z,\varsigma\cdot{}\mathfrak{C}^{\prime}}_{\infty}\}=\dia{z_{0},\mathfrak{C}^{\prime}_{0}}\cdot{}\{\varsigma\}.
\end{equation}
(Note that $\iota(\varsigma)\equiv{}-\varsigma\ \mathrm{mod}\ I^{2}$.)
Since $\dia{z_{0},x}=z_{0}(x^{-1})$
for every $x\in{}\Q_{p}^{\ast}\widehat{\otimes}\Z_{p}$ by local class field theory \cite{Ser},
equation $(\ref{eq:serveder})$ follows  by combining equations $(\ref{eq:auxxx})$ and $(\ref{eq:buster 1})$.

Thanks to $(\ref{eq:serveder})$, the proposition will follow once we prove  the claim
\begin{equation}\label{eq:finclaimcolder} 
                     \mathrm{ord}(\mathfrak{C}^{\prime})=l_{\varsigma}^{-1}\in{}\Z_{p}^{\ast}.
\end{equation}  
Write $V_{\infty}$ for the inverse limit of the  groups $\Z_{p}[\zeta_{p^{m+1}}]^{\ast}$, for $m\in{}\N$.
According to Theorem A of \cite{Col}, for every $v=(v_{n})\in{}V_{\infty}$ there exists a unique power series 
$f_{v}(T)\in{}\Z_{p}\llbracket{}T\rrbracket^{\ast}$
such that  $f_{v}(\zeta_{p^{n+1}}-1)=v_{n}$
for every $n\in{}\N$. 
The association $v\mapsto{}f_{v}(T)$ is a morphism of $\Z_{p}\llbracket{}\mathrm{Gal}(\Q_{p}(\mu_{p^{\infty}})/\Q_{p})\rrbracket$-modules (see \cite{Col} for details).
Note that, with the notations of Lemma $\ref{co2}$,  $g(T)=f_{\mathfrak{C}}(T)$.
As $\mathfrak{C}=\varsigma\cdot{}\mathfrak{C}^{\prime}$
and $\mathfrak{C}^{\prime}=\kappa(\mathfrak{C}^{\prime})+\pi_{\infty}^{\mathrm{ord}(\mathfrak{C}^{\prime})}$,
one then finds
\[
               g(T)=
                     \frac{f_{\kappa(\mathfrak{C}^{\prime})}\big((1+T)^{\chi_{\mathrm{cyc}}(\sigma_{0})}-1\big)}{f_{\kappa(\mathfrak{C}^{\prime})}(T)}
                     \cdot{}\lri{\prod_{\mu\in{}\mu_{p-1}}\frac{(1+T)^{\mu\cdot{}\chi_{\mathrm{cyc}}(\sigma_{0})}-1}
                     {(1+T)^{\mu}-1}}^{\mathrm{ord}(\mathfrak{C}^{\prime})}.
\]
Evaluating this equality at $T=0$ and then applying the $p$-adic  logarithm, we easily obtain
\[
                \log_{p}\big(g(0)\big)=(p-1)\cdot{}{}\mathrm{ord}(\mathfrak{C}^{\prime}){}\cdot{}\log_{p}(\chi_{\mathrm{cyc}}(\sigma_{0}))=p{}\cdot{}
                \mathrm{ord}(\mathfrak{C}^{\prime}){}\cdot{}l_{\varsigma}.
\]
Since $\log_{p}\big(g(0)\big)=p$ by Lemma $\ref{co2}(1)$, the claim $(\ref{eq:finclaimcolder})$ follows.
\end{proof}

\subsection{The improved big dual exponential} The aim of this section is to construct 
an \emph{improved big dual exponential} $\mathcal{L}_{\T}^{\ast} : H^{1}(\Q_{p},\T^{-})\fre{}R[1/p]$.
To do this we follow the techniques of \cite[Section 5]{Ochiaicol}.

\begin{proposition}\label{improved big dual} There exists a unique morphism of $R$-modules
\[
            \mathcal{L}_{\T}^{\ast}=\mathcal{L}_{\T,\mho}^{\ast} : H^{1}(\Q_{p},\T^{-})\longrightarrow{}R\otimes_{\Z_{p}}\Q_{p}
\]
such that: for every $\mathfrak{Z}\in{}H^{1}(\Q_{p},\T^{-})$ and every   $\nu\in{}\xari(R)$ 
\[
              \nu\big(\mathcal{L}_{\T}^{\ast}(\mathfrak{Z})\big)=\lri{1-\frac{\nu(\mathbf{a}_{p})}{p}}^{-1}
              \big\langle\exp^{\ast}(\mathfrak{Z}_{\nu}),\mho_{\nu}(1)
              \big\rangle_{\mathrm{dR}},
\]
where $\exp^{\ast} : H^{1}(\Q_{p},V_{\nu}^{-})\fre{}D_{\mathrm{dR}}(V_{\nu}^{-})=\mathrm{Fil}^{0}
D_{\mathrm{dR}}(V_{\nu})$ is the Bloch--Kato dual 
exponential map.
\end{proposition}

Before giving the proof of  Proposition $\ref{improved big dual}$, we note the following corollary (cf. Section $\ref{improved integral}$).

\begin{cor}\label{improved L} Let $\varepsilon : \overline{R}\twoheadrightarrow{}R$ be the augmentation map, and let 
 $\mathfrak{Z}=(\mathfrak{Z}_{n})\in{}H^{1}_{\mathrm{Iw}}(\Q_{p,\infty},\T^{-})$. Then
\[
               \varepsilon\big(\mathcal{L}_{\T}(\mathfrak{Z})\big)=\lri{1-\mathbf{a}_{p}^{-1}}\cdot{}\mathcal{L}_{\T}^{\ast}(\mathfrak{Z}_{0}).
\]
\end{cor}
\begin{proof} Taking $\chi$ as the trivial character of $G_{\infty}$ in Proposition  $\ref{modifiedochiai}$, one has
\[
         \nu\circ{}\varepsilon\big(\mathcal{L}_{\T}(\mathfrak{Z})\big)=
         \lri{1-\nu(\mathbf{a}_{p})^{-1}}\cdot{}\nu\big(\mathcal{L}_{\T}^{\ast}(\mathfrak{Z}_{0})\big),
\]
for every weight-two arithmetic point $\nu\in{}\xari(R)$. Since such points (or better their kernels) form a dense subset of $\mathrm{Spec}(R)$,
the corollary follows. 
\end{proof}

\begin{proof}[Proof of Proposition $\ref{improved big dual}$] Let $K$ be a complete subfield of $\widehat{\Q}_{p}^{\mathrm{un}}$ and let $V$ be a $p$-adic representation of $G_{K}$.
Denote by 
$D_{\mathrm{dR},K}(V):=H^{0}(K,V\otimes_{\Q_{p}}B_{\mathrm{dR}})$, and 
by  $\exp : D_{\mathrm{dR},K}(V)\fre{}H^{1}(K,V)$  the Bloch--Kato exponential map \cite{B-K}.
For $V=\Q_{p}(1)$, it is described by the composition\vspace{-1mm}
\[
             \exp_{p} : D_{\mathrm{dR},K}(\Q_{p}(1))=K\lfre{}
                K^{\ast}\widehat{\otimes}\Q_{p}=H^{1}(K,\Q_{p}(1)),
\]
where the first equality  refers to the canonical identification $D_{\mathrm{dR},K}(\Q_{p}(1))=K\cdot{}\zeta_{\mathrm{dR}}\cong{}K$
(see Section $\ref{beka0}$),
the  arrow is given by the usual $p$-adic exponential and the last equality is the Kummer isomorphism. 
As $K$ is unramified, $\exp_{p}$ maps the ring of integers of $K$ into $\frac{1}{p}H^{1}(K,\Z_{p}(1))\subset{}H^{1}(K,\Q_{p}(1))$.

Set $G_{p}:=G_{\Q_{p}}$, $I_{p}:=I_{\Q_{p}}$ and $G_{p}^{\mathrm{un}}:=G_{p}/I_{p}$. 
With the notations of Section $\ref{expbk}$, consider the morphism of
$R[G_{p}^{\mathrm{un}}]$-modules\vspace{-1mm} 
\begin{equation}\label{eq:an arrow}
         \exp_{p}\widehat{\otimes}\mathrm{id} :  \widehat{\Z}_{p}^{\mathrm{un}}\widehat{\otimes}_{\Z_{p}}\check{\T}^{+}
           \fre{}
           \Big(H^{1}(I_{p},\Z_{p}(1))\widehat{\otimes}_{\Z_{p}}\check{\T}^{+}\Big)\otimes_{\Z_{p}}\Q_{p}
           =H^{1}\big(I_{p},\check{\T}^{+}(1)\big)\otimes_{\Z_{p}}\Q_{p}\vspace{-1mm}
\end{equation}
(recall that  $\check{\T}^{+}$ is unramified).  
As $H^{0}(I_{p},\check{\T}^{+}(1))=0$, restriction gives
an isomorphism between $H^{1}(\Q_{p},\check{\T}^{+}(1))$ and 
$H^{0}\big(G_{p}^{\mathrm{un}},H^{1}\big(I_{p},\check{\T}^{+}(1)\big)\big)$.
Taking $G_{p}^{\mathrm{un}}$-invariants in $(\ref{eq:an arrow})$  then yields a morphism of $R$-modules
\[
                 \mathrm{exp}_{\T} : \mathcal{D}\lfre{}H^{1}(\Q_{p},\check{\T}^{+}(1))\otimes_{\Z_{p}}\Q_{p}.
\]
We claim that for every arithmetic point  $\nu\in{}\xari(R)$ 
\begin{equation}\label{eq:claim improved exp}
               \nu_{\ast}\big(\mathrm{exp}_{\T}(\mho)\big)=\exp\big(\mho_{\nu}(1)\big),
\end{equation}
where  $\nu_{\ast} : H^{1}(\Q_{p},\check{\T}^{+}(1))\fre{}H^{1}(\Q_{p},\check{V}_{\nu}^{+}(1))$ is the  morphism induced 
by $\check{\T}^{+}\twoheadrightarrow{}\check{\T}^{+}_{\nu}\subset{}\check{V}_{\nu}^{+}$, and 
$\mathrm{exp}$ is the exponential on $D_{\mathrm{dR}}(\check{V}_{\nu}^{+}(1))$.
As above,
the restriction map gives an isomorphism between $H^{1}(\Q_{p},\check{V}_{\nu}^{+}(1))$
and the $G_{p}^{\mathrm{un}}$-invariants of $H^{1}(I_{p},\check{V}_{\nu}^{+}(1))$.
It follows that the exponential 
$\exp : D_{\mathrm{dR}}(\check{V}^{+}_{\nu}(1))\fre{}H^{1}(\Q_{p},\check{V}^{+}_{\nu}(1))$ 
is identified with the restriction of
\[
\exp_{p}\otimes{}\mathrm{id} :              \widehat{\Q}_{p}^{\mathrm{un}}\otimes_{\Q_{p}}\check{V}_{\nu}^{+}\lfre{}
              H^{1}(I_{p},\Q_{p}(1))\otimes_{\Q_{p}}\check{V}_{\nu}^{+}
              =H^{1}(I_{p},\check{V}_{\nu}^{+}(1))
\]
to the $G_{p}^{\mathrm{un}}$-invariants.
Equation $(\ref{eq:claim improved exp})$ then follows from the  definitions of $\mathrm{exp}_{\T}$ and $\mho_{\nu}(1)$.

Let $\dia{-,-}_{R} : H^{1}(\Q_{p},\T^{-})\otimes_{R}H^{1}(\Q_{p},\check{\T}^{+}(1))\fre{}R$ 
be the $R$-adic local  Tate pairing and define
\[
                      \mathrm{exp}_{\T}^{\ast}=\mathrm{exp}_{\T,\mho}^{\ast}:=
                      \big\langle\cdot{},\mathrm{exp}_{\T}(\mho)\big\rangle_{R} : H^{1}(\Q_{p},\T^{-})\longrightarrow{}R
                      \otimes_{\Z_{p}}\Q_{p}.
\]   
By  $(\ref{eq:claim improved exp})$ one obtains: for every $\mathfrak{Z}\in{}H^{1}(\Q_{p},\T^{-})$ and every $\nu\in{}\xari(R)$  
\begin{equation}\label{eq:final improved}
               \nu\big(\mathrm{exp}_{\T}^{\ast}(\mathfrak{Z})\big)=
               \dia{\mathfrak{Z}_{\nu},\nu_{\ast}\big(\exp_{\T}(\mho)\big)}_{\nu}
         =
               \big\langle\mathfrak{Z}_{\nu},\exp\big(\mho_{\nu}(1)\big)\big\rangle_{\nu}
               =\big\langle\exp^{\ast}(\mathfrak{Z}_{\nu}),\mho_{\nu}(1)\big\rangle_{\mathrm{dR}}.
\end{equation}
Here $\dia{-,-}_{\nu} : H^{1}(\Q_{p},V_{\nu}^{-})\times{}H^{1}(\Q_{p},\check{V}^{+}_{\nu}(1))\fre{}K_{\nu}$
is the local Tate pairing and 
$\exp^{\ast}$ is the Bloch--Kato dual exponential map on $H^{1}(\Q_{p},V_{\nu}^{-})$;
the first equality follows from the functoriality of the 
local Tate duality,
while the  last  equality  is \cite[Chapter II, Theorem 1.4.1]{Katiw}.
Define
\[
                \mathcal{L}_{\T}^{\ast}:=\lri{1-\frac{\mathbf{a}_{p}}{p}}^{-1}\mathrm{exp}_{\T}^{\ast} : H^{1}(\Q_{p},\T^{-})\lfre{}R\otimes_{\Z_{p}}\Q_{p}.
\]
According to  $(\ref{eq:final improved})$, the morphism $\mathcal{L}_{\T}^{\ast}$
satisfies the desired interpolation property, which characterises it uniquely (as the kernels of the arithmetic points are 
dense in $\mathrm{Spec}(R)$).   
\end{proof}

\subsection{Proof of Theorem $\ref{mainderalg}$}\label{proofmainderalg}
Write $\mathscr{R}$  for the localisation of $\overline{R}$ at $\overline{\mathfrak{p}}$,
and  $\mathscr{P}$ for  its maximal ideal.  Then $\mathscr{P}=(\varpi,\varsigma)\cdot{}\mathscr{R}$,
where $\varpi=\gamma_{0}-1$ (resp., $\varsigma=\sigma_{0}-1$) is the generator of $\mathfrak{p}R_{\mathfrak{p}}$ 
(resp., $I$)
fixed in $(\ref{eq:uniforler})$ (resp., Section $\ref{fix varsigma}$).
Moreover 
the $\Q_{p}$-module $\mathscr{P}/\mathscr{P}^{2}$ is isomorphic to 
$\big(I/I^{2}\otimes{}_{\Z_{p}}\Q_{p}\big)\oplus{}
\big(\mathfrak{p}R_{\mathfrak{p}}/\mathfrak{p}^{2}R_{\mathfrak{p}}\big)$.

Let $\mathfrak{Z}=(\mathfrak{Z}_{n})\in{}H^{1}_{\mathrm{Iw}}(\Q_{p},\T^{-})$
and $\mathfrak{z}:=\mathfrak{Z}_{0,\psi}\in{}\mathrm{Hom}_{\mathrm{cont}}(\Q_{p}^{\ast},\Q_{p})$.
According to Theorem 3.18 of \cite{G-S}
\[
                     1-\mathbf{a}_{p}^{-1}\equiv{}-\frac{\mathscr{L}_{p}(A)}{2\log_{p}(\varpi)}\cdot{}\varpi\  \ 
                     (\mathrm{mod}\ \mathfrak{p}^{2}R_{\mathfrak{p}}),
\] 
where $\log_{p}(\varpi):=\log_{p}(\gamma_{0})$.
Corollary $\ref{Ochiai vs Coleman}$ and Corollary $\ref{improved L}$ then yield the equality in $\mathscr{P}/\mathscr{P}^{2}$:
\[
          \mathcal{L}_{\T}(\mathfrak{Z})\ \ \mathrm{mod}\ \mathscr{P}^{2}
          =\mathcal{L}_{A}^{\prime}\big(\mathfrak{Z}_{\psi}\big)
          -\frac{\mathscr{L}_{p}(A)}{2\log_{p}(\varpi)}\cdot{}
          \psi\big(\mathcal{L}_{\T}^{\ast}(\mathfrak{Z}_{0})\big)\cdot{}\{\varpi\},
\]
where as usual $\{\cdot{}\} : \mathscr{P}\twoheadrightarrow{}\mathscr{P}/\mathscr{P}^{2}$ denotes the projection.
Thanks to Proposition $\ref{dercol}$ and Proposition $\ref{improved big dual}$,  the last congruence can be rewritten  as
\[
           \lri{1-p^{-1}}\cdot{}\mathcal{L}_{\T}(\mathfrak{Z})\ \ 
           \mathrm{mod}\ \mathscr{P}^{2}=
           \frac{\mathfrak{z}(p^{-1})}{\log_{p}(\varsigma)}\cdot{}\{\varsigma\}-\frac{\mathscr{L}_{p}(A)}
           {2\log_{p}(\varpi)}\cdot{}\mathfrak{z}(e(1))\cdot{}\{\varpi\}.
\]
Here we  used that $\psi(\mathbf{a}_{p})=a_{p}(A)=1$ and the equality $\dia{\exp^{\ast}(\mathfrak{z}),\mho_{\psi}(1)}_{\mathrm{dR}}
=\mathfrak{z}(e(1))$. The latter follows from the definition of 
$\exp^{\ast} : H^{1}(\Q_{p},\Q_{p})\fre{}D_\mathrm{dR}(\Q_{p})=\Q_{p}$
(see the proof of Corollary $\ref{derivative coleman}$) and our normalisation $(\ref{eq:norm mho})$
of $\mho_{\psi}(1)$. 
Applying $\Mm$ to both sides of the last equation, one obtains the formula  displayed in Part 1 of  Theorem $\ref{mainderalg}$.
(Strictly speaking, the Mellin transform is defined on $\overline{R}$, but it  
extends to a morphism $\Mm : \mathscr{R}\fre{}\mathscr{M}^{\mathrm{reg}}$, where $\mathscr{M}^{\mathrm{reg}}$
is the localisation of $\mathscr{A}$ at the multiplicative subset $\{g(k,s)\in{}\mathscr{A} : g(2,1)\not=0\}$.)

To prove Part 2 of the theorem, let $\mathfrak{Z}=(\mathfrak{Z}_{n})\in{}H^{1}_{\mathrm{Iw}}(\Q_{p,\infty},\T)$ and
let $\mathfrak{z}:=\mathfrak{Z}_{0,\psi}\in{}H^{1}(\Q_{p},V_{p}(A))$.
Since $\exp_{A}^{\ast}(\mathfrak{z})$ is equal to 
$p^{-}(\mathfrak{z})\big(e(1)\big)$,
using Corollary $\ref{derivative coleman}$
in place of Proposition $\ref{dercol}$, the same argument as  above yields 
\[
          (1-p^{-1})\cdot{}\Mm\circ{}\mathcal{L}_{\T}(\mathfrak{Z})\equiv{}
          \mathscr{L}_{p}(A)\cdot{}\exp_{A}^{\ast}(\mathfrak{z})\cdot{}(s-1)-\frac{1}{2}\mathscr{L}_{p}(A)\cdot{}
          \exp_{A}^{\ast}(\mathfrak{z})\cdot{}(k-2)\ \ \big(\mathrm{mod}\ 
          \mathscr{J}^{2}\big),
\]
thus concluding the proof of Theorem $\ref{mainderalg}$.

\section{Selmer complexes and the height-weight pairing}\label{pairing section}
Inspired by \neko's formalism of height pairings \cite[Section 11]{Ne},
we  define   the \emph{height-weight pairing} mentioned in the introduction.
We then summarise its main properties, referring to \cite{Ne} and \cite{Ven} for the proofs.

\subsection{Selmer complexes}\label{Selmer complexes} With the notations of Section $\ref{bekaochiai}$,
set $\gaun:=\gaun_{0}$.
Let $S$ be a complete, local Noetherian ring with finite residue field of 
characteristic $p$, and let $\mathscr{S}$ be a localisation of $S$.
Let $M=(M,M^{+})$ be an \emph{$\mathscr{S}$-adic, nearly-ordinary representation of $\gaun$}.
More precisely,  
$M=\mathbb{M}\otimes_{S}\mathscr{S}$ and $M^{+}=\mathbb{M}^{+}\otimes_{S}\mathscr{S}$,
where $\mathbb{M}$ is a finitely generated, free $S$-module, 
equipped with a continuous, $S$-linear action of $\gaun$,
and
$\mathbb{M}^{+}\subset{}\mathbb{M}$ is an $S$-direct summand of  $\mathbb{M}$,  
which is stable for the action of the decomposition group 
$G_{p}:=G_{\Q_{p}}\hookrightarrow{}G_{\Q}$
determined by the embedding $i_{p} : \overline{\Q}\hookrightarrow{}\overline{\Q}_{p}$.

For every prime $q|N$, fix an embedding $i_{q} : \overline{\Q}\hookrightarrow{}\overline{\Q}_{q}$,
and write $G_{q}:=G_{\Q_{q}}\hookrightarrow{}G_{\Q}$ for the corresponding decomposition group at $q$.
Following  \cite{Ne}, define  \emph{\neko's Selmer complex of $M$} as the complex of $S$-modules:
\[
             \scob(\gaun,M):=\mathrm{Cone}\lri{\ctsb(\gaun,M)\oplus{}\ctsb(\Q_{p},M^{+})
             \stackrel{\mathrm{res}_{Np}-i^{+}}{\longrightarrow{}}
             \bigoplus_{l|Np}\ctsb(\Q_{l},M)}[-1],
\] 
where the notations are as follows. 
For $G=\gaun$ or $G=G_{l}$ ($l|Np$),   $\ctsb(G,\star)$ is  the complex of continuous (non-homogeneous) cochains of $G$
with values in $\star$ and  $\ctsb(\Q_{l},\star):=\ctsb(G_{l},\star)$
(see Section 3 of \cite{Ne}).
$i^{+} : \ctsb(\Q_{p},M^{+})\fre{}\ctsb(\Q_{p},M)$
is the morphism induced by  $M^{+}\subset{}M$. Finally, for every prime $l|Np$,
$\mathrm{res}_{l} : \ctsb(\gaun,M)\fre{}\ctsb(\Q_{l},M)$
is the restriction morphism associated with the decomposition group $G_{l}\hookrightarrow{}G_{\Q}$
and  $\mathrm{res}_{Np}$ is the direct sum of the morphisms $\mathrm{res}_{l}$, for $l|Np$.

Denote by $\mathrm{D}(\mathscr{S})$  the derived category of complexes of $\mathscr{S}$-modules and by 
$\mathrm{D}(\mathscr{S})^{b}_{\mathrm{ft}}\subset{}\mathrm{D}(\mathscr{S})$ the subcategory 
of cohomologically bounded complexes  with cohomology of finite 
type over $\mathscr{S}$. Write
\[
            \derco(\Q,M)\in{}\mathrm{D}(\mathscr{S})_{\mathrm{ft}}^{b};\ \ \ \ 
            \exsel^{\ast}(\Q,M):=H^{\ast}\lri{\derco(\Q,M)}\vspace{-1mm}
\]
for the image of $\scob(\gaun,M)$ in $\mathrm{D}(\mathscr{S})_{\mathrm{ft}}^{b}$ and its cohomology respectively.

By construction, there is   an exact triangle in $\mathrm{D}(\mathscr{S})^{b}_{\mathrm{ft}}$ (cf. Section 6 of \cite{Ne}): 
\[
             \derco(\Q,M)\lfre{}\dercts(\gaun,M)\lfre{}\dercts(\Q_{p},M^{-})\oplus\bigoplus_{l|N}\dercts(\Q_{l},M),
             \vspace{-2mm}
\]
which gives rise to a long exact cohomology sequence of $\mathscr{S}$-modules
\begin{equation}\label{eq:long nekovar}
     \cdots{}\fre{}H^{q-1}(\Q_{p},M^{-})\oplus{}H^{q-1}_{N}(M)\fre{}
     \exsel^{q}(\Q,M)  \fre{}H^{q}(\gaun,M)\fre{}H^{q}(\Q_{p},M^{-})\oplus{}H^{q}_{N}(M)
     \fre{}\cdots.
\end{equation}
Here $M^{-}:=M/M^{+}=\mathbb{M}/\mathbb{M}^{+}\otimes_{S}\mathscr{S}$, 
$\dercts(G,\star)\in{}\mathrm{D}(\mathscr{S})^{b}_{\mathrm{ft}}$
is the image of $\ctsb(G,\star)$ in the derived category, and  we write 
for simplicity  $H^{q}_{N}(M):=\bigoplus_{l|N}H^{q}(\Q_{l},M)$.

\subsection{The extended Selmer group}\label{Extended Selmer Section} Let $\mathscr{S}=\Q_{p}$ and $M=V_{p}(A)$,
with the nearly-ordinary structure $i^{+} : \Q_{p}(1)\hookrightarrow{}V_{p}(A)$ given in  $(\ref{eq:tate exact})$.
By \cite[12.5.9.2]{Ne}, one can extract from $(\ref{eq:long nekovar})$ a short exact sequence of $\Q_{p}$-modules
\begin{equation}\label{eq:extended selmer}
        0\fre{}\Q_{p}\fre{}\exsel^{1}(\Q,V_{p}(A))\fre{}H^{1}_{f}(\Q,V_{p}(A))\fre{}0,
\end{equation}
where the left-most term arises as $H^{0}(\Q_{p},\Q_{p})=H^{0}(\Q_{p},V_{p}(A)^{-})$
and $H^{1}_{f}(\Q,V_{p}(A))\subset{}H^{1}(\gaun,V_{p}(A))$ is the Bloch--Kato Selmer group of $V_{p}(A)$ \cite{B-K}.
In addition the projection in $(\ref{eq:extended selmer})$ admits a natural splitting 
\[
         \sigma^{\text{u-r}} : H^{1}_{f}(\Q,V_{p}(A))\lfre{}\exsel^{1}(\Q,V_{p}(A)),
\]
characterised by the following property. Let $\wp^{+} : \exsel^{1}(\Q,V_{p}(A))\fre{}
H^{1}(\Q_{p},\Q_{p}(1))=\Q_{p}^{\ast}\widehat{\otimes}\Q_{p}$
be the morphism induced by the natural projection $\derco(\Q,V_{p}(A))\fre{}\dercts(\Q_{p},\Q_{p}(1))$.
Then 
\begin{equation}\label{eq:property section}
         \wp^{+}\circ{}\sigma^{\text{u-r}}\Big(H^{1}_{f}(\Q,V_{p}(A))\Big)
         \subset{}H^{1}_{f}(\Q_{p},\Q_{p}(1))=\Z_{p}^{\ast}\widehat{\otimes}\Q_{p}.
\end{equation}
This follows  from   Section 11.4 of \cite{Ne}, thanks to the fact that 
$\mathscr{L}_{p}(A)\not=0$ by \cite{M-man}. We use the section $\sigma^{\text{u-r}}$ to obtain the identification
$\exsel^{1}(\Q,V_{p}(A))\cong{}\Q_{p}\oplus{}H^{1}_{f}(\Q,V_{p}(A))$. 
Moreover, we identify the Tate period $q_{A}$ with the canonical generator of $\Q_{p}\subset{}\exsel^{1}(\Q,V_{p}(A))$.
In other words, from now on 
\begin{equation}\label{eq:naive extended}
           \exsel^{1}(\Q,V_{p}(A))=\Q_{p}\cdot{}q_{A}\oplus{}H^{1}_{f}(\Q,V_{p}(A)).
\end{equation}


\subsection{The height-weight pairing}\label{the pairing} As in Section $\ref{proofmainderalg}$, let $\mathscr{R}$ be the localisation 
of $\R=R\llbracket{}G_{\infty}\rrbracket$ at $\overline{\mathfrak{p}}=(\mathfrak{p},I)$
and let $\mathscr{P}=(\varpi,\varsigma)\cdot{}\mathscr{R}$
be its maximal ideal. Let $\mathscr{M}^{\mathrm{reg}}\subset{}\mathrm{Frac}(\mathscr{A})$
be the localisation of $\mathscr{A}$ at the multiplicative subset consisting of elements $g(k,s)\in{}\mathscr{A}$
such that $g(2,1)\not=0$,
and write again $\mathscr{J}\subset{}\mathscr{M}^{\mathrm{reg}}$ for the ideal of functions vanishing at $(2,1)$.
The Mellin transform extends to a morphism $\Mm : \mathscr{R}\fre{}\mathscr{M}^{\mathrm{reg}}$
mapping $\mathscr{P}$ into $\mathscr{J}$ and then induces a morphism
of $\Q_{p}$-modules  $\Mm : \mathscr{P}/\mathscr{P}^{2}\fre{}\mathscr{J}/\mathscr{J}^{2}$.

Denote by $\chi_{\infty} : \gaun\twoheadrightarrow{}G_{\infty}\subset{}\R^{\ast}$ the tautological representation
of $\gaun$
and define 
\[
                \overline{\T}:=\T\otimes_{R}\R(\chi_{\infty}^{-1})\in{}_{\R[\gaun]}\mathrm{}\mathrm{Mod};\ \ \  
                 T:=\overline{\T}\otimes_{\R}\mathscr{R}\in{}_{\mathscr{R}[\gaun]}\mathrm{Mod}.
\]
Similarly, define the $\R[G_{p}]$-modules $\overline{\T}^{\pm}:=\T^{\pm}\otimes_{R}\R(\chi_{\infty}^{-1})$
and the $\mathscr{R}[G_{p}]$-modules $T^{\pm}:=\overline{\T}^{\pm}\otimes_{R}\mathscr{R}$.
Then
$\overline{\T}^{\pm}$  are free $\R$-modules of rank one, so that 
$T=\big(T,T^{+}\big)$ is a nearly-ordinary $\mathscr{R}$-adic representation of $\gaun$.
In particular, there is  a short exact sequence of $\mathscr{R}[G_{p}]$-modules \vspace{-2mm}
\begin{equation}\label{eq:filtration T}
      0\lfre{}T^{+}\lfre{i^{+}}T\lfre{p^{-}}T^{-}\lfre{}0\vspace{-1mm}
\end{equation}
and the Selmer complex $\derco(\Q,T)\in{}\mathrm{D}(\mathscr{R})^{b}_{\mathrm{ft}}$ is defined. 

Denote by $\xi : \mathscr{R}\twoheadrightarrow{}\Q_{p}$  the composition of $\psi : R_{\mathfrak{p}}\twoheadrightarrow{}\Q_{p}$ with  the augmentation map
$\varepsilon : \mathscr{R}\twoheadrightarrow{}R_{\mathfrak{p}}$. Since $\varepsilon\circ{}\chi_{\infty}$
is the trivial character,  $(\ref{eq:ES tate})$ induces 
a natural isomorphism of $\Q_{p}[\gaun]$-modules
\begin{equation}\label{eq:isoTxi}
                        T_{\xi}:=T\otimes_{\mathscr{R},\xi}\Q_{p}\cong{}V_{p}(A).
\end{equation}
Similarly $T_{\xi}^{+}:=T^{+}\otimes_{\mathscr{R},\xi}\Q_{p}\cong{}\Q_{p}(1)$
and $T_{\xi}^{-}:=T^{-}\otimes_{\mathscr{R},\xi}\Q_{p}\cong{}\Q_{p}$ as $\Q_{p}[G_{p}]$-modules,
and $(\ref{eq:isoTxi})$ extends to an isomorphism between the $\xi$-base change of 
$(\ref{eq:filtration T})$ and the tensor product of  $(\ref{eq:tate exact})$ with $\Q_{p}$.
This induces a canonical isomorphism of complexes of $\Q_{p}$-modules 
\begin{equation}\label{eq:iso 2-variable}
              \derco(\Q,T_{\xi})\cong{}\derco(\Q,V_{p}(A)).
\end{equation}

\subsubsection{The Bockstein map} By the general behaviour of Selmer complexes under base change, $\derco(\Q,T_{\xi})$
is isomorphic to the derived base change  $\derco(\Q,T)\derot{\mathscr{R},\xi}\Q_{p}$. This yields
via $(\ref{eq:iso 2-variable})$ 
natural isomorphisms in $\mathrm{D}(\mathscr{R})^{b}_{\mathrm{ft}}$:
\begin{equation}\label{eq:control Selmer complexes}
                 \derco(\Q,T)\derot{\mathscr{R},\xi}\Q_{p}\cong{}\derco(\Q,V_{p}(A));\ \ \ 
                 \derco(\Q,T)\derot{\mathscr{R}}\mathscr{P}/\mathscr{P}^{2}\cong{}\derco(\Q,V_{p}(A))\otimes_{\Q_{p}}\mathscr{P}/\mathscr{P}^{2}.
\end{equation}
(For the details see the proof of Lemma $\ref{explicit bockstein}$ below; see also the  proof of  Proposition 8.10.1 of \cite{Ne}.)
Applying the functor $\derco(\Q,T)\derot{\mathscr{R}}-$ to the exact triangle \vspace{-1mm}
\begin{equation}\label{eq:exact delta Bockstein}
            \mathscr{P}/\mathscr{P}^{2}\lfre{}\mathscr{R}/\mathscr{P}^{2}\lfre{\xi}\Q_{p}\lfre{\partial_{\xi}}\mathscr{P}/\mathscr{P}^{2}[1]
\end{equation}
then induces a morphism in $\mathrm{D}(\mathscr{R})^{b}_{\mathrm{ft}}$:
\[
            \widetilde{\boldsymbol{\beta}}_{p}
       : \derco(\Q,V_{p}(A))\lfre{}\derco(\Q,V_{p}(A))[1]\otimes_{\Q_{p}}\mathscr{P}/\mathscr{P}^{2},
\]
called the \emph{derived Bockstein map}. It induces in cohomology the \emph{Bockstein map}
\[
           \bock:=H^{1}(\boldsymbol{\widetilde{\beta}}_{p})  : \exsel^{1}(\Q,V_{p}(A))\lfre{}\exsel^{2}(\Q,V_{p}(A))\otimes_{\Q_{p}}\mathscr{P}/\mathscr{P}^{2}.
\]

\subsubsection{Definition of the pairing} \neko's generalisation of Poitou-Tate duality  attaches to the Weil pairing on $V_{p}(A)$
a perfect, global cup-product pairing
\cite[Section 6]{Ne}
\[
        \dia{-,-}_{\mathrm{Nek}} : \exsel^{2}(\Q,V_p(A))\otimes_{\Q_p}\exsel^{1}(\Q,V_p(A))
        \longrightarrow{}H^{3}_{c,\mathrm{cont}}(\Q,\Q_p(1))\cong{}\Q_p,
\]
where $H^{\ast}_{c,\mathrm{cont}}(\Q,-)$ denotes the compactly supported cohomology and 
the last \emph{trace isomorphism} comes from global class field theory \cite[Section 5]{Ne}. 
(See in particular Sections 5.3.1.3, 5.4.1 and 6.3 of \cite{Ne}.)

We define the \emph{(cyclotomic) height-weight pairing }
\[
         \hwp{-}{-} : \exsel^{1}(\Q,V_{p}(A))\otimes_{\Q_{p}}\exsel^{1}(\Q,V_{p}(A))\lfre{}\mathscr{J}/\mathscr{J}^{2}
\]
as the composition of 
\[
        \widetilde{\beta}_{p}\otimes\mathrm{id} : 
        \exsel^{1}(\Q,V_{p}(A))\otimes_{\Q_{p}}\exsel^{1}(\Q,V_{p}(A))\lfre{}
        \exsel^{2}(\Q,V_{p}(A))\otimes_{\Q_{p}}\exsel^{1}(\Q,V_{p}(A))\otimes_{\Q_{p}}\mathscr{P}/\mathscr{P}^{2}
\]
with 
\[        
     \dia{-,-}_{\mathrm{Nek}}\otimes{}\Mm :
     \exsel^{2}(\Q,V_{p}(A))\otimes_{\Q_{p}}\exsel^{1}(\Q,V_{p}(A))\otimes_{\Q_{p}}\mathscr{P}/\mathscr{P}^{2}
     \lfre{}\mathscr{J}/\mathscr{J}^{2}.
\]
We also write $\hwp{-}{-}(k,s):=\hwp{-}{-}$ when we want to emphasise the 
dependence of $\hwp{-}{-}$ on the variables $(k,s)$. 
If $F : \mathscr{M}^{\mathrm{reg}}\fre{}\mathscr{M}^{\mathrm{reg}}$ is a morphism  of $\Q_{p}$-algebras 
s.t. $F(\mathscr{J})\subset{}\mathscr{J}$,
then $\hwp{-}{-}(F(k,s)):=F\circ{}\hwp{-}{-}$.

\begin{remark}\label{normalisation pairing}\emph{Let $W : V_{p}(A)\otimes_{\Q_{p}}V_{p}(A)\fre{}\Q_{p}(1)$ be the Weil pairing, 
normalised as in \cite[Chapter III]{Sil-1}. In order to define $\hwp{-}{-}$ without ambiguities, one has to fix
the Tate parametrisation $\Phi_{\mathrm{Tate}}$ introduced in $(\ref{eq:Tate iso})$, which is unique up to sign.
We do this by requiring: $W(a,i^{+}(b))=p^{-}(a)\cdot{}b$ for every $a\in{}V_{p}(A)$ and $b\in{}\Q_{p}(1)$. }

\end{remark}

\subsubsection{Basic properties}\label{basic properties} 
In this section we discuss the basic properties satisfied by the height-weight pairing, referring to \cite[Section 11]{Ne}
and \cite{Ven} for the proofs. 

Section 7 of \cite{Nekh} defines a symmetric \emph{(cyclotomic) canonical height pairing}
\[
                 \dia{-,-}^{\mathrm{cyc}}_{p} : H^{1}_{f}(\Q,V_{p}(A))\otimes_{\Q_{p}}H^{1}_{f}(\Q,V_{p}(A))\lfre{}\Q_{p},
\] 
denoted $h^{\mathrm{can}}$ in \cite{Nekh}. More precisely, after identifying $V_{p}(A)$
with its Kummer dual under the Weil pairing,
the definition of $h^{\mathrm{can}}$ rests on the choices of a continuous morphism 
$\lambda_{p} : \mathbb{A}^{\ast}_{\Q}/\Q^{\ast}\fre{}\Q_{p}$
(where $\mathbb{A}^{\ast}_{\Q}$ is the group  of ideles of $\Q$)
and a splitting $\texttt{sp} : D_{\mathrm{dR}}(V_{p}(A))\twoheadrightarrow{}\mathrm{Fil}^{0}D_{\mathrm{dR}}(V_{p}(A))$
of the natural filtration.
In the definition of $\dia{-,-}^{\mathrm{cyc}}_{p}$,
$\lambda_{p}$ is the  composition of the Artin map $\mathbb{A}^{\ast}_{\Q}/\Q^{\ast}\fre{}G_{\Q}^{\mathrm{ab}}$
with $\log_{p}\circ{}\chi_{\mathrm{cyc}} : G_{\Q}^{\mathrm{ab}}\fre{}\Q_{p}$,
and $\texttt{sp}$
is the splitting induced by $(\ref{eq:tangent and fil})$. 
Let $\{\cdot{}\} : \mathscr{J}\twoheadrightarrow{}\mathscr{J}/\mathscr{J}^{2}$ denote  the projection.
Given $g(k,s)=a\cdot{}\{s-1\}+b\cdot{}\{k-2\}\in{}\mathscr{J}/\mathscr{J}^{2}$,
write $\frac{d}{ds}g(2,s)_{s=1}:=a$ and $\frac{d}{dk}g(k,1)_{k=2}:=b$.

\begin{theo}\label{mainreg} The $\Q_p$-bilinear form $\hwp{-}{-}$ enjoys the following properties. \par 
$1.$ (Cyclotomic specialisation) For every $x,y\in{}H^{1}_{f}(\Q,V_{p}(A))$:
\[
                        \frac{d}{ds}\lri{\hwp{x}{y}(2,s)}_{s=1}={}\dia{x,y}_{p}^{\mathrm{cyc}}.
\]\par
$2.$ (Exceptional zero formulae) For every $z\in{}H^{1}_{f}(\Q,V_p(A))$:
\[
      \hwp{q_{A}}{q_{A}}=\log_p(q_{A})\cdot{}\{s-k/2\};\ \ \hwp{q_{A}}{z}=\log_{A}\big(\mathrm{res}_p(z)\big)\cdot{}\{s-1\},
\]
where $\log_{A}=\log_{q_{A}}\circ{}\Phi_{\mathrm{Tate}}^{-1} : H^{1}_{f}(\Q_p,V_p(A))\cong{}
A(\Q_p)\widehat{\otimes{}}\Q_p
\fre{}\Q_p$ is the formal group logarithm.\par
$3.$ (Functional equation) For every $x,y\in{}\exsel^{1}(\Q,V_p(A))$:
\[
            \hwp{y}{x}(k,s)=-\hwp{x}{y}(k,k-s).
\]
\end{theo}
\begin{proof} 
Part $1$ is proved in \cite[Corollary 11.4.7]{Ne}. Part 2 and Part 3
are proved  in  \cite{Ven}.
\end{proof}

\section{Exceptional zero  formulae \`a la Rubin}\label{Rubin GZ} 
Recall the extended height-weight 
$\widetilde{h}_{p} : H^{1}_{f}(\Q,V_{p}(A))\fre{}\mathscr{J}^{2}/\mathscr{J}^{3}$ 
introduced in  $(\ref{eq:hwp definition})$. 
For every 
$\mathfrak{Z}=(\mathfrak{Z}_{n})\in{}H^{1}_{\mathrm{Iw}}(\Q_{\infty},\T)$,
write $\mathfrak{Z}_{n,\psi}\in{}H^{1}(\gaun_{n},V_{p}(A))$ for the image of 
$\mathfrak{Z}_{n}$ under the morphism induced by $\T\twoheadrightarrow{}\T_{\psi}\subset{}V_{p}(A)$.
The aim of this section is to prove the following theorem, reminiscent of the  \emph{Rubin formulae} proved by Rubin 
\cite{Rubh} and Perrin-Riou \cite[Section 2.3]{PRconj} in a different setting
(see also \cite[Sec. 11]{Ne}).

\begin{theo}\label{p-adic GZ} Let $\mathfrak{Z}=(\mathfrak{Z}_{n})\in{}H^{1}_{\mathrm{Iw}}(\Q_{\infty},\T)$
and let $\mathfrak{z}:=\mathfrak{Z}_{0,\psi}
\in{}H^{1}(\gaun,V_{p}(A))$.

$1.$ We have the equality in $\mathscr{J}/\mathscr{J}^{2}$:
\[
              \mathcal{L}_{\T}\big(\mathrm{res}_{p}(\mathfrak{Z}),k,s\big)\ \mathrm{mod}\ \mathscr{J}^{2}=\frac{1}{\mathrm{ord}_{p}(q_{A})}\lri{1-\frac{1}{p}}^{-1}
              \exp_{A}^{\ast}\big(\mathrm{res}_{p}(\mathfrak{z})\big)\cdot{}\hwp{q_{A}}{q_{A}}.
\]
In particular: $\mathcal{L}_{\T}\big(\mathrm{res}_{p}(\mathfrak{Z}),k,s\big)\in{}\mathscr{J}^{2}$ if and only if $\mathfrak{z}\in{}H^{1}_{f}(\Q,V_{p}(A))$.

$2.$ If $\mathfrak{z}\in{}H^{1}_{f}(\Q,V_{p}(A))$, we have the equality in $\mathscr{J}^{2}/\mathscr{J}^{3}$:
\[
                  \log_{A}\big(\mathrm{res}_{p}(\mathfrak{z})\big)\cdot{}\mathcal{L}_{\T}\big(\mathrm{res}_{p}(\mathfrak{Z}),k,s\big)\ \mathrm{mod}\ \mathscr{J}^{3}=
                  \frac{-1}{\mathrm{ord}_{p}(q_{A})}\lri{1-\frac{1}{p}}^{-1}
                  \cdot{}\widetilde{h}_{p}(\mathfrak{z}).
\]
\end{theo}

This result, whose proof is given  in Section $\ref{proof of p-adic GZ}$ below, becomes particularly relevant when combined with the work
of Kato. Recall the class  $\bki_{\infty}\in{}H^{1}_{\mathrm{Iw}}(\Q_{\infty},\T)$ appearing in Theorem $\ref{maincohpad}$. By \emph{loc. cit.} and equation $(\ref{eq:definition LMK})$
\begin{equation}\label{eq:kato Main}
     \mathcal{L}_{\T}\big(\mathrm{res}_{p}\big(\bki_{\infty}\big),k,s\big)
     =L_{p}(f_{\infty},k,s).
\end{equation}      
With the notations of the introduction, we set 
\[
                    \zeta^{\mathrm{BK}}_{\infty}:=\bki_{\infty,\psi}\in{}H^{1}_{\mathrm{Iw}}(\Q_{\infty},T_{p}(A));\ \ \ 
                    \zeta^{\mathrm{BK}}=\bki_{0,\psi}\in{}H^{1}(\gaun,V_{p}(A)).
\]
By Corollary $\ref{Ochiai vs Coleman}$ and equation $(\ref{eq:hht})$, 
$\mathcal{L}_{A}\big(\mathrm{res}_{p}(\zeta_{\infty}^{\mathrm{BK}})\big)=L_{p}(A/\Q)$; 
this is  equation $(\ref{eq:kato reciprocity +})$ in the  introduction.

Equation $(\ref{eq:kato Main})$ and  Theorem $\ref{p-adic GZ}(1)$ yield the following result, which in light of  Kato's reciprocity law $(\ref{eq:kato reciprocity})$
and Theorem $\ref{mainreg}(2)$ can be seen as a variant of the main result of \cite{G-S}.

\begin{theo}\label{main GS} We have the equality in $\mathscr{J}/\mathscr{J}^{2}$:
\[
                    L_{p}(f_{\infty},k,s)\ \mathrm{mod}\ \mathscr{J}^{2}=\frac{1}{\mathrm{ord}_{p}(q_{A})}\lri{1-\frac{1}{p}}^{-1}
              \exp_{A}^{\ast}\big(\mathrm{res}_{p}\big(\zeta^{\mathrm{BK}}\big)\big)\cdot{}\hwp{q_{A}}{q_{A}}.
\] 
In particular, $L_{p}(f_{\infty},k,s)\in{}\mathscr{J}^{2}$ if and only if $\zeta^{\mathrm{BK}}$ is  a Selmer class. 
\end{theo}

Theorem $\ref{p-adic GZ}(2)$ and $(\ref{eq:kato Main})$ combine to give the following theorem (cf. Section $\ref{outline of the proof}$).

\begin{theo}\label{probably main paper} Assume that $\zeta^{\mathrm{BK}}\in{}H^{1}_{f}(\Q,V_{p}(A))$. Then we have the equality 
in $\mathscr{J}^{2}/\mathscr{J}^{3}$:
\[
         \log_{A}\big(\mathrm{res}_{p}\big(\zeta^{\mathrm{BK}}\big)\big)\cdot{}
         L_{p}(f_{\infty},k,s)\ \ \mathrm{mod}\ \mathscr{J}^{3}
         =\frac{-1}{\mathrm{ord}_{p}(q_{A})}\lri{1-\frac{1}{p}}^{-1}\cdot{}
         \widetilde{h}_{p}\big(\zeta^{\mathrm{BK}}\big).
\]
\end{theo}

\subsection{Derivatives of cohomology classes}\label{derivatives of cohomology section} With the notations of Section $\ref{the pairing}$,
Shapiro's Lemma gives a natural isomorphism of $\mathscr{R}$-modules 
\[
                H^{1}(\Q_{p},T^{-})\cong{}
                H^{1}_{\mathrm{Iw}}(\Q_{p,\infty},\T^{-})\otimes_{\R}\mathscr{R},
\]
under which the morphism $\xi_{\ast} : H^{1}(\Q_{p},T^{-})\fre{}H^{1}(\Q_{p},\Q_{p})$ induced by $T^{-}\twoheadrightarrow{}T_{\xi}^{-}\cong{}\Q_{p}$
(see $(\ref{eq:isoTxi})$) corresponds  to the $\mathscr{R}$-base change of  
$H^{1}_{\mathrm{Iw}}(\Q_{p,\infty},\T^{-})\fre{}H^{1}(\Q_{p},\Q_{p});$
$(\mathfrak{Z}_{n})\mapsto\mathfrak{Z}_{0,\psi}$.
Under this isomorphism,  $\mathcal{L}_{\T}(\cdot{},k,s)$
gives rise to a morphism of $\mathscr{R}$-modules
(denoted again by the same symbol)
\[
           \mathcal{L}_{\T}(\cdot{},k,s) : H^{1}(\Q_{p},T^{-})\lfre{}\mathscr{J}\subset{}\mathscr{M}^{\mathrm{reg}}.
\]
As usual, one writes $\mathcal{L}_{\T}(\cdot{},k,s) : H^{1}(\Q_{p},T)\fre{}\mathscr{J}$ also for the morphism induced by the projection $p^{-} : T\twoheadrightarrow{}T^{-}$.

Denote by $H^{1}(\Q_{p},T)^{o}\subset{}H^{1}(\Q_{p},T)$ the submodule consisting of classes 
$\mathfrak{Y}$ such that  $p^{-}(\mathfrak{Y})\in{}\mathscr{P}\cdot{}H^{1}(\Q_{p},T^{-})$.
Given $\mathfrak{Y}\in{}H^{1}(\Q_{p},T)^{o}$,
choose  $\mathfrak{Y}_{\varpi},\mathfrak{Y}_{\varsigma}\in{}H^{1}(\Q_{p},T^{-})$
s.t. $p^{-}(\mathfrak{Y})=\varpi\cdot{}\mathfrak{Y}_{\varpi}+\varsigma\cdot{}\mathfrak{Y}_{\varsigma}$,
write $\mathfrak{y}_{\varpi},\mathfrak{y}_{\varsigma}\in{}\mathrm{Hom}_{\mathrm{cont}}(\Q_{p}^{\ast},\Q_{p})$
for the images of $\mathfrak{Y}_{\varpi},\mathfrak{Y}_{\varsigma}$ under 
$\xi_{\ast}$
and define 
\[
         \mathrm{Der}_{\mathrm{wt}}(\mathfrak{Y}):=\log_{p}(\varpi)\cdot{}\mathfrak{y}_{\varpi}\big(e(1)\big);\ \ \ \ 
         \mathrm{Der}_{\mathrm{cyc}}(\mathfrak{Y}):=\log_{p}(\varsigma)\cdot{}\mathfrak{y}_{\varsigma}\big(p^{-1}\big);
\]
\[
     \mathrm{Der}_{\dag}(\mathfrak{Y}):=\log_{p}(\varpi)\cdot{}\mathfrak{y}_{\varpi}(p^{-1})
     -\frac{1}{2}\log_{p}(\varsigma)\cdot{}\mathscr{L}_p(A)\cdot{}\mathfrak{y}_{\varsigma}(e(1)),
\]
where $\log_{p}(\varpi):=\log_{p}(\gamma_{0})$ and $\log_{p}(\varsigma):=\log_{p}\big(\chi_{\mathrm{cyc}}(\sigma_{0})\big)$.
Note that, for $\ast\in\{\mathrm{wt},\mathrm{cyc},\dag\}$, the definition of $\mathrm{Der}_{\ast}(\mathfrak{Y})$
depends \emph{a priori} on the choice of the classes $\mathfrak{Y}_{\varpi}$ and $\mathfrak{Y}_{\varsigma}$.
That it is indeed independent of this choice is a consequence of the following 
corollary of Theorem $\ref{mainderalg}$ and the non-vanishing of 
$\mathscr{L}_{p}(A)$.

\begin{cor}\label{mainderalgdisg} For every $\mathfrak{Y}\in{}H^{1}(\Q_{p},T)^{o}$, we have
\[
             \lri{1-\frac{1}{p}}\mathcal{L}_{\T}(\mathfrak{Y},k,s)\equiv{}\mathrm{Der}_{\mathrm{cyc}}(\mathfrak{Y})\cdot{}(s-1)^{2}
             +\mathrm{Der}_{\dag}(\mathfrak{Y})\cdot{}(s-1)(k-2)-\frac{1}{2}\mathscr{L}_{p}(A)\cdot{}\mathrm{Der}_{\mathrm{wt}}(\mathfrak{Y})\cdot{}
             (k-2)^{2}\ \big(\mathrm{mod}\ \mathscr{J}^{3}\big).
\]
\end{cor}
\begin{proof} As $\log_{p}(\varpi)(k-2)$ and $\log_{p}(\varsigma)(s-1)$
are the linear terms of $\Mm(\varpi)$ and $\Mm(\varsigma)$ respectively, and $\mathcal{L}_{\T}(\cdot{},k,s)$
factorises through an $\mathscr{R}$-linear map on $H^{1}(\Q_{p},T^{-})$, this is  a direct consequence  of Theorem $\ref{mainderalg}(1)$.
\end{proof}

\subsection{Proof of Theorem $\ref{p-adic GZ}$}\label{proof of p-adic GZ}
Part 1 of the theorem follows by combining 
Theorem $\ref{mainderalg}(2)$ with Theorem $\ref{mainreg}(2)$.  We then concentrate 
on the proof of Part 2 in the rest of this section. \vspace{1mm}

\subsubsection*{Notations} With the notations of Section $\ref{Selmer complexes}$, set
$\scob(M):=\scob(\gaun,M)$.
Write 
 $\widetilde{x}=(x,x^{+},y)$ for an $n$-cochain of $\scob(M)$, where 
$x\in{}\cts^{n}(\gaun,M)$, $x^{+}\in{}\cts^{n}(\Q_{p},M^{+})$ and $y=(y_{l})_{l|Np}\in{}\bigoplus_{l|Np}\cts^{n-1}(\Q_{l},M)$.
Denote by $\widetilde{d}$ the differentials of $\scob(M)$, so that 
$\widetilde{d}(\widetilde{x})=(d(x),d(x^{+}),i^{+}(x^{+})-\mathrm{res}_{Np}(x)-d(y))$,
where the  $d$'s are   the differentials of $\ctsb(-,-)$. 
Write $\xi_{\ast} : \scob(T)\fre{}\scob(V_{p}(A))$,
$\xi_{\ast} : \ctsb(\gaun,T)\fre{}\ctsb(\gaun,V_{p}(A))$ and 
$\xi_{\ast} : \ctsb(\Q_{p},T^{?})\fre{}\ctsb(\Q_{p},V_{p}(A)^{?})$ (with $?\in{}\{\emptyset,\pm\}$)
to denote the morphisms induced on cochains 
by $T\twoheadrightarrow{}T_{\xi}\cong{}V_{p}(A)$ $(\ref{eq:isoTxi})$.
Finally, write $\derco(M):=\derco(\Q,M)$
and $\exsel^{\ast}(M):=\exsel^{\ast}(\Q,M)$.\vspace{1mm}

\subsubsection{A description of $\widetilde{\beta}_{p}$} In order to prove the theorem,
we need a more concrete description of the Bockstein map 
$\widetilde{\beta}_{p}$. This is addressed in the following  lemma.

\begin{lemma}\label{explicit bockstein} Let $\widetilde{x}\in{}\sco^{1}(V_{p}(A))$ be a $1$-cocycle, and  let 
$\widetilde{X}\in{}\sco^{1}(T)$ and  $\widetilde{Y}_{\varpi},\widetilde{Y}_{\varsigma}\in{}\sco^{2}(T)$ be cochains such that:

$(a)$ $\xi_{\ast}(\widetilde{X})=\widetilde{x}{}$;\vspace{1mm}

$(b)$    $\widetilde{d}(\widetilde{X})=\varpi\cdot{}\widetilde{Y}_{\varpi}+\varsigma\cdot{}\widetilde{Y}_{\varsigma}$.\\
Then $\widetilde{y}_{\varpi}:=\xi_{\ast}(\widetilde{Y}_{\varpi})$ and $\widetilde{y}_{\varsigma}:=\xi_{\ast}(\widetilde{Y}_{\varsigma})$
are $2$-cocycles of $\scob(V_{p}(A))$ and 
\[
           -\bock([\widetilde{x}])=[\widetilde{y}_{\varpi}]\otimes{}\{\varpi\}
           +[\widetilde{y}_{\varsigma}]\otimes{}\{\varsigma\}\in{}
           \exsel^{2}(V_{p}(A))\otimes_{\Q_{p}}\mathscr{P}/\mathscr{P}^{2}
\]
(where $[\star]$ denotes the cohomology class of   $\star$, and $\{\cdot{}\} : \mathscr{P}\twoheadrightarrow{}\mathscr{P}/\mathscr{P}^{2}$
the projection). 
\end{lemma}

\begin{proof} Consider the complex of $\mathscr{R}$-modules, concentrated in degrees $(-2,0)$:
\[
                  K_{\bullet}:=K_{\bullet}(\varpi,\varsigma) : \mathscr{R}\lfre{d_{2}}\mathscr{R}\oplus\mathscr{R}\lfre{d_{1}}\mathscr{R},
\] 
where $d_{2}(r)=(-r\varsigma,r\varpi)$ and $d_{1}(r,s)=r\varpi+s\varsigma$.
It is the Koszul complex of the $\mathscr{R}$-sequence $(\varpi,\varsigma)$ generating $\mathscr{P}$.
Note that the morphism $\xi$ in degree zero defines a quasi-isomorphism $\xi : K_{\bullet}\fre{}\Q_{p}$.
Similarly, one has a quasi-isomorphism $\xi^{\prime} : K_{\bullet}^{2}\fre{}\Q_{p}^{2}\cong{}\mathscr{P}/\mathscr{P}^{2}$,
defined in degree zero by $\xi^{\prime}(r,s)=\xi(r)\{\varpi\}+\xi(s)\{\varsigma\}$.
It is then easily verified that there is a commutative diagram in $\mathrm{D}(\mathscr{R})$:
\begin{equation}\label{eq:connecting Bockstein}
           \xymatrix{          K_{\bullet} \ar[d]_{\xi}\ar[rr]^{\widehat{\partial_{\xi}}\ \ }   &&  K_{\bullet}^{2}[1] \ar[d]^{\xi^{\prime}[1]}\\
                                         \Q_{p}  \ar[rr]^{\partial_{\xi}\ \  } && \mathscr{P}/\mathscr{P}^{2}[1],  
                                                      }
\end{equation}   
where $\partial_{\xi}$ is the morphism which appears  in the exact triangle $(\ref{eq:exact delta Bockstein})$ and $\widehat{\partial_{\xi}}$ is 
the morphism of complexes
\[
      \xymatrix{   &&  \mathscr{R} \ar[rr]^{d_{2}} \ar[d]_{\mu} && \mathscr{R}^{2} \ar[rr]^{d_{1}}\ar@{=}[d] && \mathscr{R} 
                          \\
                        \mathscr{R}^{2} \ar[rr]^{-d_{2}\oplus{}-d_{2}} && \mathscr{R}^{4} 
                        \ar[rr]^{-d_{1}\oplus{}-d_{1}} && \mathscr{R}^{2} & 
                                     }
\]
with  $\mu(r):=(0,r,-r,0)$ for every $r\in{}\mathscr{R}$.

As   $K_{\bullet }\cong{}\Q_{p}$ in $\mathrm{D}(\mathscr{R})$ and $K_{\bullet}$ is a complex of free $\mathscr{R}$-modules,
there are  functorial isomorphisms in $\mathrm{D}(\mathscr{R})$:
\begin{equation}\label{eq:derived tensor}
               C\otimes_{\mathscr{R}}K_{\bullet}\cong{}C\derot{\mathscr{R},\xi}\Q_{p};\ \ \ \ 
               C\otimes_{\mathscr{R}}K_{\bullet}^{2}\cong{}C\derot{\mathscr{R}}\mathscr{P}/\mathscr{P}^{2}
\end{equation}
for every cohomologically bounded complex  $C\in{}\mathrm{D}(\mathscr{R})^{b}$.
Since $T$ and $T^{\pm}$ are free $\mathscr{R}$-modules, the natural projection 
$K_{\bullet}\fre{}\mathscr{R}/\mathscr{P}$ (in degree zero) induces a  quasi-isomorphism 
\begin{equation}\label{eq:????}
                \scob(T)\otimes_{\mathscr{R}}K_{\bullet}\lfre{\mathrm{qis}}\scob(T)\otimes_{\mathscr{R}}\mathscr{R}/\mathscr{P}.
\end{equation}
The complex  on the right  is isomorphic to $\scob(T_{\xi})\cong{}\scob(V_{p}(A))$, as follows  from \cite[Proposition 3.4.2]{Ne}.
Then $\xi_{\ast} : \scob(T)\twoheadrightarrow{}\scob(V_{p}(A))$ and $(\ref{eq:????})$ define a quasi-isomorphism 
\[
                      \Xi : \scob(T)\otimes_{\mathscr{R}}K_{\bullet}\lfre{\mathrm{qis}}\scob(V_{p}(A)),
\]
inducing via $(\ref{eq:derived tensor})$  the first isomorphism in $(\ref{eq:control Selmer complexes})$.
Similarly, consider the quasi-isomorphism  
\[
        \Xi^{\prime}  : \scob(T)\otimes_{\mathscr{R}}K_{\bullet}^{2}\lfre{\Xi^{2}}
        \scob(V_{p}(A))\otimes_{\Q_{p}}\Q_{p}^{2}\cong{}\scob(V_{p}(A))\otimes_{\Q_{p}}\mathscr{P}/\mathscr{P}^{2}.
\]
The second isomorphism displayed in $(\ref{eq:control Selmer complexes})$ is then induced  by
$\Xi^{\prime}$ via $(\ref{eq:derived tensor})$.

Together with  $(\ref{eq:connecting Bockstein})$, the preceding discussion describes the morphism $\widetilde{\beta}_{p}$ as the composition
\begin{equation}\label{eq:explicit Bockstein}
          \widetilde{\beta}_{p}  : 
          \exsel^{1}(V_{p}(A))\lfre{\Xi_{1}^{-1}}
          H^{1}\lri{\scob(T)\otimes_{\mathscr{R}}K_{\bullet}}
          \lfre{\lri{\mathrm{id}\otimes\widehat{\partial_{\xi}}}_{1}}
          H^{2}\lri{\scob(T)\otimes_{\mathscr{R}}K_{\bullet}^{2}}
          \lfre{\Xi^{\prime}_{2}}\exsel^{2}(V_{p}(A))\otimes_{\Q_{p}}\mathscr{P}/\mathscr{P}^{2},
\end{equation}
where $(\cdot{})_{n}:=H^{n}(\cdot)$.  
Take now $\widetilde{x},\widetilde{X},\widetilde{Y}_{\varpi}$ and $\widetilde{Y}_{\varsigma}$ as in the statement. 
The relation $(b)$ gives $\varpi\cdot{}\widetilde{d}(\widetilde{Y}_{\varpi})=-\varsigma\cdot{}\widetilde{d}(\widetilde{Y}_{\varsigma})$.
This easily implies that $\widetilde{d}(\widetilde{Y}_{\varpi})=\varsigma\cdot{}\widetilde{U}$ and 
$\widetilde{d}(\widetilde{Y}_{\varsigma})=-\varpi\cdot{}\widetilde{U}$,
for a $3$-cocycle $\widetilde{U}\in{}\sco^{3}(T)$. Then $(b)$ tells us  that 
\[
          \mathbf{X}:=\Big(\widetilde{U},\big(-\widetilde{Y}_{\varpi},-\widetilde{Y}_{\varsigma}\big),\widetilde{X}\Big)\in{}
          \sco^{3}(T)\oplus{}\sco^{2}(T)^{2}\oplus{}\sco^{1}(T)=\Big(\scob(T)\otimes_{\mathscr{R}}K_{\bullet}\Big)^{1}
\]
is a $1$-cocycle,  and by $(a)$: $\Xi_{1}([\mathbf{X}])=[\xi_{\ast}(\widetilde{X})]=[\widetilde{x}]$.
Applying  $\big(\mathrm{id}\otimes{}\widehat{\partial_{\xi}}\big)_{1}$ to $\mathbf{X}$ we obtain the $2$-cocycle 
\[
          \mathbf{Y}:=\Big(\big(0,\widetilde{U},-\widetilde{U},0\big),\big(-\widetilde{Y}_{\varpi},-\widetilde{Y}_{\varsigma}\big)\Big)\in{}\sco^{3}(T)^{4}\oplus{}\sco^{2}(T)^{2}
          \subset{}\Big(\scob(T)\otimes_{\mathscr{R}}K_{\bullet}^{2}\Big)^{2}.
\]
By equation $(\ref{eq:explicit Bockstein})$  
one has $\widetilde{\beta}_{p}([\widetilde{x}])=\Xi^{\prime}_{2}\big([\mathbf{Y}]\big)
=\big[\xi_{\ast}\big(-\widetilde{Y}_{\varpi}\big)\big]\otimes\{\varpi\}+\big[\xi_{\ast}\big(-\widetilde{Y}_{\varsigma}
\big)\big]\otimes\{\varsigma\}$,
as was to be shown.
\end{proof}

\subsubsection{Proof of Part 2 of Theorem $\ref{p-adic GZ}$}
Let us begin with two simple lemmas.

\begin{lemma}\label{ISOH2} $1.$ The natural projections induce isomorphisms 
\[
            H^{1}(\Q_{p},T^{-})/\varpi\cong{}H^{1}(\Q_{p},T^{-}/\varpi);\ \ H^{1}(\Q_{p},T^{-})/\varsigma\cong{}H^{1}(\Q_{p},T^{-}/\varsigma). 
\]

$2.$ $\xi_{\ast}$ induces an isomorphism 
$H^{1}(\Q_{p},T^{-})\otimes_{\mathscr{R}}\mathscr{R}/\mathscr{P}\cong{}H^{1}(\Q_{p},\Q_{p})$.
\end{lemma}
\begin{proof} $1.$ We prove the first isomorphism, the other being similar. 
Since  $(T^{-}/\varpi)/\varsigma=T^{-}/\mathscr{P}\cong{}\Q_{p}$,   
$H^{2}(\Q_{p},T^{-}/\varpi)/\varsigma$ is a submodule of $H^{2}(\Q_{p},\Q_{p})=0$,
hence  $H^{2}(\Q_{p},T^{-}/\varpi)=0$ by Nakayama's Lemma.
We have short exact sequences 
\[
        0\lfre{}H^{q}(\Q_{p},T^{-})/\varpi\lfre{}H^{q}(\Q_{p},T^{-}/\varpi)\lfre{}H^{q+1}(\Q_{p},T^{-})[\varpi]\lfre{}0.
\]
Taking $q=2$ yields  $H^{2}(\Q_{p},T^{-})/\varpi=0$, and then $H^{2}(\Q_{p},T^{-})=0$
by another application of Nakayama's Lemma.
Taking now  $q=1$ in the exact sequence above, one finds $H^{1}(\Q_{p},T^{-})/\varpi\cong{}H^{1}(\Q_{p},T^{-}/\varpi)$.


$2.$ By an argument similar to that proving Part 1, the vanishing of $H^{2}(\Q_{p},\Q_{p})$ implies that 
$H^{1}(\Q_{p},T^{-}/\varpi)/\varsigma$ is isomorphic to $H^{1}(\Q_{p},\Q_{p})$. Together with 
Part 1 this concludes the proof.
\end{proof}

Taking $\mathscr{S}=\Q_{p}$ and $M=V_{p}(A)$ in $(\ref{eq:long nekovar})$
(so that $M^{-}=\Q_{p}$),
one can extract from the long exact sequence a morphism 
$\jmath : \mathrm{Hom}_{\mathrm{cont}}(\Q_{p}^{\ast},\Q_{p})=H^{1}(\Q_{p},\Q_{p})\fre{}\exsel^{2}(V_{p}(A))$.
We recall also  the morphism $\wp^{+} : \exsel^{1}(V_{p}(A))\fre{}H^{1}(\Q_{p},\Q_{p}(1))=\Q_{p}^{\ast}\widehat{\otimes}\Q_{p}$
introduced in Section $\ref{Extended Selmer Section}$.

\begin{lemma}\label{adjoint nekovar pairing} For every $\boldsymbol{x}\in{}\exsel^{1}(V_{p}(A))$
and every  $\kappa\in{}H^{1}(\Q_{p},\Q_{p})$:
\[
             \dia{\jmath(\kappa),\boldsymbol{x}}_{\mathrm{Nek}}=-\kappa\big(\wp^{+}(\boldsymbol{x})\big).
\]
\end{lemma}
\begin{proof} Let $\hat{\kappa}\in{}C^{1}_{\mathrm{cont}}(\Q_{p},V_{p}(A))$ be a $1$-cochain 
lifting  $\kappa$ under  $p^{-}_{\ast} : \ctsb(\Q_{p},V_{p}(A))\fre{}\ctsb(\Q_{p},\Q_{p})$ and let 
$d\hat{\kappa}=i^{+}(c(\kappa)^{+})$, for a $2$-cocycle $c(\kappa)^{+}\in{}C^{2}_{\mathrm{cont}}(\Q_{p},\Q_{p}(1))$.
By construction
\begin{equation}\label{eq:to cite in proof}
         \jmath(\kappa)=\big[\big(0,c(\kappa)^{+},\hat{\kappa}\big)\big]\in{}\exsel^{2}(V_{p}(A)).
\end{equation}     
Let $(x,x^{+},y)\in{}\sco^{1}(V_{p}(A))$
be a $1$-cocycle representing $\boldsymbol{x}$, so  $\wp^{+}(\boldsymbol{x})$
is represented by  $x^{+}\in{}\cts^{1}(\Q_{p},\Q_{p}(1))$.
The definition of  $\dia{-,-}_{\mathrm{Nek}}$ in \cite[Section 6.3]{Ne} yields
\[
         \dia{\jmath(\kappa),\boldsymbol{x}}_{\mathrm{Nek}}
         =\mathrm{inv}_{p}\big(\left[\hat{\kappa}\cup_{W}i^{+}(x^{+})\right]\big)
         =\mathrm{inv}_{p}\big(\kappa\cup{}\wp^{+}(\boldsymbol{x})\big)=-\kappa\big(\wp^{+}(\boldsymbol{x})\big).
\]
Here $\cup_{W} : \ctsb(\Q_{p},V_{p}(A))\otimes_{\Q_{p}}\ctsb(\Q_{p},V_{p}(A))
\fre{}\ctsb(\Q_{p},\Q_{p}(1))$ is the cup-product induced by the  Weil pairing $W$,
and $\mathrm{inv}_{p} : H^{2}(\Q_{p},\Q_{p}(1))\cong{}\Q_{p}$ is the  invariant map.
The second equality follows from Remark $\ref{normalisation pairing}$, while the last equality is a consequence of 
local class field theory \cite{Ser}.
\end{proof}

We are now ready to begin the actual proof of Part 2 of Theorem $\ref{p-adic GZ}$. Let 
$\mathfrak{Z}=(\mathfrak{Z}_{n})\in{}H^{1}_{\mathrm{Iw}}(\Q_{\infty},\T)$,
let $\mathfrak{z}:=\mathfrak{Z}_{0,\psi}$ and assume that 
$\mathfrak{z}\in{}H^{1}_{f}(\Q,V_{p}(A))$.
As in Section $\ref{derivatives of cohomology section}$, Shapiro's Lemma gives a natural isomorphism  
\[
                H^{1}(\gaun,T)\cong{}H^{1}_{\mathrm{Iw}}(\Q_{\infty},\T)\otimes_{\R}\mathscr{R}.
\]
Write again $\mathfrak{Z}\in{}H^{1}(\gaun,T)$
for the class corresponding to  $\mathfrak{Z}\otimes{}1\in{}H^{1}_{\mathrm{Iw}}(\Q_{\infty},\T)\otimes_{\R}\mathscr{R}$
under this isomorphism,
which satisfies $\mathfrak{z}=\xi_{\ast}(\mathfrak{Z})$.
Choose a $1$-cocycle $Z\in{}C^{1}_{\mathrm{cont}}(\gaun,T)$ representing $\mathfrak{Z}$,
and a $1$-cochain 
\[
          \widetilde{Z}=(Z,\dag,\ddag)\in{}\sco^{1}(T)\ \ \ \ 
          \text{such that}\ \ \ \ [\xi_{\ast}(\widetilde{Z})]=\mathfrak{z}\in{}\exsel^{1}(V_{p}(A)).
\]
(The  shape of $\dag\in{}C^{1}_{\mathrm{cont}}(\Q_{p},T^{+})$ and $\ddag\in{}\bigoplus_{l|Np}C^{0}_{\mathrm{cont}}(\Q_{l},T)$
will  not be relevant, and 
we use $(\ref{eq:naive extended})$ to identify $H^{1}_{f}(\Q,V_{p}(A))$ with a submodule of $\exsel^{1}(V_{p}(A))$.)
As $\xi_{\ast}\big(\widetilde{d}(\widetilde{Z})\big)=0$, 
there exist 
$\widetilde{Y}_{\varpi},\widetilde{Y}_{\varsigma}\in{}\sco^{2}(T)$ such that 
\begin{equation}\label{eq:differential}
          \widetilde{d}(\widetilde{Z})=\varpi\cdot{}\widetilde{Y}_{\varpi}+\varsigma\cdot{}\widetilde{Y}_{\varsigma}.
\end{equation}
Write $\widetilde{y}_{\varpi}:=\xi_{\ast}(\widetilde{Y}_{\varpi})$ and $\widetilde{y}_{\varsigma}:=\xi_{\ast}(\widetilde{Y}_{\varsigma})$. 
Lemma $\ref{explicit bockstein}$ yields 
\begin{equation}\label{eq:vicini1}
            -\widetilde{\beta}_{p}(\mathfrak{z})=[\widetilde{y}_{\varpi}]\otimes\{\varpi\}+
               [\widetilde{y}_{\varsigma}]\otimes\{\varsigma\}\in{}
               \exsel^{2}(V_{p}(A))\otimes_{\Q_{p}}\mathscr{P}/\mathscr{P}^{2}.
\end{equation}
For $?\in{}\{\varpi,\varsigma\}$, write  
$\widetilde{Y}_{?}=\big(Y_{?},Y_{?}^{+},\hat{K}_{?}+\hat{L}_{?}\big)$, where 
\[
          Y_{?}\in{}C^{2}_{\mathrm{cont}}(\gaun,T);\ \ Y_{?}^{+}\in{}C^{2}_{\mathrm{cont}}(\Q_{p},T^{+});\ \ 
                 \hat{K}_{?}\in{}C^{1}_{\mathrm{cont}}(\Q_{p},T);\ \ 
                             \hat{L}\in{}\bigoplus_{l|N}C^{1}_{\mathrm{cont}}(\Q_{l},T).
\]                             
Since $d(Z)=0$, 
$(\ref{eq:differential})$ gives $\varpi\cdot{}Y_{\varpi}=-\varsigma\cdot{}Y_{\varsigma}$ and this implies 
$\xi_{\ast}(Y_{?})=0$,
as $T$ and $T^{+}$ are free $\mathscr{R}$-modules. 
Define $$y_{?}^{+}:=\xi_{\ast}(Y_{?}^{+})\in{}C^{2}_{\mathrm{cont}}(\Q_{p},\Q_{p}(1));\ \  \  \ \ 
\hat{\kappa}_{?}:=\xi_{\ast}(\hat{K}_{?})\in{}C^{1}_{\mathrm{cont}}(\Q_{p},V_{p}(A)).$$
Since $\dercts(\Q_{l},V_{p}(A))\cong{}0$ for every prime $l\not=p$ one  deduces  
\[
                    [\widetilde{y}_{?}]=\big[\big(0,y_{?}^{+},\hat{\kappa}_{?}\big)\big]=\jmath\big(\kappa_{?}\big);\ \ \ \ 
                     \kappa_{?}:=p^{-}_{\ast}\big(\hat{\kappa}_{?}\big)\in{}H^{1}(\Q_{p},\Q_{p})
\]
(see equation $(\ref{eq:to cite in proof})$).
Lemma $\ref{adjoint nekovar pairing}$ and $(\ref{eq:vicini1})$ then give: for every $\boldsymbol{x}\in{}\exsel^{1}(V_{p}(A))$
\begin{equation}\label{eq:vicini3}
         \hwp{\mathfrak{z}}{\boldsymbol{x}}=\dia{-,-}_{\mathrm{Nek}}\otimes{}\Mm
         \Big(\widetilde{\beta}_{p}(\mathfrak{z})\otimes\boldsymbol{x}\Big)
         =\log_{p}(\varpi)\cdot{}\kappa_{\varpi}\big(\wp^{+}(\boldsymbol{x})\big)\cdot{}\{k-2\}
         +\log_{p}(\varsigma)\cdot{}\kappa_{\varsigma}\big(\wp^{+}(\boldsymbol{x})\big)\cdot{}\{s-1\}.
\end{equation}

\begin{lemma}\label{v4} The class $\mathrm{res}_{p}(\mathfrak{Z})$ belongs to $H^{1}(\Q_{p},T)^{o}$ and 
we have 
\[
           \log_{p}(\varsigma)\cdot{}\kappa_{\varsigma}\big(p^{-1}\big)=-\mathrm{Der}_{\mathrm{cyc}}\big(\mathrm{res}_{p}(\mathfrak{Z})\big);\ \ \ \ \ 
                   \log_{p}(\varpi)\cdot{}\kappa_{\varpi}\big(e(1)\big)=-\mathrm{Der}_{\mathrm{wt}}\big(\mathrm{res}_{p}(\mathfrak{Z})\big);
\]
\[
            \log_{p}(\varpi)\cdot{}\kappa_{\varpi}\big(p^{-1}\big)-\frac{1}{2}\log_{p}(\varsigma)\cdot{}\mathscr{L}_{p}(A)\cdot{}\kappa_{\varsigma}\big(e(1)\big)
            =-\mathrm{Der}_{\dag}\big(\mathrm{res}_{p}(\mathfrak{Z})\big).
\]
\end{lemma}
\begin{proof} Since $\mathfrak{z}$ is a Selmer class,  $p^{-}\big(\mathrm{res}_{p}(\mathfrak{Z})\big)$ is in the kernel 
of $\xi_{\ast} : H^{1}(\Q_{p},T^{-})\fre{}H^{1}(\Q_{p},\Q_{p})$ (see $(\ref{eq:extended selmer})$). 
Lemma $\ref{ISOH2}(2)$ then implies that 
$\mathrm{res}_{p}(\mathfrak{Z})\in{}H^{1}(\Q_{p},T)^{o}$.

Write $K_{?}:=p^{-}_{\ast}(\hat{K}_{?})$, so that $\xi_{\ast}(K_{?})=\kappa_{?}$. By equation $(\ref{eq:differential})$
\[
                    -p^{-}_{\ast}\big(\mathrm{res}_{p}(Z)\big)\approx\varpi\cdot{}K_{\varpi}+\varsigma\cdot{}K_{\varsigma},
\]
where $\approx$ denotes equality up to   coboundaries. In particular the sum in the R.H.S. is a $1$-cocycle in $C^{1}_{\mathrm{cont}}(\Q_{p},T^{-})$.
Then $\varpi\cdot{}\big(K_{\varpi}\ \mathrm{mod}\ \varsigma\big)\in{}C^{1}_{\mathrm{cont}}(\Q_{p},T^{-}/\varsigma)$ is a $1$-cocycle,
so $\big(K_{\varpi}\ \mathrm{mod}\ \varsigma\big)$ is  a $1$-cocycle, as $T^{-}/\varsigma$ is free over $\mathscr{R}/\varsigma$.  
Similarly, $\big(K_{\varsigma}\ \mathrm{mod}\ \varpi\big)\in{}C^{1}_{\mathrm{cont}}(\Q_{p},T^{-}/\varpi)$ is a $1$-cocycle. 
Lemma $\ref{ISOH2}$ then implies the existence of $1$-\emph{cocycles}
$A_{\varpi},A_{\varsigma}\in{}C^{1}_{\mathrm{cont}}(\Q_{p},T^{-})$ and $1$-\emph{cochains} 
$B_{\varpi},B_{\varsigma}\in{}C^{1}_{\mathrm{cont}}(\Q_{p},T^{-})$
such that 
\[
             A_{\varpi}\approx{}K_{\varpi}+\varsigma\cdot{}B_{\varpi};\ \ \ A_{\varsigma}\approx{}K_{\varsigma}+\varpi
             \cdot{}B_{\varsigma}.
\]
Note that $\varpi\varsigma\cdot{}(B_{\varpi}+B_{\varsigma})\in{}C^{1}_{\mathrm{cont}}(\Q_{p},T^{-})$ is a 
$1$-cocycle; using again 
the fact that  $T^{-}$
is  $\mathscr{R}$-free,  this implies that  $B_{\varpi}+B_{\varsigma}$ itself is a $1$-\emph{cocycle}. We then deduce the congruence 
\[
          -p^{-}\big(\mathrm{res}_{p}(\mathfrak{Z})\big)=\big[\varpi\cdot{}K_{\varpi}+\varsigma\cdot{}K_{\varsigma}\big]
          \equiv{}\varpi\cdot{}\big[A_{\varpi}\big]+\varsigma\cdot{}\big[A_{\varsigma}\big]\ \ \ 
           \big(\mathrm{mod}\ \mathscr{P}^{2}\cdot{}H^{1}(\Q_{p},T^{-})\big).
\]
Since $\kappa_{\varpi}=\xi_{\ast}([A_{\varpi}])$ and $\kappa_{\varsigma}=\xi_{\ast}([A_{\varsigma}])$,
the lemma follows from the definition of the derivatives of $\mathrm{res}_{p}(\mathfrak{Z})$.
\end{proof}

Coming back to our proof, since the $p$-adic logarithm  $\log_{p}$ and the $p$-adic valuation  $\mathrm{ord}_{p}$
give a $\Q_{p}$-basis of $\mathrm{Hom}_{\mathrm{cont}}(\Q_{p}^{\ast},\Q_{p})$, Lemma $\ref{v4}$ 
allows us to write
\begin{equation}\label{eq:vicini23}
           -\log_{p}(\varsigma)\cdot{}\kappa_{\varsigma}:=a(\varsigma)\cdot{}\log_{p}-\mathrm{Der}_{\mathrm{cyc}}\big(\mathrm{res}_{p}(\mathfrak{Z})\big)\cdot{}
           \mathrm{ord}_{p};\ \ \ 
           -\log_{p}(\varpi)\cdot{}\kappa_{\varpi}=\mathrm{Der}_{\mathrm{wt}}\big(\mathrm{res}_{p}(\mathfrak{Z})\big)
           \cdot{}\log_{p}+b(\varpi)\cdot{}\mathrm{ord}_{p},
\end{equation}
for (unique) constants $a(\varsigma),b(\varpi)\in{}\Q_{p}$ which satisfy
\begin{equation}\label{eq:vicini6}
          b(\varpi)+\frac{1}{2}\mathscr{L}_{p}(A)\cdot{}a(\varsigma)=-\mathrm{Der}_{\dag}\big(\mathrm{res}_{p}(\mathfrak{Z})\big).
\end{equation}

Since $\wp^{+}(\mathfrak{z})\in{}\Z_{p}^{\ast}\widehat{\otimes}\Q_{p}$ by $(\ref{eq:property section})$
and $\wp^{+}(q_{A})=q_{A}\widehat{\otimes}1$ (cf. $(\ref{eq:naive extended})$), equations $(\ref{eq:vicini3})$ and 
$(\ref{eq:vicini23})$ yield
\begin{equation}\label{eq:vicinif1}
            -\hwp{\mathfrak{z}}{\mathfrak{z}}=a(\varsigma)\cdot{}\log_{A}\big(\mathrm{res}_{p}(\mathfrak{z})\big)\cdot{}\{s-1\}
            +\mathrm{Der}_{\mathrm{wt}}(\mathrm{res}_{p}\big(\mathfrak{Z})\big)\cdot{}\log_{A}\big(\mathrm{res}_{p}(\mathfrak{z})\big)\cdot{}\{k-2\},
\end{equation}
and 
\begin{align}\label{eq:vicinif2}
              -\hwp{\mathfrak{z}}{q_{A}}=\Big(a(\varsigma)\cdot{}\log_{p}(q_{A})-\mathrm{Der}_{\mathrm{cyc}}\big(\mathrm{res}_{p}(\mathfrak{Z})\big)
              \cdot{}\mathrm{ord}_{p}(q_{A})\big)\Big)\cdot{}\{s-1\}   & \\
               +\Big(   \mathrm{Der}_{\mathrm{wt}}\big(\mathrm{res}_{p}(\mathfrak{Z})\big)\cdot{}\log_{p}(q_{A})+
               b(\varpi)\cdot{}\mathrm{ord}_{p}(q_{A})\big)   &\Big)  \cdot{}\{k-2\} \nonumber
\end{align}
(where we have used the formula $\log_{p}\circ{}\wp^{+}(\mathfrak{z})=\log_{A}\big(\mathrm{res}_{p}(\mathfrak{z})\big)$,
which follows immediately retracing  the definitions of $\log_{A}$ and $\wp^{+}$).
Moreover, the exceptional zero formulae displayed  in Theorem $\ref{mainreg}(2)$ give  the identities 
\begin{equation}\label{eq:vicinif3}
                 \hwp{q_{A}}{q_{A}}=\log_{p}(q_{A})\cdot{}\{s-k/2\};\ \ \ \hwp{q_{A}}{\mathfrak{z}}=\log_{A}\big(\mathrm{res}_{p}(\mathfrak{z})\big)\cdot{}\{s-1\}.
\end{equation}
Using equations $(\ref{eq:vicini6})$, $(\ref{eq:vicinif1})$, $(\ref{eq:vicinif2})$ and $(\ref{eq:vicinif3})$ 
and writing for simplicity $\mathrm{Der}_{?}(\mathfrak{Z}):=\mathrm{Der}_{?}\big(\mathrm{res}_{p}(\mathfrak{Z})\big)$, we compute
\begin{align*}
         \frac{-\widetilde{h}_{p}(\mathfrak{z})}{\mathrm{ord}_{p}(q_{A})}=
        \log_{A}\big(\mathrm{res}_{p}(\mathfrak{z})\big)
                 \times{} \Big(\mathrm{Der}_{\mathrm{cyc}}(\mathfrak{Z})\cdot{}\{s-1\}^{2}
                 +\mathrm{Der}_{\dag}(\mathfrak{Z})\cdot{}\{s-1\}\{k-2\}-\frac{1}{2}\mathscr{L}_{p}(A)
                 \cdot{}\mathrm{Der}_{\mathrm{wt}}
                 (\mathfrak{Z})\cdot{}\{k-2\}^{2}  \Big)
\end{align*}
in $\mathscr{J}^{2}/\mathscr{J}^{3}$.
Part 2 of Theorem $\ref{p-adic GZ}$ follows by combining  the last equation 
with  Corollary $\ref{mainderalgdisg}$.

\section{Proofs of the main results}

In this section we  prove the results stated in the introduction.

\subsection{Proof of Theorem A} As in Section $\ref{outline of the proof}$, let $L_{p}^{\mathrm{cc}}(f_{\infty},k):=L_{p}(f_{\infty},k,k/2)\in{}\mathscr{A}_{U}$
and let 
$$\widetilde{h}^{\mathrm{cc}}_{p} : H^{1}_{f}(\Q,V_{p}(A))\fre{}\Q_{p}$$
be the composition of $\widetilde{h}_{p}$ with the morphism $\mathscr{J}^{2}/\mathscr{J}^{3}\fre{}\Q_{p}$
sending $\alpha(k,s)\in{}\mathscr{J}^{2}$ to $\frac{d^{2}}{dk^{2}}\alpha(k,k/2)_{k=2}$.
By the functional equation for $\hwp{-}{-}(k,s)$ stated in Theorem $\ref{mainreg}(3)$, 
$\hwp{-}{-}(k,k/2)$ 
is a \emph{skew-symmetric} pairing on $\exsel^{1}(\Q,V_{p}(A))$. 
Together with  Theorem $\ref{mainreg}(2)$ this gives 
\begin{equation}\label{eq:ccpairing}
           \widetilde{h}^{\mathrm{cc}}_{p}(x)=
           \left.\frac{d^{2}}{dk^{2}}
           \det\begin{pmatrix}  0    & \frac{1}{2}\log_{A}\big(\mathrm{res}_{p}(x)\big)\cdot{}(k-2) \\ &
           \\ -\frac{1}{2}\log_{A}\big(\mathrm{res}_{p}(x)\big)\cdot{}(k-2)   & 0  \end{pmatrix}\right|_{k=2}=\frac{1}{2}\log_{A}^{2}
           \big(\mathrm{res}_{p}(x)\big),
\end{equation}
for every Selmer class $x\in{}H^{1}_{f}(\Q,V_{p}(A))$. 

Assume that $L(A/\Q,1)=0$, i.e. that $\zeta^{\mathrm{BK}}$ is a Selmer class by Kato's reciprocity $(\ref{eq:kato reciprocity})$. 
Combining the Bertolini--Darmon exceptional zero formula  of Theorem $\ref{main Bertolini-Darmon}$,
Theorem $\ref{probably main paper}$ and equation $(\ref{eq:ccpairing})$, one obtains the identity 
\begin{equation}\label{eq:main Theorem A}
           \log_{A}\lri{\mathrm{res}_{p}\big(\zeta^{\mathrm{BK}}\big)}\cdot{}2\ell\cdot{}\log_{A}^{2}(\mathbf{P})=
           \frac{-1}{\mathrm{ord}_{p}(q_{A})}\lri{1-\frac{1}{p}}^{-1}\cdot{}\log_{A}^{2}\lri{\mathrm{res}_{p}\big(\zeta^{\mathrm{BK}}\big)},
\end{equation}
for a non-zero rational number $\ell\in{}\Q^{\ast}$ and a rational point $\mathbf{P}\in{}A(\Q)\otimes\Q$.
Moreover, $\mathbf{P}\not=0$ precisely if $L(A/\Q,s)$ has a simple zero at $s=1$.
In order to conclude the proof of Theorem A, we need the following lemma. 
For every $\mathfrak{Z}\in{}H^{1}(\gaun,T)$, write 
$\mathcal{L}_{\T}^{\mathrm{cc}}\big(\mathrm{res}_{p}(\mathfrak{Z}),k\big)$ for the restriction of 
$\mathcal{L}_{\T}\big(\mathrm{res}_{p}(\mathfrak{Z}),k,s\big)$ to the central critical line $s=k/2$,
and let  $\xi_{\ast} : H^{1}(\gaun,T)\fre{}H^{1}(\gaun,V_{p}(A))$  be the morphism induced by $(\ref{eq:isoTxi})$.

\begin{lemma} Let $\mathfrak{Z}\in{}H^{1}(\gaun,T)$ be such that $\xi_{\ast}(\mathfrak{Z})\in{}H^{1}_{f}(\Q,V_{p}(A))$. The following statements are equivalent:\vspace{1mm}\\
$(a)$ $\xi_{\ast}(\mathfrak{Z})$ is in the kernel of the restriction map 
$\mathrm{res}_{p} : H^{1}_{f}(\Q,V_{p}(A))\fre{}A(\Q_{p})\widehat{\otimes}\Q_{p}$.\vspace{1mm}\\
$(b)$ $\mathcal{L}_{\T}^{\mathrm{cc}}\big(\mathrm{res}_{p}(\mathfrak{Z}),k\big)$ vanishes to order (strictly) greater than $2$ at $k=2$.
\end{lemma}
\begin{proof} Write $\mathfrak{z}:=\xi_{\ast}(\mathfrak{Z})$. Theorem $\ref{p-adic GZ}(2)$
 and equation $(\ref{eq:ccpairing})$ yield 
\[
                     \log_{A}\big(\mathrm{res}_{p}(\mathfrak{z})\big)\cdot{}\frac{d^{2}}{dk^{2}}
                     \mathcal{L}_{\T}^{\mathrm{cc}}\big(\mathrm{res}_{p}(\mathfrak{Z}),k\big)_{k=2}
                     \stackrel{\cdot{}}{=}\log_{A}^{2}\big(\mathrm{res}_{p}(\mathfrak{z})\big),
\]
where $\stackrel{\cdot{}}{=}$ denotes equality up to a non-zero rational factor. 
Since $\log_{A} : A(\Q_{p})\widehat{\otimes}\Q_{p}\cong{}\Q_{p}$ is an isomorphism, this shows that $(b)$ implies $(a)$. 

Assume now that $(a)$ holds.
 Since $0=\mathrm{res}_{p}\big(\xi_{\ast}(\mathfrak{Z})\big)=\xi_{\ast}\big(\mathrm{res}_{p}(\mathfrak{Z})\big)$,
one can write  $\mathrm{res}_{p}(\mathfrak{Z})=\varsigma\cdot{}\mathfrak{Z}_{\varsigma}
+\varpi\cdot{}\mathfrak{Z}_{\varpi}$, for classes $\mathfrak{Z}_{\varsigma},\mathfrak{Z}_{\varpi}\in{}H^{1}(\Q_{p},T)$.
(Indeed, as $H^{2}(\Q_{p},V_{p}(A))=0$, an argument similar to the one appearing in the proof of  Lemma 
$\ref{ISOH2}(2)$
proves that $H^{1}(\Q_{p},T)\otimes_{\mathscr{R},\xi}\Q_{p}\cong{}H^{1}(\Q_{p},V_{p}(A))$.)
By Theorem $\ref{mainderalg}(2)$
\[
          \mathcal{L}_{\T}\big(\mathrm{res}_{p}(\mathfrak{Z}),k,s\big)\equiv{}\mathscr{L}_{p}(A)\cdot{}\Big(\exp_{A}^{\ast}\big(\mathfrak{z}_{\varsigma}\big)\cdot{}(s-1)
          +\exp_{A}^{\ast}\big(\mathfrak{z}_{\varpi}\big)\cdot{}(k-2)\Big)\cdot{}(s-k/2)\ \ 
          \big(\mathrm{mod}\  \mathscr{J}^{3}\big),
\]
where  $\mathfrak{z}_{\varpi}:=\log_{p}(\varpi)\cdot{}\xi_{\ast}(\mathfrak{Z}_{\varpi}),\ \mathfrak{z}_{\varsigma}:=
\log_{p}(\varsigma)\cdot{}\xi_{\ast}(\mathfrak{Z}_{\varsigma})\in{}H^{1}(\Q_{p},V_{p}(A))$.
This shows that $(a)$ implies $(b)$, thus concluding the  proof of the lemma.
\end{proof}

Coming back to our proof, the preceding lemma, applied to $\mathfrak{Z}=\bki_{\infty}$, tells us that 
$\mathrm{res}_{p}\big(\zeta^{\mathrm{BK}}\big)=0$ (or equivalently $\log_{A}\big(\mathrm{res}_{p}\big(\zeta^{\mathrm{BK}}\big)\big)=0$)
if and only if $L_{p}^{\mathrm{cc}}(f_{\infty},k)$ vanishes to order greater than $2$ at $k=2$.
In addition, Theorem $\ref{main Bertolini-Darmon}$ tells us that the latter condition is equivalent to $\mathbf{P}=0$.
To sum up: $\mathrm{res}_{p}\big(\zeta^{\mathrm{BK}}\big)$ is non-zero if and only if $\mathbf{P}$ is non-zero.
Defining  $\ell_{1}:=-2\ell\cdot{}\mathrm{ord}_{p}(q_{A})\cdot{}(1-p^{-1})\in{}\Q^{\ast}$,  equation $(\ref{eq:main Theorem A})$ then gives 
\[
                    \log_{A}\big(\mathrm{res}_{p}\big(\zeta^{\mathrm{BK}}\big)\big)=\ell_{1}\cdot{}\log_{A}^{2}(\mathbf{P}),
\]
concluding the proof of Theorem A.

\subsection{Proof of Theorem B} Write $r_{\mathrm{an}}:=\mathrm{ord}_{s=1}L(A/\Q,s)$. 
That $r_{\mathrm{an}}\leq{}1$ implies $\zeta^{\mathrm{BK}}\not=0$ follows from 
Kato's  reciprocity law $(\ref{eq:kato reciprocity})$ (if $r_{\mathrm{an}}=0$) and Theorem A (if $r_{\mathrm{an}}=1$).

Conversely, assume that $\zeta^{\mathrm{BK}}$ is non-zero. 
The method of Kolyvagin, applied to the Euler system constructed by Kato in \cite{kateul},
then tells us that the strict Selmer group 
\[
             \{x\in{}H^{1}(\gaun,V_{p}(A)) : \mathrm{res}_{p}(x)=0\}\subset{}H^{1}_{f}(\Q,V_{p}(A))
\]
is trivial. For a proof of this result, see Theorem 2.3 and Chapter III, Section 5  of \cite{Rub}.
(Note that $A$ does not have  complex multiplication, since $\mathrm{ord}_{p}(j_{A})=-\mathrm{ord}_{p}(q_{A})<0$  \cite[Theorem 6.1]{Sil-2}.
This implies that the hypotheses of \cite[Theorem 2.3]{Rub} are satisfied.)
Then the restriction $\mathrm{res}_{p}\big(\zeta^{\mathrm{BK}}\big)$ is non-zero.
Using again  Theorem A (resp., equation $(\ref{eq:kato reciprocity})$), one deduces that 
$r_{\mathrm{an}}=1$ (resp., $r_{\mathrm{an}}=0$) if $\zeta^{\mathrm{BK}}$
is (resp., is not) a Selmer class.

\subsection{An interlude} In the proofs of  Theorems C-E, we need the following lemma.

\begin{lemma}\label{cruclele} Assume  that $\mathbf{(Loc)}$ holds and that  
$\mathrm{ord}_{s=1}L_{p}(A/\Q,s)=2$. Then 
$\zeta^{\mathrm{BK}}\not=0$.
\end{lemma}



\begin{proof} 
We have short exact sequences of $\Q_p$-modules (easily deduced from  Shapiro's Lemma):
\[
             0\fre{}H^q_\mathrm{Iw}(\Q_\infty,V_p(A))/\varsigma\fre{}H^q(\gaun,V_p(A))\fre{}
             H^{q+1}_\mathrm{Iw}(\Q_\infty,V_p(A))[\varsigma]\fre{}0,
\]   
where $H^q_\mathrm{Iw}(\Q_\infty,V_p(A)):=H^q_\mathrm{Iw}(\Q_\infty,T_{p}(A))\otimes_{\Z_{p}}\Q_{p}$.
Since $H^0(\gaun,V_p(A))=0$,  $H^1_\mathrm{Iw}(\Q_\infty,V_p(A))$ has no non-trivial $\varsigma$-torsion. 
Moreover  a theorem of Rohrlich \cite{Rohka} states that $L_{p}(A/\Q)\not=0$, so in particular 
$\zeta_{\infty}^{\mathrm{BK}}\not=0$ by $(\ref{eq:kato reciprocity +})$. 
There exist then  a unique class
$z_\infty^{\mathrm{BK}}=\big(z^{\mathrm{BK}}_{n}\big)\in{}H^1_\mathrm{Iw}(\Q_\infty,V_p(A))$
and a unique  integer $\rho=\rho_{\mathrm{BK}}\geq{}0$ such that
\[
              \zeta_\infty^{\mathrm{BK}}=\varsigma^{\rho}\cdot{}z_\infty^{\mathrm{BK}};\ \ \ \ 0\not=z^{\mathrm{BK}}_0\in{}H^1(\gaun,V_p(A)).
\]
By  Poitou--Tate duality and hypothesis $\mathbf{(Loc)}$ one has 
$H^1_f(\Q,V_p(A))=H^1(\gaun,V_p(A))$
(see Lemme 2.3.9 of \cite{PRconj}). 
In particular $z_0^\text{BK}\in{}H^1_f(\Q,V_p(A))$, so that 
\[
           \mathcal{L}_{A}\big(\mathrm{res}_{p}\big(z_{\infty}^{\mathrm{BK}}\big)\big)\in{}\varsigma^{2}\cdot{}\Ic
\]
by Corollary $\ref{derivative coleman}$. This yields 
\[
            L_{p}(A/\Q)=\mathcal{L}_{A}\big(\mathrm{res}_{p}\big(\zeta^{\mathrm{BK}}_{\infty}\big)\big)\in{}\varsigma^{\rho+2}\cdot{}\Ic,
\]
i.e. $\mathrm{ord}_{s=1}L_p(A/\Q,s)\geq{}\rho+2$.
Our assumption then forces
 $\rho=0$ and $\zeta^{\mathrm{BK}}=z^{\mathrm{BK}}_{0}\not=0$,
as was to be shown. 
\end{proof}

\subsection{Proof of Theorem C} Assume that  $\mathrm{sign}(A/\Q)=-1$ and that 
hypothesis $\mathbf{(Loc)}$ is satisfied. 
Given $x\in{}H^{1}_{f}(\Q,V_{p}(A))$, write for simplicity $\log_{A}(x)=\log_{A}\big(\mathrm{res}_{p}(x)\big)$.

\subsubsection{Step I} Assume that $\mathbf{P}$ is non-zero, i.e. that $\mathrm{ord}_{s=1}L(A/\Q,s)=1$.
Thanks to the work of Gross--Zagier \cite{Gross-Zagier} and Kolyvagin \cite{Koles},
$A(\Q)$ has rank one and  $A(\Q)\otimes\Q_{p}\cong{}H^{1}_{f}(\Q,V_{p}(A))$.  One  can then write 
$\zeta^{\mathrm{BK}}=\lambda\cdot{}\mathbf{P}$, where 
$\lambda=\log_{A}(\zeta^{\mathrm{BK}})/\log_{A}(\mathbf{P})$,
so that $\widetilde{h}_{p}(\zeta^{\mathrm{BK}})=\lambda^{2}\cdot{}\widetilde{h}_{p}(\mathbf{P})$.
Setting $\ell_{2}:=2\ell$, Theorem A and Theorem $\ref{probably main paper}$
combine to give the identity
\[
              L_{p}(f_{\infty},k,s)\ \ \mathrm{mod}\ \mathscr{J}^{3}=\ell_{2}\cdot{}\widetilde{h}_{p}(\mathbf{P}).
\]

\subsubsection{Step II} Assume that $\mathbf{P}=0$. We claim that
\begin{equation}\label{eq:ginost}
                  L_{p}(f_{\infty},k,s)\in{}\mathscr{J}^{3}.
\end{equation} 
Indeed  $\mathrm{ord}_{s=1}L(A/\Q,s)>1$ under our assumptions,
so that $\zeta^{\mathrm{BK}}=0$ by Theorem B. Lemma $\ref{cruclele}$ then yields
$$\left.\frac{\partial^{2}}{\partial{}s^{2}}L_{p}(f_{\infty},k,s)\right|_{(k,s)=(2,1)}=\frac{d^{2}}{ds^{2}}L_{p}(A/\Q,s)_{s=1}=0.$$
Moreover, by the functional equation $(\ref{eq:funeqfinal})$ and Theorem $\ref{main Bertolini-Darmon}$
\[
                    \left.\lri{\frac{\partial^{2}}{\partial{}k^{2}}-\frac{1}{4}\frac{\partial^{2}}{\partial{}s^{2}}}L_{p}(f_{\infty},k,s)
                    \right|_{(k,s)=(2,1)}=\frac{d^{2}}{dk^{2}}L_{p}^{\mathrm{cc}}(f_{\infty},k)_{k=2}=0.
\]
Since $L_{p}(f_{\infty},k,s)\in{}\mathscr{J}^{2}$ by Theorem $\ref{main GS}$,
the  claim $(\ref{eq:ginost})$ follows from the preceding two equations.

\subsubsection{Step III (conclusions)} We now prove Theorem C.
First of all, $L_{p}(f_{\infty},k,s)\in{}\mathscr{J}^{2}$ by $(\ref{eq:kato reciprocity})$
and Theorem $\ref{main GS}$. 
The $p$-adic Gross--Zagier formula which appears  in the statement follows from Steps I and II. Finally, the last assertion
in the statement is a direct consequence of Theorem $\ref{main Bertolini-Darmon}$ and Step II.

\subsection{Proof of Theorem D} Assume that $\mathbf{(Loc)}$ holds.

If $\mathrm{sign}(A/\Q)=+1$, then  $\mathbf{P}=0$ and  the order of vanishing of $L_{p}(A/\Q,s)$
at $s=1$ is odd by  $(\ref{eq:funeqfinal})$. Moreover $\frac{d}{ds}L_{p}(A/\Q,s)_{s=1}=0$,
as follows from $(\ref{eq:kato reciprocity})$ and  Theorem $\ref{main GS}$.
Theorem D follows in this case.

Assume now that $\mathrm{sign}(A/\Q)=-1$. 
As above, one easily proves that $\mathrm{ord}_{s=1}L_{p}(A/\Q,s)\geq{}2$. 
Moreover, writing $\widetilde{h}_{p}(\mathbf{P};k,s)=\widetilde{h}_{p}(\mathbf{P})$,
Theorem $\ref{mainreg}$ yields 
\[
            \left.\widetilde{h}_{p}(\mathbf{P};k,s)\right|_{k=2}=
            \det\begin{pmatrix} \log_{p}(q_{A})  &   \log_{A}(\mathbf{P}) \\ & \\ \log_{A}(\mathbf{P}) & 
            \dia{\mathbf{P},\mathbf{P}}^{\mathrm{cyc}}_{p}\end{pmatrix}\cdot{}\{s-1\}^{2}
            =\log_{p}(q_{A})\cdot{}\dia{\mathbf{P},\mathbf{P}}^{\mathrm{Sch}}_{p}\cdot{}\{s-1\}^{2}.
\]
Setting $\ell_{3}:=2\ell_{2}\cdot{}\mathrm{ord}_{p}(q_{A})^{-1}$ and recalling that $L_{p}(A/\Q,s)=L_{p}(f_{\infty},2,s)$ by equation $(\ref{eq:hht})$,
Theorem D follows by restricting the formula displayed in Theorem C
to the cyclotomic line $k=2$.

\subsection{Proof of Theorem E} Assume that $\mathrm{sign}(A/\Q)=-1$ and that $\mathbf{(Loc)}$ holds. 

Writing as above $\widetilde{h}_{p}(\cdot{};k,s)=\widetilde{h}_{p}(\cdot)$, 
Theorem $\ref{mainreg}$ gives
\[
            \frac{d^{2}}{dk^{2}}\widetilde{h}_{p}(\mathbf{P};k,1)_{k=2}=
            2\det\begin{pmatrix} -\frac{1}{2}\log_{p}(q_{A})  &   0 \\ & \\ -\log_{A}(\mathbf{P}) & 
            \dia{\mathbf{P},\mathbf{P}}^{\mathrm{wt}}_{p}\end{pmatrix}
            =-\log_{p}(q_{A})\cdot{}\dia{\mathbf{P},\mathbf{P}}^{\mathrm{wt}}_{p}.
\]
On the other hand, by $(\ref{eq:essendefimproved})$ and  Theorem 3.18 of \cite{G-S}
\[
         \frac{1}{2}\frac{d^{2}}{dk^{2}}L_{p}(f_{\infty},k,1)_{k=2}=
         \frac{d}{dk}\lri{1-a_{p}(k)^{-1}}_{k=2}\cdot{}\frac{d}{dk}L_{p}^{\ast}(f_{\infty},k)_{k=2}=-\frac{1}{2}\mathscr{L}_{p}(A)\cdot{}
         \frac{d}{dk}L^{\ast}_{p}(f_{\infty},k)_{k=2}.
\]
Since $\mathscr{L}_{p}(A)\not=0$ \cite{M-man},   Theorem C and the preceding two equations yield
the identity 
\[
             \frac{d}{dk}L_{p}^{\ast}(f_{\infty},k)_{k=2}=2\ell_{4}\cdot{}\dia{\mathbf{P},\mathbf{P}}^{\mathrm{wt}}_{p};\ \ \ 
             \ell_{4}:=\ell_{2}\cdot{}\mathrm{ord}_{p}(q_{A}).
\]
To conclude the proof, it remains to show that
$2\dia{\mathbf{P},\mathbf{P}}^{\mathrm{wt}}_{p}=-\dia{\mathbf{P},\mathbf{P}}^{\mathrm{cyc}}_{p}$. This
 follows from  Theorem $\ref{mainreg}(3)$.

\bibliography{myref1}{}

\begin{thebibliography}{BSDGP96}

\bibitem[BD98]{B-Dcd}
M.~Bertolini and H.~Darmon.
\newblock \emph{Heegner points, $p$-adic L-functions and the Cerednik-Drinfeld
  uniformization}.
\newblock {\em Invent. Math.}, 131, 1998.

\bibitem[BD07]{B-D}
M.~Bertolini and H.~Darmon.
\newblock \emph{Hida families and rational points on elliptic curves}.
\newblock {\em Invent. Math.}, 168(2), 2007.

\bibitem[BD14]{B-Dkato}
M.~Bertolini and H.~Darmon.
\newblock \emph{Kato's Euler system and rational points on elliptic curves I: A
  $p$-adic Beilinson formula}.
\newblock {\em Israel J. Math.}, 199(1), 2014.

\bibitem[BDP13]{Be-Da-Pr}
M.~Bertolini, H.~Darmon, and K.~Prasanna.
\newblock \emph{Generalized Heegner cycles and p-adic Rankin L-series. With an
  appendix by Brian Conrad }.
\newblock {\em Duke Math. J.}, 162(6), 2013.

\bibitem[Ber77]{Ber-1}
D.~Bertrand.
\newblock \emph{Transcendence et lois de groupes alg\'ebriques}.
\newblock {\em S\'eminaire Delange-Pisot-Poitou. Th\'eorie de nombres}, 18(1),
  1976-1977.

\bibitem[BK90]{B-K}
S.~Bloch and K.~Kato.
\newblock \emph{$L$-functions and Tamagawa numbers of motives }.
\newblock In P.~Cartier, L.~Illusie, N.~Katz, and G.~Laumon, editors, {\em The
  Grothendieck Festschrift}. Modern Birkh\"auser Classics, 1990.

\bibitem[BSDGP96]{M-man}
K.~Barr\'e-Sirieix, G.~Diaz, F.~Gramain, and G.~Philibert.
\newblock \emph{Une Preuve de la Conjecture de Mahler-Manin}.
\newblock {\em Invent. Math.}, 124, 1996.

\bibitem[Col79]{Col}
R.~Coleman.
\newblock \emph{Division values in local fields}.
\newblock {\em Invent. Math.}, 53, 1979.

\bibitem[EPW06]{EPW}
M.~Emerton, R.~Pollack, and T.~Weston.
\newblock \emph{Variation of Iwasawa invariants in Hida families}.
\newblock {\em Invent. Math.}, 163, 2006.

\bibitem[GS93]{G-S}
R.~Greenberg and G.~Stevens.
\newblock \emph{$p$-adic L-functions and $p$-adic periods of modular forms}.
\newblock {\em Invent. Math.}, 4(111), 1993.

\bibitem[GV00]{GV}
R.~Greenberg and V.~Vatsal.
\newblock \emph{On the Iwasawa invariants of elliptic curves}.
\newblock {\em Invent. Math.}, 142, 2000.

\bibitem[GZ86]{Gross-Zagier}
B.~Gross and D.~Zagier.
\newblock \emph{Heegner points and derivatives of $L$-series}.
\newblock {\em Invent. Math.}, 86(2), 1986.

\bibitem[Hid86a]{H-2}
H.~Hida.
\newblock \emph{Galois representations into
  ${\it{GL}}_{{2}}(\mathbf{Z}_{p}\llbracket\it{X}\rrbracket)$ attached to
  ordinary cusp forms}.
\newblock {\em Invent. Math.}, 85, 1986.

\bibitem[Hid86b]{H-1}
H.~Hida.
\newblock \emph{Iwasawa modules attached to congruences of cusp forms}.
\newblock {\em Ann. scient. \'Ec. Norm. Sup.}, 19, 1986.

\bibitem[Kat93]{Katiw}
K.~Kato.
\newblock \emph{Lectures in the approach to Iwasawa theory for Hasse-Weil
  L-functions via $B_{\mathrm{dR}}$}.
\newblock In {\em Arithmetic Algebraic Geometry (Trento 1991)}. Lecture Notes
  in Math. 1553, New York: Springer-Verlag, 1993.

\bibitem[Kat04]{kateul}
K.~Kato.
\newblock \emph{$p$-adic Hodge theory and values of zeta functions of modular
  forms}.
\newblock {\em Ast\'erisque}, 295, 2004.

\bibitem[Kit94]{Kit}
K.~Kitagawa.
\newblock \emph{On Standard $p$-adic $L$-functions of Families of Elliptic Cusp
  Forms }.
\newblock In B~Mazur and G.~Stevens, editors, {\em $p$-adic monodromy and the
  Birch and Swinnerton-Dyer conjecture}. American Mathematical Society, 1994.

\bibitem[Kob13]{KobGZ}
S.~Kobayashi.
\newblock \emph{The $p$-adic Gross-Zagier formula for elliptic curves at
  supersingular primes}.
\newblock {\em Invent. Math.}, 191(3), 2013.

\bibitem[Kol90]{Koles}
V.~A. Kolyvagin.
\newblock \emph{Euler systems}.
\newblock In {\em The Grothendieck Festschrift, Vol. II}. Progr. Math., 87,
  Birkh\"{a}user Boston, Boston, MA, 1990.

\bibitem[LVZ15]{L-V-Z}
D.~Loeffler, O.~Venjacob, and S.~L. Zerbes.
\newblock \emph{Local epsilon isomorphisms}.
\newblock {\em Kyoto J. Math.}, 55(1), 2015.

\bibitem[Mok11]{Mok}
C.P. Mok.
\newblock \emph{Heegner points and p-adic L-functions for elliptic curves over
  certain totally real fields}.
\newblock {\em Commentarii Mathematici Helvetici}, 86, 2011.

\bibitem[MT90]{M-T}
B.~Mazur and J.~Tilouine.
\newblock \emph{Repr\'esentations galoisiennes, diff\'erentielles de K\"ahler
  et} \emph{``conjectures principales''}.
\newblock {\em Publ. Math. de l'I.H.E.S.}, 71, 1990.

\bibitem[MTT86]{M-T-T}
B.~Mazur, J.~Tate, and J.~Teitelbaum.
\newblock \emph{ On p-adic analogues of the conjectures of Birch and
  Swinnerton-Dyer}.
\newblock {\em Invent. Math.}, 126, 1986.

\bibitem[Nek93]{Nekh}
J.~Nekovar.
\newblock \emph{On p-adic height pairings}.
\newblock In {\em Seminaire de Theorie des Nombres, Paris, 1990-91}. Progr. in
  Math., 108, Birkh\"auser Boston, 1993.

\bibitem[Nek06]{Ne}
J.~Nekovar.
\newblock \emph{Selmer complexes}.
\newblock {\em Ast\'erisque}, 310, 2006.

\bibitem[NP00]{N-P}
J.~Nekovar and A.~Plater.
\newblock \emph{On the parity of ranks of Selmer groups}.
\newblock {\em Asian J. Math.}, 4(2), 2000.

\bibitem[Och03]{Ochiaicol}
T.~Ochiai.
\newblock \emph{A generalization of the Coleman map for Hida deformations}.
\newblock {\em Amer. J. Math.}, 125(4), 2003.

\bibitem[Och06]{Ochiai}
T.~Ochiai.
\newblock \emph{On the two-variable Iwasawa main conjecture}.
\newblock {\em Compositio Math.}, 142, 2006.

\bibitem[Oht95]{Ohta}
M.~Ohta.
\newblock \emph{ On the $p$-adic Eichler-Shimura isomorphism for $\Lambda$-adic
  cusp forms}.
\newblock {\em J. Reine Angew. Math.}, 463, 1995.

\bibitem[PR87]{PRpbsd}
B.~Perrin-Riou.
\newblock \emph{Points de Heegner et d\'eriv\'ees de fonctions L p-adiques.}
\newblock {\em Invent. Math.}, 89(3), 1987.

\bibitem[PR92]{PR}
B.~Perrin-Riou.
\newblock \emph{Th\`eorie d'Iwasawa et hauteurs $p$-adiques}.
\newblock {\em Invent. Math}, 109, 1992.

\bibitem[PR93]{PRconj}
B.~Perrin-Riou.
\newblock \emph{Fonctions $L$ $p$-adiques d'une courbe elliptique et points
  rationnels}.
\newblock {\em Ann. Inst. Fourier, Grenoble}, 43(4), 1993.

\bibitem[PR94]{PRlog}
B.~Perrin-Riou.
\newblock \emph{Th\'eorie d'Iwasawa des repr\'esentations p-adiques sur un
  corps local. With an appendix by Jean-Marc Fontaine}.
\newblock {\em Invent. Math}, 115(1), 1994.

\bibitem[Roh84]{Rohka}
D.~E. Rohrlich.
\newblock \emph{On L-functions of elliptic curves and cyclotomic towers}.
\newblock {\em Invent. Math.}, 75(3), 1984.

\bibitem[Rub94]{Rubh}
K.~Rubin.
\newblock \emph{Abelian varieties, $p$-adic heights and derivatives}.
\newblock In G.~Frey and J.~Ritter, editors, {\em Algebra and Number Theory}.
  de Gruyter, 1994.

\bibitem[Rub98]{Rubs}
K.~Rubin.
\newblock \emph{Euler systems and modular elliptic curves}.
\newblock In A.~Scholl and R.~Taylor, editors, {\em Galois representations in
  arithmetic algebraic geometry}. London Math. Soc. Lect. Notes 254, Cambridge:
  Cambridge University Press, 1998.

\bibitem[Rub00]{Rub}
K.~Rubin.
\newblock {\em Euler Systems}.
\newblock Annals of Mathematics Studies, 147. Hermann Weyl Lectures. The
  Institute for Advanced Study. Princeton University Press, Princeton, 2000.

\bibitem[Sch82]{Sch-1}
P.~Schneider.
\newblock \emph{$p$-Adic Height Pairings I}.
\newblock {\em Invent. Math.}, 69, 1982.

\bibitem[Ser67]{Ser}
J.P. Serre.
\newblock \emph{Local class field theory}.
\newblock In J.~W.~S. Cassels and A.~Frohlich, editors, {\em Algebraic Number
  Theory}. Academic Press, 1967.

\bibitem[Sil86]{Sil-1}
J.~Silverman.
\newblock {\em The arithmetic of elliptic curves}.
\newblock Springer-Verlag New York, Inc., 1986.

\bibitem[Sil94]{Sil-2}
J.~Silverman.
\newblock {\em Advanced Topics in the arithmetic of elliptic curves}.
\newblock Springer-Verlag New York, Inc., 1994.

\bibitem[Ven13]{PhD}
R.~Venerucci.
\newblock \emph{$p$-adic regulators and $p$-adic analytic families of modular
  forms}.
\newblock {\em Ph.D. Thesis, University of Milan}, 2013.

\bibitem[Ven14]{Ven}
R.~Venerucci.
\newblock \emph{$p$-adic regulators and Hida $p$-adic $L$-functions}.
\newblock {\em Preprint}, 2014.

\bibitem[Wil78]{Wiles-rec}
A.~Wiles.
\newblock \emph{Higher explicit reciprocity laws}.
\newblock {\em Ann. Math. (2)}, 107(2), 1978.

\end{thebibliography}
\bibliographystyle{alpha}

\end{document}